\documentclass[12pt,reqno]{amsart}
\usepackage{amsmath,a4wide}
\usepackage{xfrac}
\usepackage{stmaryrd,mathrsfs,bm,amsthm,mathtools,yfonts,amssymb,color}
\usepackage{esint}
\usepackage{xcolor}
\usepackage[pagebackref,citecolor=black,linkcolor=blue,colorlinks=true]{hyperref}
\usepackage{tikz,tikz-3dplot}

\begingroup
\newtheorem{theorem}{Theorem}[section]
\newtheorem{lemma}[theorem]{Lemma}
\newtheorem{proposition}[theorem]{Proposition}
\newtheorem{corollary}[theorem]{Corollary}
\endgroup

\theoremstyle{definition}
\newtheorem{definition}[theorem]{Definition}
\newtheorem{remark}[theorem]{Remark}
\newtheorem{ipotesi}[theorem]{Assumption}

\numberwithin{equation}{section}

\newcommand{\N}{\mathbb{N}} %naturali
\newcommand{\R}{\mathbb{R}} %reali
\newcommand{\bA}{\mathbf{A}} %norma 2nd ff
\newcommand{\bC}{\mathbf{C}} %cilindri
\newcommand{\bB}{\mathbf{B}} %palle
\newcommand{\bS}{\mathbf{S}} %libri
\newcommand{\bbE}{{\mathbb E}}

\newcommand{\p}{\mathbf{p}} %orthogonal projections
\newcommand{\cB}{\mathcal{B}} %palla in A_Q
\newcommand{\Dir}{\mathrm{Dir}} %multivalued Dirichlet functional
\newcommand{\bG}{\mathbf{G}} %grafico

\newcommand{\mass}{\mathbf{M}} %mass of a current
\newcommand{\Rc}{\mathscr{R}} %integer rectifiable currents
\newcommand\res{\mathop{\hbox{\vrule height 7pt width .3pt depth 0pt\vrule height .3pt width 5pt depth 0pt}}\nolimits}

\newcommand{\V}{\mathbf{v}} %varifold associato

\newcommand{\bE}{\mathbf{E}} %eccesso riscalato
\newcommand{\bh}{\mathbf{h}} %height

\newcommand{\modp}{{\rm mod}(p)} %modulo p
\newcommand{\Ha}{\mathcal{H}} %Hausdorff distance

\newcommand{\eps}{\varepsilon} %epsilon
\newcommand{\spt}{\mathrm{spt}} %support (usually of a measure)
\newcommand{\dist}{\mathrm{dist}} %distance
\newcommand{\tr}{\mathrm{tr}} %trace
\newcommand{\diam}{\mathrm{diam}} %diameter
\newcommand{\Lip}{\mathrm{Lip}} %Lipschitz
\renewcommand{\epsilon}{\varepsilon}

\def\XXint#1#2#3{{\setbox0=\hbox{$#1{#2#3}{\int}$ }
		\vcenter{\hbox{$#2#3$ }}\kern-.6\wd0}}

\newcommand{\mres}{\mathop{\hbox{\vrule height 7pt width .3pt depth 0pt\vrule height .3pt width 5pt depth 0pt}}\nolimits}

\DeclareMathOperator{\Div}{div} %divergence (conflict?)
\def\a#1{\left\llbracket{#1}\right\rrbracket}
\newcommand{\abs}[1]{\lvert#1\rvert} %absolute value
\newcommand{\Abs}[1]{\left\lvert#1\right\rvert} %large absolute value
\newcommand{\norm}[1]{\left\lVert#1\right\rVert} %norm
\newcommand{\etab}{\boldsymbol{\eta}}

\newcommand{\Iqspec}{{\mathscr{A}_Q (\R)}}

\newcommand{\bH}{\mathbf{H}}

\author[C. De Lellis]{Camillo De Lellis}
\address{School of Mathematics, Institute for Advanced Study, 1 Einstein Dr., Princeton NJ 05840, USA}
\email{camillo.delellis@math.ias.edu}
\author[J. Hirsch]{Jonas Hirsch}
\address{Mathematisches Institut, Universit\"at Leipzig, Augustusplatz 10, D-04109 Leipzig, Germany}
\email{Jonas.Hirsch@math.uni-leipzig.de}
\author[A. Marchese]{Andrea Marchese}
\address{Dipartimento di Matematica, Universit\`a degli Studi di Trento, Via Sommarive 14, I-38123 Povo (TN), Italy}
\email{andrea.marchese@unitn.it}
\author[L. Spolaor]{Luca Spolaor}
\address{Department of Mathematics, UC San Diego, AP\&M, La Jolla, California, 92093, USA}
\email{lspolaor@ucsd.edu}
\author[S. Stuvard]{Salvatore Stuvard}
\address{Dipartimento di Matematica, Universit\`a degli Studi di Milano, Via Saldini 50, I-20133 Milano (MI), Italy}
\email{salvatore.stuvard@unimi.it}

\begin{document}

\title[Excess decay for minimizing hypercurrents mod $2Q$]{Excess decay for minimizing hypercurrents mod $2Q$}

\begin{abstract}
We consider codimension $1$ area-minimizing $m$-dimensional currents $T$ mod an even integer $p=2Q$ in a $C^2$ Riemannian submanifold $\Sigma$ of the Euclidean space. We prove a suitable excess-decay estimate towards the unique tangent cone at every point $q\in \spt (T)\setminus \spt^p (\partial T)$ where at least one such tangent cone is $Q$ copies of a single plane. While an analogous decay statement was proved in \cite{MW} as a corollary of a more general theory for stable varifolds, in our statement we strive for the optimal dependence of the estimates upon the second fundamental form of $\Sigma$. This technical improvement is in fact needed in \cite{DHMSS_final} to prove that the singular set of $T$ can be decomposed into a $C^{1,\alpha}$ $(m-1)$-dimensional submanifold and an additional closed remaining set of Hausdorff dimension at most $m-2$.
\end{abstract}

\maketitle

\tableofcontents

\section{Introduction}

In this paper we consider area minimizing currents mod an integer $p \geq 2$ which have codimension $1$ in a given smooth Riemannian ambient manifold. 

\begin{definition} \label{def:am_modp}
	Let $p\geq 2$, $\Omega \subset \R^{m+n}$ be open, and let $\Sigma \subset \R^{m+n}$ be a complete submanifold without boundary of dimension $m+\bar{n}$ and class $C^{2}$. We say that an $m$-dimensional integer rectifiable current $T \in \Rc_{m}(\Sigma)$ is \emph{area minimizing} $\modp$ in $\Sigma \cap \Omega$ if
	\begin{equation}\label{e:am_mod_p}
	\mass (T) \leq \mass (T + W) \qquad \mbox{for any $W \in \Rc_{m}(\Omega\cap \Sigma)$ which is a boundary $\modp$}\,.
	\end{equation}
\end{definition}

In \cite{MW} the authors leverage the regularity theory of \cite{Wic} for stable integral varifolds in codimension $\bar n = 1$ and use an observation in \cite{DHMSS} to prove, among other things, the uniqueness of tangent cones at every interior point $q$ where at least one tangent cone is flat, namely contained in an $m$-dimensional plane. Recall that at any such $q$ the density $\Theta_T (q)$ is necessarily an integer no larger than $\frac{p}{2}$. Moreover, if $1 \leq \Theta_T (q) \leq \lfloor \frac{p-1}{2}\rfloor$, the regularity results of Allard \cite{Allard72} and White \cite{White86} apply: in this case $T$ is a regular submanifold in a neighborhood of $q$, counted with multiplicity $\Theta_T (q)$. The case of interest here is therefore that of even moduli $p=2Q$ and interior points $q$ with at least one flat tangent cone and density $Q=\frac{p}{2}$. Under the latter assumption, in fact, $q$ can be a singular point (cf. \cite{White79} and \cite[Example 1.6]{DLHMS}). The result of \cite{MW} on which we focus here can therefore be stated as follows.

\begin{theorem}\label{t:main}
Let $p= 2Q$ be even, $\Sigma, T$, and $\Omega$ be as in Definition \ref{def:am_modp} with $\dim (\Sigma)=\dim (T)+1 = m+1$. If at $q\in \spt (T)\setminus \spt^p (\partial T)$ one tangent cone to $T$ is of the form $\bC = Q \a{\pi}$ for some $m$-dimensional plane $\pi$, then $\bC$ is the {\em unique} tangent cone to $T$ at $q$.
\end{theorem}

In \cite{DHMSS_final} the above theorem was one of the starting points to complete the study of the fine structure of the singular set of area-minimizing hypercurrents mod $p$ in the case when $p$ is even. More precisely, we prove there that, outside an exceptional closed subset of (Hausdorff) dimension at most $m-2$, the rest of the interior singular set of $T$ is, locally, an $(m-1)$-dimensional submanifold of class $C^{1,\alpha}$. This generalizes the classical theorem in \cite{White79} to the case of even $p=2Q > 4$ (the case $p=2$ is special, because locally area-minimizing hypercurrents mod $2$ are area-minimizing integral currents, see the discussion in the introduction to \cite{DLHMS}; note also that \cite{DHMSS_final} needs a slightly stronger regularity for the ambient manifold $\Sigma$ than the one stated in Definition \ref{def:am_modp}, and more precisely $\Sigma$ is assumed to be of class $C^{3,\alpha}$ for some $\alpha>0$). The main theorem in \cite{DHMSS_final} complements the analogous theorem for odd moduli, first shown by Taylor in \cite{Taylor} for $p=3$ and $m=2$, and extended recently to any odd $p$ in \cite{DHMSS}. As was later pointed out in \cite{MW}, the one proposition in \cite{DHMSS} which is used in combination with \cite{Wic} to yield Theorem \ref{t:main} above can in fact be also used to derive the same regularity results of \cite{DHMSS} via the theory of stable varifolds of \cite{Wic}. 

While uniqueness of flat tangent cones is the starting point of the analysis we carried out in \cite{DHMSS_final}, in fact we do need there an important refinement of Theorem \ref{t:main}.
%, namely the validity of a specific power-law decay of the rescaled currents (around the base point $q$) towards the unique tangent plane $Q\a{\pi}$. 
In order to give the precise statement, we introduce the $L^2$ excess of $T$ from an $m$-dimensional plane $\bar\pi$ in a ball $\bB_r (q)$, namely
\begin{equation}\label{e:plane-excess}
\bE (T, \bar\pi, q, r):= \frac{1}{r^{m+2}} \int_{\bB_r (q)} \dist^2 (q'-q, \bar\pi)\, d\|T\| (q')\, ,
\end{equation}
the minimal planar $L^2$ excess
\begin{equation} \label{e:minimal-plane-excess}
    \bar\bE(T,q,r) := \min_{\bar \pi \subset T_q \Sigma} \bE (T, \bar \pi, q, r)\,,
\end{equation}
and the notation $\bA$ for the supremum norm of the second fundamental form of $\Sigma$, i.e. $\bA = \|A_\Sigma\|_\infty$.
The precise decay statement which is needed in \cite{DHMSS_final} is then the following.

\begin{theorem}\label{t:main-with-decay}
There are positive constants $\varepsilon$, $\alpha$, and $C$ with the following properties.
Let $p= 2Q$ be even and $\Sigma, T$, $\Omega$, $q$, and $\pi$ be as in Theorem \ref{t:main}. Assume in addition that $T$ is a representative $\modp$, that $\bB_r (q)\subset \Omega\setminus \spt^p (\partial T)$, $\|T\| (\bB_r (q)) < (Q+\frac{1}{2})\omega_m r^m$ and that
\begin{equation}
\bar \bE (T,q, r) + r^2 \bA^2  \leq \varepsilon\, .
\end{equation}
Then, for every $\rho < \frac{r}{2}$ we have
\begin{equation}\label{e:decay-0}
\bE (T, \pi, q, \rho) \leq C \left(\frac{\rho}{r}\right)^\alpha (\bar \bE (T, q, r) + r^2 \bA^2)\, .
\end{equation}
\end{theorem}

Note that the quadratic dependence on $\bA$ in the right-hand side is essential for the arguments in \cite{DHMSS_final}: the power $2$ in $\bA$ and a subtle analysis of the anisotropic rescalings of $T$ around $q$ allow us to improve $\alpha$ in \eqref{e:decay-0} to any exponent strictly smaller than $2$; this almost quadratic decay is then a crucial ingredient in the rest of the work. 

Estimate \eqref{e:decay-0} is certainly an outcome of \cite{MW} when $\Sigma$ is flat, i.e. if $\bA=0$. On the other hand, the ``obvious'' modification of the arguments in \cite{MW} seem to yield an $\bA$-dependence of the right-hand side of \eqref{e:decay-0} which is linear, rather than quadratic, since $\bA$ bounds the $L^\infty$ norm of the generalized mean curvature $\vec{H}_T$ of the varifold induced by $T$. The aim of this note is to show that the technical improvement from $\bA$ to $\bA^2$ is however possible, and hence the regularity theory of \cite{DHMSS_final} holds in the full generality claimed there.  

Roughly speaking, we need to control error terms in inequalities and identities derived through first variations along some test vector fields $X$. All the vector fields $X$ relevant to the proof of Theorem \ref{t:main-with-decay} are almost tangential to the ambient manifold $\Sigma$ and the deviation from tangentiality can be controlled with $\bA$. Since the mean curvature vector $\vec{H}_T$ is directed normally to $\Sigma$, the $L^\infty$ norm of the scalar product $\vec{H}_T \cdot X$ can then be estimated by $\bA^2$. This idea is used already in \cite[Appendix A]{DLS_Lp} to improve the $\bA$-dependence in the classical monotonicity formula. Incidentally, this quadratic improvement plays also a pivotal role in the work \cite{DLMiSk}.

While the underlying idea towards the improvement is simple, there are many cumbersome technicalities to overcome; for this reason, while we believe that an analogous improvement on the $\bA$-dependence can be achieved in the context of varifolds which are stationary and stable in $\Sigma$ and enjoy the special structure detailed in \cite{MW} (of which the varifolds induced by area-minimizing hypercurrents mod $p$ are a special case), we choose to take advantage of the minimizing property mod $p$ to highlight more efficiently the aforementioned technical obstructions.

\medskip

\noindent\textbf{Acknowledgements.} C.D.L. acknowledges support from the National Science Foundation through the grant FRG-1854147. J.H. was partially supported by the German Science Foundation DFG in context of the Priority Program SPP 2026 “Geometry at Infinity”. A.M. and S.S. were partially supported by the \textit{Gruppo Nazionale per l'Analisi Matematica, la Probabilit\`a e le loro Applicazioni} of INdAM. L.S. acknowledges the support of the NSF Career Grant DMS-2044954. 

\section{Notation and preliminaries}

In this section we collect the main notation in use in the paper as well as one important estimate that will be used multiple times in the sequel.

\subsection{Notation} The symbol $\mathbf{p}$ will be typically used for orthogonal projections: in particular, given a linear subspace $\pi\subset \mathbb R^{m+n}$, $\mathbf{p}_\pi$ is the orthogonal projection onto $\pi$, while $\mathbf{p}^\perp_\pi$ is the orthogonal projection on the orthogonal complement. The symbol $T_{q,r}$ will denote the recentered and rescaled current, with base point $q$ and scale $r$: more precisely, if $\lambda_{q,r}$ is the map $q'\mapsto \lambda_{q,r} (q') := r^{-1} (q'-q)$, then $T_{q,r}:= (\lambda_{q,r})_\sharp T$. 
%The set of tangent cones to $T$ and $q$ are the limits of convergent sequences $T_{q, r_k}$ with $r_k\downarrow 0$.
We next introduce two families of sets which are central to the rest of our work.

\begin{definition}
Let $T$, $\Sigma$, $p=2Q$ and $\Omega$ be as in Definition \ref{def:am_modp}. We let:
\begin{itemize}
    \item[(a)] $\mathscr{P} (q, \Sigma)$ be the set of $m$-dimensional planes $\pi \subset T_q \Sigma$, where $m =\dim (T)$; 
    \item[(b)] $\mathscr{B} (q, \Sigma)$ be all the sets of the form
\[
\bS = \bigcup_{i=1}^N \bH_i\, ,
\]
where $2\leq N \leq 2Q$ and the $\bH_i$ are pairwise distinct $m$-dimensional half-planes of $T_q \Sigma$ joining at a common $(m-1)$-dimensional linear subspace $V = V (\bS) \subset T_q\Sigma$. Any such $\mathbf{S}$ will be called an {\em open book}, $V (\bS)$ will be called the {\em spine} \footnote{Note that, according to our definition, an $m$-dimensional plane $\pi \subset T_q\Sigma$ is an open book, and however in the latter case the spine $V$ is not uniquely defined and can be taken to be an arbitrary $(m-1)$-dimensional linear subspace of $\pi$. When we regard $\pi$ as an open book, we assume that a choice of $V$ has been specified, too.} of $\mathbf{S}$, and $\bH_i$ will be called the {\em pages of $\bS$}.
\end{itemize}
\end{definition}

We will simply write $\mathscr{P} (q)$ and $\mathscr{B} (q)$ when $\Sigma$ is clear from the context.

\begin{remark}\label{r:book-to-cones}
Observe that, if:
\begin{itemize}
    \item[(i)] $\bS \in \mathscr{B} (q)$, 
    \item[(ii)] we orient the halfplanes $\bH_i$ so that $\partial \a{\bH_i} = \a{V}$,
    \item[(iii)] and we choose multiplicities $\kappa_i\in [1, Q]\cap \mathbb N$ so that $\sum_i \kappa_i = 2Q$,
\end{itemize}
then $\bC := \sum_i \kappa_i \a{\mathbf{H}_i}$ is a cycle $\modp$. There is of course only a finite number of possible choices for the weights, and the choice is unique if and only if $N=2Q$. 
\end{remark}

Next we introduce various notions of excess that will be used throughout the paper.

\begin{definition} \label{d:excess of excesses}
Let $T \in \mathscr{R}_m(\Sigma)$, let $q\in \R^{m+n}$ and let $\bB_r(q)\subset \R^{m+n}$ be an open ball. 
%and let $\mathscr{P}(q)$ and $\mathscr{B}(q)$ be as above.
\begin{itemize}
    \item[(a)] The \emph{one-sided} $L^2$ excess of $T$ from $\bS \in \mathscr{B}(q)$ in $\bB_r(q) \subset \R^{m+n}$ is
    \[
    \bE(T,\bS,q,r) := r^{-(m+2)} \int_{\bB_r(q)} \dist^2(q'-q,\bS)\, d\|T\|(q')\,. 
    \]
    
    \item[(b)] The \emph{one-sided} $L^2$ excess of $\bS \in \mathscr{B}(q)$ from $T$ in $\bB_r(q)$ is defined by
    \[
    \bE(\bS,T,q,r) := r^{-(m+2)} \int_{\bS \cap (\bB_r \setminus B_{r/8}(V(\bS)))} \dist^2(q+q',\spt (T))\, d\Ha^m(q')\, ,
    \]
    where $B_s(V)$ denotes the tubular neighborhood of $V$ in $\R^{m+n}$ of radius $s$.
    \item[(c)] The \emph{double-sided} 
     $L^2$ excess between $T$ and $S \in \mathscr{B}(q)$ in $\bB_r(q)$ is
     \[\mathbb E (T,\bS,q,r) := \bE (T,\bS,q,r) + \bE(\bS,T,q,r)\,.
     \]
\end{itemize}
Furthermore we shall write:
\begin{itemize}
    \item[(d)] $\bE(T,q,r)$ and $\mathbb E (T,q,r)$ for, respectively, the minima of $\bE(T,\bS,q,r)$ and $\mathbb E(T,\bS,q,r)$ over all open books $\bS \in \mathscr{B}(q)$;
    \item[(e)] $\bE (T,\pi,q,r)$ with $\pi \in \mathscr{P}(q)$ and $\bar\bE (T,q,r)$ as in \eqref{e:plane-excess} and \eqref{e:minimal-plane-excess}, respectively.    
 \end{itemize}
\end{definition}

We will often denote with $\pi_{q,r}$ an \emph{optimizing plane} in the ball $\bB_r(q)$, i.e. such that 
\[
\bE(T,\pi_{q,r}, q,r)= \min_{\pi \in \mathscr{P}(q)} \bE(T,\pi,q,r) = \bar\bE(T,q,r)\,.
\]

\subsection{Allard's height bound}
We end up this section recalling a useful $L^\infty-L^2$ estimate due to Allard that will be used in several places later on in the paper.

\begin{lemma}[$L^\infty$-$L^2$ estimates] \label{Linfty-L2}
There exists a geometric constant $C>0$ such that, if $T, \Sigma$ are as in Definition \ref{def:am_modp}, $0\in \Sigma$, $\bB_1\cap \spt^p (\partial T) = \emptyset$, and $\|T\|(\bB_1) < (Q + \frac12) \omega_m$, then 
\begin{equation}
    \sup_{q\in \spt^p(T)\cap \bB_{15/16}} |\p_{\pi_0}^\perp (q)|^2 \leq C\,(\bE(T,\pi_0,0,1)+\bA^2) \qquad \mbox{for every $\pi_0 \in \mathscr{P}(0)$}\,.
\end{equation}
\end{lemma}

A proof can be found for instance in \cite[Lemma 1.7]{Spolaor}, and is based on an argument of Allard (see \cite[Theorem (6)]{Allard72}). Note that the argument in \cite[Lemma 1.7]{Spolaor} just uses the fact that $T$ induces a varifold in $\mathbb R^{m+n}$ with generalized mean curvature bounded by $\bA$.

%Next the desired parameterization result, which is split into two statements. The first is essentially borrowed from \cite{DLHMS}. 

\section{Excess decay in the two regimes and the proof of Theorem \ref{t:main-with-decay}}

% \begin{definition}
% Let $T$, $\Sigma$, $p=2Q$ and $\Omega$ be as in Definition \ref{def:am_modp}. We let:
% \begin{itemize}
%     \item $\mathscr{P} (q)$ be the set of $m$-dimensional planes $\pi \subset T_q \Sigma$, where $m =\dim (T)$; 
%     \item $\mathscr{B} (q)$ be all the sets of the form
% \[
% \bS = \bigcup_{i=1}^N \bH_i\, ,
% \]
% where $2\leq N \leq 2Q$ and the $\bH_i$ are pairwise distinct $m$ dimensional half-planes of $T_q \Sigma$ joining at a common $m-1$-dimensional plane $V = V (\bS)$.
% \end{itemize}
% We then let 
% \begin{align}
% \bar \bE (T, q, r) &:= \min_{\pi \in \mathscr{P} (q)} \bE (T, \pi, q, r)\\
% &:= \min_{\pi\in \mathscr{P} (q)} r^{-m-2} \int_{\bB_r (q)} \dist (x-q, \pi)^2\, d\|T\| (x)\, \\
% \bE (T, q,r) &:= \min_{\bS\in \mathscr{B} (q)} \bE (T, \bS, q, r)\\
% &:= \min_{\bS\in \mathscr{B} (q)} r^{-m-2} \Bigg[\int_{\bB_r (q)} \dist (x-q, \bS)^2\, d\|T\| (x)\\
% &\qquad\qquad +
% \int_{\bS \cap \bB_r\setminus B_{r/8} (V (\bS))} \dist (q+x, \spt (T))^2\, d\mathcal{H}^m (x)\Bigg]\, \, .
% \end{align}
% \end{definition}

% Observe that, if $\bS \in \mathscr{B} (q)$, we orient the halfplanes $\bH_i$ so that $\partial \a{\bH_i} = \a{V}$ and we choose multiplicities $\kappa_i\in [1, Q]\cap \mathbb N$ so that $\sum_i \kappa_i = 2Q$, then $S:= \sum_i \kappa_i \a{\mathbf{H}_i}$ is a cycle $\modp$. There is of course only a finite number of possible choices for the weights, which is in fact even unique when $N=2Q$. For $N<2Q$ there is instead more than one choice.

% We are now ready to state the two decay lemmas

For the rest of the paper we will mostly work under the following assumption:

\begin{ipotesi}\label{ass:everywhere}
We let $T$, $\Sigma$ be as in Definition \ref{def:am_modp} with $\Omega = \bB_1(0)$, $\bar n = 1$, and $p=2Q$, and let $T$ be a representative $\modp$. Moreover we assume $\bA \leq 1$, 
\begin{equation}\label{e:ass_density}
    \Theta_T (0) \geq Q\, , \qquad \bB_1(0)\cap \spt^p (\partial T ) = \emptyset\, ,\qquad \mbox{and}\qquad 
    \|T\| (\bB_1 (0)) \leq \left(Q+\textstyle{\frac{3}{4}}\right) \omega_m\, .
\end{equation}
\end{ipotesi}

Our main Theorem \ref{t:main-with-decay} will then be proved by showing two suitable decay propositions in two different regimes based on the value of the ratio $\bar\bE^{-1} \bbE$. 

\begin{proposition}\label{p:decay-1}
For every $p=2Q$, $m$, $n$, and any fixed $\delta_1>0$ there are $\frac{1}{2}\geq r_1 = r_1 (\delta_1, p,m,n)>0$ and $\varepsilon_1 = \varepsilon_1 (\delta_1,p,m,n)>0$ with the following property. Assume that Assumption \ref{ass:everywhere} holds and that in addition  \begin{equation}\label{e:ex_trapped} 
    \bA^2 \leq \varepsilon_1 \bar \bE (T, 0,1) \leq \varepsilon_1^2\,,
\end{equation}
then 
\begin{equation}\label{e:decay-1}
\mathbb{E} (T, 0, r_1) \leq \delta_1 \bar \bE (T, 0, 1)\, . 
\end{equation}
\end{proposition}

\begin{proposition}\label{p:decay-2}
For every $p=2Q$, $m$, $n$, there are $0 < r_2\leq \frac{1}{2}$, $\varepsilon_2> 0$, and $0 < \eta_2 \leq 1$ with the following property. Assume that Assumption \ref{ass:everywhere} holds and that in addition
\begin{equation}\label{e:ex_free}
\mathbb{E} (T, 0,1) \leq \eta_2 \bar \bE (T, 0, 1) \quad\mbox{and}\quad \bA^2 \leq \varepsilon_2 \mathbb{E} (T, 0, 1)\leq \varepsilon_2 \bar\bE (T,0,1) \leq \varepsilon_2^2
\, .
\end{equation}
Then 
\begin{equation}\label{e:decay-2}
\mathbb{E} (T, 0, r_2) \leq \frac{1}{2} \mathbb{E} (T, 0, 1)\, .
\end{equation}
\end{proposition}

Proposition \ref{p:decay-1} will be proved in Section \ref{s:decay-1}, whereas Sections \ref{s:Q-points} to \ref{s:final-step} will be devoted to the proof of Proposition \ref{p:decay-2}. In the rest of this section, we will show how Theorem \ref{t:main-with-decay} follows from the two decay propositions. First, we show the validity of a slightly modified version of Proposition \ref{p:decay-1}. 

\begin{corollary} \label{c:decay-1-improved}
For every $p=2Q$, $m$, $n$, and $\delta_1 > 0$ there are $\frac12 \geq r_1 = r_1(\delta_1, p,m,n) >0$ and $\varepsilon_3 = \varepsilon_3(\delta_1, p,m,n) > 0$ with the following property. If Assumption \ref{ass:everywhere} holds, and if furthermore \begin{equation}\label{e:ex_trapped_2} 
    \bA^2 \leq \varepsilon_3 \mathbb{E} (T, 0,\sfrac{1}{2}) \quad\mbox{and}\quad\bar \bE (T, 0,1) \leq \varepsilon_3\,, 
\end{equation} then \eqref{e:decay-1} holds true.
\end{corollary}
\begin{proof}
We show that \eqref{e:ex_trapped_2} implies \eqref{e:ex_trapped} when $\varepsilon_3$ is chosen sufficiently small. To this aim, it is sufficient to show that there exists a geometric constant $C > 0$ such that
\begin{equation} \label{e:2sided_vs_flat}
    \mathbb{E}(T,0,\sfrac{1}{2}) \leq C \left( \bar \bE(T,0,1) + \bA^2 \right)\,.
\end{equation}
To prove \eqref{e:2sided_vs_flat}, we let $\pi_0 \in \mathscr{P}(0) \subset \mathscr{B} (0)$ be a plane realizing $\bar\bE(T,0,1)$, so that
\begin{equation} \label{e:planes-are-books}
\mathbb{E}(T,0,\sfrac{1}{2}) \leq \mathbb{E}(T,\pi_0,0,\sfrac{1}{2})\,,
\end{equation}
where in the calculation of $\mathbb{E}(T,\pi_0,0,\sfrac{1}{2})$ the subspace $V(\pi_0)$ can be chosen arbitrarily. Next, if $\varepsilon_3$ is chosen sufficiently small then 
\[
((\mathbf{p}_{\pi_0})_\sharp (T\res\bB_1))\res \bB_{3/4} = Q \a{\pi_0\cap \bB_{3/4}} \quad \modp\, .
\]
In particular, by Lemma \ref{Linfty-L2} we easily see that for every $z\in \pi_0 \cap \bB_{3/4}$ there is a point $q\in \spt (T)$ such that $\mathbf{p}_{\pi_0} (q) = z$ and $|q-z|^2\leq C (\bar \bE (T,0,1) + \bA^2)$. This implies easily that 
\begin{equation}\label{e:1/2}
\bE (\pi_0, T, 0, \sfrac12) \leq C (\bar \bE (T,0,1) + \bA^2)\, .
\end{equation}
Since
\[
\bE(T,\pi_0,0,\sfrac12) \leq 2^{m+2} \, \bE(T,\pi_0,0,1)\,,
\]
\eqref{e:2sided_vs_flat} follows immediately from \eqref{e:planes-are-books} and \eqref{e:1/2}.
\end{proof}

Finally, before coming to the proof of Theorem \ref{t:main-with-decay} we come to another important ingredient, which is in fact an outcome of the analysis leading to Propositions \ref{p:decay-1} and \ref{p:decay-2}.

\begin{lemma}\label{l:compare-books}
For every $p=2Q,m,n$, and $\gamma>0$ there is $C(p,m,n, \gamma) >0$ with the following property. If Assumption \ref{ass:everywhere} holds, if $\bS$ and $\bS'$ are open books in $\mathscr{B} (0)$ realizing $\mathbb{E}(T,0,1)$ and $\mathbb{E}(T,0,r)$ respectively, and $r\geq \gamma$, then
\begin{equation}\label{e:compare-books}
d_{\mathcal{H}} (\bS\cap \bB_1, \bS'\cap \bB_1)^2 
\leq C \left(\bA^2 + \bbE(T, \bS, 0, 1) + \bbE (T, \bS', 0, r)\right)\, ,
\end{equation}
where $d_\mathcal{H}$ denotes Hausdorff distance. Moreover, there is a constant $C(p,m,n)$ such that
\begin{equation}\label{e:almost-monotonicity-book}
\bbE (T, \bS, 0, \textstyle{\frac{1}{2}}) \leq C \,
(\bbE (T, \bS, 0, 1)+\bA^2)\, .
\end{equation}
\end{lemma}

Lemma \ref{l:compare-books} will be proved in Section \ref{s:reparametrization}.

\subsection{Proof of Theorem \texorpdfstring{\ref{t:main-with-decay}}{t:decay}} Without loss of generality, using the scaling and translation invariance of the problem we assume that $r=1$ and $q=0$. We fix therefore $\Sigma, T, \Omega= \bB_1$, and assume that $\spt^p (\partial T) \cap \bB_1=\emptyset$, $\|T\| (\bB_1)< (Q + \frac{1}{2}) \omega_m$ and fix a plane $\pi_0$ such that 
\begin{equation}\label{e:start-epsilon}
\bar \bE (T, 0, 1) + \bA^2 = \bE (T, \pi_0, 0, 1) + \bA^2 \leq \varepsilon\, .
\end{equation}
The choice of $\varepsilon$ will be subject to various smallness specifications along the argument. In fact we first notice that, by the classical monotonicity formula, if it is smaller than some geometric constant then $\|T\| (\bB_r) \leq (Q + \frac{3}{4}) \omega_m r^m$ for every $r\leq 1$. In particular, Assumption \ref{ass:everywhere} holds for $T_{0,r}$ and $\Sigma_{0,r}$ in place of $T$ and $\Sigma$, whenever $r\leq 1$.

Next, we fix $\varepsilon_2$, $\eta_2$, and $r_2$ as in Proposition \ref{p:decay-2}. We then specify $\delta_1 = \frac{\eta_2}{2}$, and fix correspondingly $r_1$ and $\varepsilon_3$ as in Corollary \ref{c:decay-1-improved}. For convenience we define $\bar\varepsilon := \min \{\varepsilon_2, \varepsilon_3,1\}$ and we next proceed to define inductively a family of radii $t_k$ indexed by a set $\mathcal{K}$ which is either the set of natural numbers or the subset of all natural numbers up to a maximum $k_{max}$. The procedure will also give a suitable estimate for $\bbE (T,0, t_k)$.

First of all we set $t_0:=\frac{1}{2}$, and notice that  
\begin{equation}\label{e:decay-start}
\bbE (T,0,{\textstyle{\frac{1}{2}}})
\leq \bbE (T,\pi_0,0, {\textstyle{\frac{1}{2}}})\leq C (\bar \bE (T,0,1) + \bA^2)\,,
\end{equation}
as in the proof of Corollary \ref{c:decay-1-improved}.
Assume next that $t_k$ has been chosen and consider the current $T_{0, t_k}$, the manifold $\Sigma_{0,t_k}$ and $\bA_k^2 := \|A_{\Sigma_{0,t_k}}\|_\infty^2 = t_k^2 \bA^2\leq \varepsilon t_k^2$. We then examine the following three conditions:
\begin{align}
\bar \bE (T_{0,t_k}, 0, 1) &\leq \bar \varepsilon\label{e:condizione-a}\\
\bA_k^2 &\leq \bar \varepsilon \min \{\bbE (T_{0,t_k}, 0, 1),
\bbE (T_{0,t_k}, 0, {\textstyle{\frac{1}{2}}})\}\label{e:condizione-b}\\
\bbE (T_{0,t_k}, 0, 1) &\leq \eta_2 \bar\bE (T_{0, t_k}, 0, 1)\label{e:condizione-c}\, .
\end{align}
\begin{itemize}
\item[(a)] If \eqref{e:condizione-a} fails we set $k_{max}=k$.
\item[(b)] If \eqref{e:condizione-a} holds but \eqref{e:condizione-b} fails we set $t_{k+1}= \frac{t_k}{2}$ and we invoke \eqref{e:almost-monotonicity-book} to conclude
\begin{equation}\label{e:inductive-decay-1}
\bbE (T,0, t_{k+1}) \leq \frac{C}{\bar\varepsilon} \bA_k^2 = \leq C_b  t_{k+1}^2 \bA^2\, 
\end{equation}
for a constant $C_b$ depending only on $p,m,n$, and $\bar\varepsilon$.
\item[(c)] If \eqref{e:condizione-a} and \eqref{e:condizione-b} hold, but \eqref{e:condizione-c} fails we apply Corollary \ref{c:decay-1-improved}, set $t_{k+1}= r_1 t_k$ and estimate
\begin{align}
\bbE (T, 0, t_{k+1}) &= \bbE (T_{0, t_k}, 0, r_1)
\leq \delta_1 \bar \bE (T_{0, t_k}, 0, 1) = \delta_1 \bar \bE (T, 0, t_k)\nonumber\\
&\leq \frac{\delta_1}{\eta_2} \bbE (T, 0, t_k) =
\frac{1}{2} \bbE (T, 0, t_k)\, .\label{e:inductive-decay-2}
\end{align}
\item[(d)] If \eqref{e:condizione-a}, \eqref{e:condizione-b}, and \eqref{e:condizione-c} hold we apply
Proposition \ref{p:decay-2}, set $t_{k+1} = r_2 t_k$ and conclude
\begin{equation}\label{e:inductive-decay-3}
\bbE (T, 0, t_{k+1}) = \bbE (T_{0, t_k},0, r_2) \leq \frac{1}{2} \bbE (T_{0, t_k}, 0, 1) = \frac{1}{2} \bbE (T, 0, t_k)\, .
\end{equation}
\end{itemize}
Next observe that the following inequality holds for $k=0$ and for those $k$ for which $(k-1)$ falls under alternative (b), because $t_k^2 \leq t_k\leq 2^{-k}$:
\begin{equation}\label{e:wrap-it-up}
\bbE (T, 0, t_k) \leq C_b (\bA^2 + \bar \bE (T,0,1)) 2^{-k}\, .
\end{equation}
For any other $k\in \mathcal{K}$ we let $k'$ be the largest integer smaller than $k$ for which $k'-1$ falls in alternative (b), if it exists, or set $k'=0$. We now can use \eqref{e:inductive-decay-2} and \eqref{e:inductive-decay-3} and the validity of \eqref{e:wrap-it-up} for $k'$ to conclude
\[
\bbE (T, 0, t_k) \leq 2^{-(k-k')} \bbE (T, 0, t_{k'})
\leq 2^{-(k-k')} C_b (\bA^2 + \bar \bE (T,0,1)) 2^{-k'}\, ,
\]
and hence the validity of \eqref{e:wrap-it-up} for every $k\in \mathcal{K}$.
Next set $\gamma := \min \{r_1, r_2\}$ and observe that $\gamma \leq \frac{t_{k+1}}{t_k}\leq \frac{1}{2}$. For each $k$ let $\bS_k$ be an open book such that $\bbE (T, 0, t_k) = \bbE (T, \bS_k, 0, t_k)$. By Lemma \ref{l:compare-books} we have
\begin{align}
d_{\mathcal{H}} (\bS_k \cap \bB_1, \bS_{k-1}\cap \bB_1)^2
&\leq C (\bA^2 + \bar \bE (T,0,1)) 2^{-k} \qquad \forall k\in \mathcal{K}\setminus \{0\}\label{e:tilt-books}\\
d_{\mathcal{H}} (\pi_0 \cap \bB_1, \bS_0\cap \bB_1)^2
&\leq C (\bA^2 + \bar \bE (T,0,1))\, , \label{e:tilt-books-start}
\end{align}
where in \eqref{e:tilt-books-start} we have applied \eqref{e:compare-books} to the current $T_{0,t_0} = T_{0,\frac12}$ together with \eqref{e:decay-start}. In particular we conclude that 
$d_{\mathcal{H}} (\pi_0 \cap \bB_1, \bS_k\cap \bB_1)^2
\leq C (\bA^2 + \bar \bE (T,0,1)) \leq C \varepsilon$, 
which in turn, together with \eqref{e:wrap-it-up} implies $\bE (T, \pi_0, 0, t_k) 
\leq C \varepsilon$.
Since the constant $C$ is independent of $\varepsilon$, upon choosing $\varepsilon$ sufficiently small, we conclude
\begin{equation}\label{e:no-escape-3}
\bE (T, \pi_0, 0, t_k) 
\leq \bar \varepsilon\, 
\qquad \forall k\in \mathcal{K}\, .
\end{equation}
On the other hand the latter estimate implies that alternative (a) never applies and the inductive procedure never stops, namely $\mathcal{K}=\mathbb N$. 

Observe next that \eqref{e:tilt-books} implies that $\bS_k \cap \overline{\bB}_1$ is a Cauchy sequence of compact sets in the Hausdorff distance. It thus converges to $\bS\cap \overline{\bB}_1$ for some unique open book $\bS$. Consider next a sequence $r_j\downarrow 0$ with the property that $T_{0, r_j}$ converges to $Q \a{\pi}$ for some plane $\pi$. We can find a sequence $k(j)$ so that $t_{k(j)+1}\leq r_j \leq t_{k(j)}$. Given that $\frac{t_{k(j)+1}}{t_{k(j)}} \geq \gamma$, we immediately conclude that $T_{0, t_{k(j)}}$ converges to $Q \a{\pi}$. On the other hand this implies that $\bbE (Q \a{\pi}, \bS, 0, 1) = 0$, which in turn forces the equality $\bS = \pi$. 

Next, summing the appropriate tail of the series \eqref{e:tilt-books} we immediately see that 
\[
d_{\mathcal{H}} (\bS_k \cap \bB_1, \pi \cap \bB_1)^2 \leq C (\bA^2 + \bar \bE (T,0,1)) 2^{-k}\, .
\]
Combined with \eqref{e:wrap-it-up} we conclude
\[
\bE (T, \pi, 0, t_k) \leq C (\bA^2 + \bar \bE (T,0,1)) 2^{-k}\, .
\]
Finally, for any $r\leq \frac12$ we choose $k$ so that $t_{k+1}\leq r \leq t_k$ and we immediately conclude
\[
\bE (T,\pi, 0, r) \leq C \gamma^{-m-2} \bE (T, \pi_0, 0, t_k)
\leq C (\bA^2 + \bar \bE (T,0,1)) 2^{-k}\, .
\]
Since $r\geq t_{k+1} \geq \gamma^{k+1}$, the latter implies the desired estimate \eqref{e:decay-0} for $\alpha =- \log_\gamma 2$. \qed

\section{Graphical parametrizations over planes}

Next we introduce the graphical parametrization that will play in this paper the same role that Allard's and White's regularity results play in \cite{Simon} and \cite{DHMSS}, respectively. The first proposition follows essentially from \cite[Theorem 16.1]{DLHMS} once we can show that the tilt-excess is controlled by the $L^2$ planar excess (an estimate which can be reduced to Allard's classical work). We follow the notation of \cite{DHMSS}, and for planes $\pi \subset T_q \Sigma$ we denote by $\pi^\perp$ their orthogonal complement in $\mathbb R^{m+n}$ and by $\pi^{\perp_q}$ their orthogonal complement in $T_q \Sigma$. Moreover, we set, for an open set $\Omega \subset \R^{m+n}$, 
\[
\bh (T, \Omega, \pi) := \sup \{|\mathbf{p}_{\pi^\perp} (x-y)|: x,y\in \spt (T)\cap \Omega\}\, ,
\]
and we introduce two further notions of excess, which are ``$W^{1,2}$-based'' rather than $L^2$-based.
\begin{definition}\label{d:E-cessi} Let $T$ be a representative mod $p$.
\begin{itemize}
    \item[(a)] $\bE^{{o}}$ is the \emph{oriented} \footnote{The other notions of excess have the property that they only depend on the $\modp$ class $[T]$ of $T$, as long as $T$ is a representative mod $p$. This is however not the case for the oriented tilt excess.} tilt excess of $T$ with respect to a plane. More precisely, let $\pi \in \mathscr{P}(q)$ be oriented by the unit $m$-vector $\vec \pi$, and set $\bC_r(q,\pi):= B_r(q,\pi) \times \pi^\perp$, where $B_r(q,\pi):= \bB_r(q) \cap (q + \pi)$. Then, we set, for $\Omega = \bB_r(q)$ or $\Omega = \bC_r(q,\pi)$:
     \[
     \bE^{{o}}(T,\pi,\Omega) := \frac{1}{2\omega_m r^m} \,\int_{\Omega} |\vec T(q') - \vec \pi|^2 \,d\|T\|(q')\, ;
     \]
     \item[(b)] $\bE^{{no}}$ for the \emph{unoriented} tilt excesses of $T$ with respect to a plane $\pi \in \mathscr{P}(q)$, namely
     \[
     \bE^{{no}}(T,\pi,\Omega) := \frac{1}{2\omega_m r^m} \,\int_{\Omega} |\vec T(q') - \pi|_{{no}}^2 \,d\|T\|(q')\,,
     \]
     where
     \[
     |\pi_1 - \pi_2|_{{no}} := \min \{|\vec\pi_1 - \vec\pi_2|, |\vec\pi_1 + \vec \pi_2|\}\,.
     \]
     \end{itemize}
\end{definition}
Note moreover that in place of $|\pi_1-\pi_2|_{{no}}$ we could use the integrand $|\mathbf{p}_{\pi_1}-\mathbf{p}_{\pi_2}|$, as the two integrands are equivalent up to a multiplicative constant.

\begin{remark}\label{r:special-multi-functions}
While we will use the notation and terminology of \cite{DLHMS_linear} for ``special'' $Q$-valued functions, since we are in a codimension one context and we will always deal with Lipschitz multifunctions, in our case they reduce to the specification of the following classical maps:
\begin{itemize}
\item $Q$ Lipschitz maps $v_1, \ldots, v_Q$ defined on some domain $\Omega$ of an oriented $m$-dimensional plane $\pi_q \subset T_q \Sigma$ with values into its orthogonal complement $\pi_q^\perp$, taking values in $\Sigma$. This means that the maps will take the special form $v_i (z) = u_i (z) + \Psi (z+u_i (z))$, where $\Psi: T_q \Sigma \mapsto T_q \Sigma^\perp$ parametrizes $\Sigma$ as a graph over $T_q \Sigma$ and $u_i$ takes values on the real line $\pi_q^{\perp_q}$. 
\item A map $\varepsilon_v : \Omega \to \{-1,1\}$. This map specifies whether we should consider the tangent planes to the graphs of each $v_i$ at $(z, v_i (z))$ as positively oriented (i.e. having the same orientation as the pushforward of $\pi_q$), or negatively oriented. Note that, by definition of special $Q$-valued functions, at any given point $z$ either all such planes are positively oriented, or they are all negatively oriented.  
\end{itemize}
Following \cite{DLHMS_linear} such an object will be denoted by $(\sum_i \a{v_i}, \varepsilon_v)$ and $\mathbf{G}_v$ will denote the integer rectifiable current associated to it (the ``graph'' of $v$), which happens to be a representative mod~$2Q$, with no boundary (mod~$2Q$) in the cylinder $\Omega \times \pi_q^\perp$. If not otherwise specified, the functions will be ordered so that $u_1 \leq u_2 \leq \ldots \leq u_Q$, for some ordering of the real line $\pi_q^{\perp_q}$ (a canonical choice of ordering of $\pi_q^{\perp_q}$ is the one which is compatible with the orientations of $T_q \Sigma$ and $\pi_q$). 
\end{remark}

The main results of this section are Propositions \ref{p:L2Alm}, \ref{p:L2Alm-piece-2}, and \ref{p:L2Alm-piece-3} below.

\begin{proposition}[Multivalued approximation] \label{p:L2Alm}
For every $p=2Q$, $m$, and $n$ there exist constants $\eps_G,\gamma,C>0$ depending on $(Q,m,n)$, with the following property. Assume that 
\begin{itemize}
    \item[(a)] $T$, $\Sigma$, and $\Omega = \bB_1$ are as in Definition \ref{def:am_modp} with $\bar n = 1$ and $p=2Q$, and $T$ is a representative $\modp$;
    \item[(b)] $\bar\bE + \bA^2 := \bar\bE (T, 0,1) + \bA^2 \leq \varepsilon_G^2$, and  $\bar\bE(T,0,1) = \bE(T,\pi_0,0,1)$;
    \item[(c)] $\bB_1\cap \spt^p (\partial T ) = \emptyset$;  
    \item[(d$1$)] either there exists $\xi \in \bB_{1/16}$ such that $\Theta_T(\xi) \ge Q$ 
    %{\color{red}[Jonas:] Do I miss something: What is with $Q+1$ planes intersecting in $0$ all with multiplicity 1. They should project to a 1-plane and are in $L^2$ as close as we want to a single plane?}  {\color{blue} Camillo: This is not an area-minimizing cone mod $2Q$: we have a classification theorem which excludes it, see \cite[Theorem 3.5]{DHMSS}.}
    \item[(d$2$)] or 
    \begin{equation}\label{e:(d')}
    ({\bf{p}}_{\pi_0})_\sharp(T\res \bC_{7/8}(\pi_0)\cap \bB_1)=Q\,\a{\pi_0\cap \bB_{7/8}}\, .
    \end{equation}
\end{itemize}
Then, there exist a function $u \colon B_{3/4}:=B_{3/4}(0,\pi_0) \to \mathscr{A}_Q(\pi_0^{\perp_0})$, and a closed set $K \subset B_{3/4}$ such that, if we set
\begin{equation} \label{lifting up}
v(z) := \left( \sum_{i=1}^Q \a{v_i(z)}, \eps_u(z) \right)\,, \qquad v_i(z) := u_i(z) + \Psi(z+u_i(z))\,,
\end{equation}
then the following holds:
\begin{align}\label{e:inclusion}
&\spt(\bG_v) \subset \Sigma\,, \\ \label{e:lip_osc_est}
&\Lip (v)\leq C(\bar \bE + \bA^2)^\gamma \quad \mbox{and} \quad 
{\rm osc} (v) \leq C (\bar \bE + \bA^2)^{\sfrac{1}{2}} 
+ \bh (T, \bB_{15/16}, \pi_0), \\ \label{e:graph_current}
&\bG_v \res (K\times \pi_0^\perp)= T\mres (K\times \pi_0^\perp)\cap \bB_{15/16} \quad \modp\,,\\
\label{e:volume_estimate}
&|B_{3/4}\setminus K|\leq \|T\| (((B_{3/4}\setminus K)\times \pi_0^\perp)\cap \bB_{15/16}) \leq C (\bar \bE + \bA^2)^{1+\gamma}\,,
\end{align}
\begin{align}
& \Abs{\|T\|(\bC_{3/4}\cap \bB_{15/16}) - Q |B_{3/4}| - \frac{1}{2} \int_{B_{3/4}} \abs{Dv}^2} \leq C\,(\bar \bE + \bA^2)^{1+\gamma}\, ,\label{e:Dirichlet}\\
&\|v\|_{L^\infty}^2 + \int_{B_{3/4}} |Dv|^2 \leq C\, (\bar \bE + \bA^2)\, .\label{e:W12}
\end{align}
\end{proposition}

The next proposition adds two conclusions which are 
useful in our situation and which follow from a careful combination of the estimates in Proposition \ref{p:L2Alm} with the classical monotonicity formula.

\begin{proposition}\label{p:L2Alm-piece-2}
Let $T$, $\Sigma$, $u$, and $v$ be as in Proposition \ref{p:L2Alm}. Then:
\begin{itemize}
\item[(i)] If $q_0=(z_0,w_0)\in B_{1/4}\times B_{1/4} \subset \pi_0 \times \pi_0^\perp$ and $\Theta_T(q_0) \geq Q$, then
\begin{equation}\label{eq.cheap Hardt-Simon}
   \int_{B_{r_0}(z_0,\pi_0)\cap K} \frac{1}{\abs{z-z_0}^{m-2}}\sum_{i=1}^Q \left|\partial_r \frac{(v_i(z) - w_0)}{\abs{z-z_0}}\right|^2\, dz \le C r_0^{-m}  \left(\bar{\bE} +\bA^2\right)\quad \forall r_0< \frac{1}{4}\, .
\end{equation}
\item[(ii)] If $\bS \in \mathscr{B} (0)$ has spine $V=V(\bS)$ and $(\omega_m^{-1} \bE(\bS,T,0,1))^{\frac{1}{m+2}} \leq \frac18$ then for any $\rho_0 \in \left[(\omega_m^{-1} \bE(\bS,T,0,1))^{\frac{1}{m+2}}, \frac18 \right]$, upon setting \[\bS (\rho):= \bS\cap \bB_{7/8} \cap \bC_{\frac12} \setminus B_{\frac18+\rho}(V)\,,\] we have that 
\begin{equation}\label{eq:graphical_madness}
    \int_{\bS (\rho_0)} \dist^2(q, \spt({\bf G}_v))\, d\Ha^m\leq C \int_{\bS (\rho_0)} \dist^2(q,\spt (T))\,d\Ha^m + C\left(\bar \bE +\bA^2\right)^{1+\beta}\, ,
\end{equation}
for some positive geometric constant $\beta = \beta (m,Q)$.
\end{itemize}
\end{proposition}

In the final proposition we take advantage of the regularity theory for area-minimizing currents in codimension 1. Before coming to the statement we introduce a suitable cylindrical version of the $L^2$ excess which is given by 
\begin{equation}\label{e:cylindrical-excess} 
\bE (T, \pi_0, \bC_r (z, \pi_0)) := \frac{1}{r^{m+2}} \int_{\bC_r (z, \pi_0)} \dist (q-z, \pi_0)^2\, d \|T\| (q)\, .
\end{equation}

\begin{proposition}\label{p:L2Alm-piece-3}
Let $T$ and $\Sigma$ be as in Proposition \ref{p:L2Alm}. The approximating map $v$ and the set $K$ can be required to satisfy the following additional property. Assume that for some $q = (z,w) \in \spt (T)\cap \bB_{7/8} \cap \bC_{1/2} \subset \pi_0 \times \pi_0^\perp$ and some cylinder $\bC_{2 r} (z,\pi_0) \subset \bC_{1/2}$ the following holds:
\begin{align} 
    &\Theta_T(q') < Q \qquad \forall q'\in \bC_{2r} (z,\pi_0) \cap \bB_{7/8},\label{no Q points}\\
    &r^{m} \geq (\bar{\mathbf{E}} + \bA^2)^{1- \gamma}\, . \label{e:tamer-cylindrical excess} 
\end{align}
Then $B_r (z) \subset K$ and there is a $C^{1,\sfrac{1}{2}}$ selection for $u|_{B_r (z)}$. More precisely:
\begin{itemize}
    \item[(i)] $\eps_u$ is constant on $B_r (z)$;
    \item[(ii)] there are $C^{1,\sfrac{1}{2}}$ functions $u_1 \leq \ldots \leq u_Q : B_r (z) \to \pi_0^{\perp_0} = \mathbb R$ such that $u = (\sum_i \a{u_i}, \varepsilon_u)$, 
    \item[(iii)] for all $i<j$, either $u_i (\zeta) < u_j (\zeta)$ $\forall \zeta\in B_r (z)$ or $u_i (\zeta)=u_j (\zeta)$ $\forall \zeta\in B_r (z)$;
    \item[(iv)] The following estimate holds for every $i \in \{1,\ldots,Q\}$
\begin{align}
& \|Du_i\|_ {C^0 (B_r(z))} + r^{-\sfrac{1}{2}} [Du_i]_{\sfrac{1}{2}, B_r (z)} \leq C\,(\bE(T\res \bB_{7/8}, \pi_0, \bC_{2r} (z,\pi_0)) + \mathbf{A}^2)^{\sfrac{1}{2}}\, . \label{eq:C12est}
\end{align}
\end{itemize}
\end{proposition}

\begin{proof}[Proof of Proposition \ref{p:L2Alm}] Having fixed the plane $\pi_0$, we will write $\bE^{no}(T,\Omega)=\bE^{no}(T,\pi_0,\Omega)$ to simplify the notation. First, observe that the unoriented excess $\bE^{no}(T,\bB_1)$ introduced in Definition \ref{d:E-cessi} is in fact equivalent, up to multiplicative constants, to the classical varifold excess (see \cite{Allard72}) of the varifold $\V(T)$ associated to $T$. Invoking \cite[Lemma 5.1]{DLHMS}, we have that 
\[
\delta \V(T) [X] = - \int X \cdot \vec{H}_T (x)\, d\|T\| (x) \qquad \mbox{for all $X \in C^1_c(\bB_1;\R^{m+n})$}\,,
\]
where $\vec{H}_T$ is a Borel function satisfying $\|\vec{H}_T\|_\infty \leq C \bA$. Hence, we can use the classical tilt excess inequality, cf. \cite[Proposition 4.1]{D-Allard}, to achieve, for every $z_0 \in B_{3/4}=B_{3/4}(0,\pi_0)$ and for any $r_0$ with $9 r_0 < 1/8$,  
\begin{equation}\label{e:tilt-excess}
\bE^{no} (T, \bB_{9r_0}(z_0)) \leq C (\bar \bE + \bA^2)\, .
\end{equation}
By Lemma \ref{Linfty-L2}, 
\begin{align*}
\sup \left\{|\mathbf{p}_{\pi_0}^\perp (x)|:x\in \spt (T)\cap \bB_{15/16}\right\}
\leq  & C (\bar\bE + \bA^2)^{\sfrac{1}{2}}\,,
\end{align*}
so that, if  we set $T' := T\res \bB_{15/16}$ and $\bC_{7/8}:=\bC_{7/8}(0,\pi_0)$, we then conclude that 
\[
\partial T' \res \bC_{7/8} = 0 \; \modp
\]
and that 
\begin{equation} \label{e:height-bound-pf}
\bh (T', \bC_{7/8}, \pi_0) \leq \bh (T, \bB_{15/16}, \pi_0) < C (\bar\bE + \bA^2)^{\sfrac{1}{2}}\, , 
\end{equation}
and thus
\[
\bE^{no}(T',\bC_{8r_0}(z_0)) \leq \bE^{no}(T,\bB_{9r_0}(z_0)) \leq C(\bar\bE + \bA^2)\,.
\]
Observe next that, by the constancy lemma \modp (see \cite[Lemma 7.4]{DLHMS}), there exist (up to a change of orientation of $\pi_0$) an integer $1 \leq k \leq Q$ such that
\[
(\mathbf{p}_{\pi_0})_\sharp T' \res \bC_{7/8} = k \a{\pi_0 \cap \bB_{7/8}} \modp\, .
\] 
We claim that it is necessarily $k=Q$. Indeed, if $k < Q$, assuming $\varepsilon_G$ sufficiently small we can appeal to White's regularity theorem \cite[Theorem 4.5]{White86}, and conclude that $T' \mres \bC_{7/64}$ is (the current associated to) a regular submanifold of $\Sigma$, which, in particular, would be free of points $q$ with $\Theta_T(q) \geq Q$: thus, having $k < Q$ would contradict both (d2) and (d1).

We can now appeal to \cite[Theorem 16.1]{DLHMS} and \eqref{e:tilt-excess} to conclude that 
\[
\bE^o (T', \bC_{4r_0}(z_0)) \leq C (\bar \bE + \bA^2)\,.
\]
 We can thus apply \cite[Theorem 15.1]{DLHMS} to find a map $u^{z_0}: B_{r_0} (z_0, \pi_0) \to \mathscr{A}_Q (\pi_0^{\perp_0})$ and a closed set $K^{z_0} \subset B_{r_0}(z_0,\pi_0)$ so that the associated map $v^{z_0}$ as in \eqref{lifting up} satisfies \eqref{e:inclusion} to \eqref{e:Dirichlet} on the cylinder $\bC_{r_0}(z_0,\pi_0)$, where $\gamma>0$ is a geometric constant. Using the same arguments as in \cite[Section 6.2]{DLS_Center} to patch the maps $u^{z_0}$ together when $z_0 \in B_{3/4}$ varies, we conclude the existence of a unique map $u \colon B_{3/4} \subset \pi_0 \to \mathscr{A}_Q(\pi_0^{\perp_0})$ and a closed set $K \subset B_{3/4}$ so that the associated map $v$ satisfies \eqref{e:inclusion} to \eqref{e:Dirichlet}. 
Moreover, 
using \cite[Theorem 16.1]{DLHMS}, a standard covering argument, and the tilt excess inequality as in \eqref{e:tilt-excess}, it holds
\[
\|T'\|(\bC_{3/4}) - Q|B_{3/4}| \leq C (\bE^{no}(T',\bC_{7/8}) + \bA^2) \leq C(\bar\bE + \bA^2)\,, 
\]
we can use \eqref{e:Dirichlet} to infer
\begin{equation}\label{e:DuL2}
\int_{B_{3/4}} |Dv|^2 \leq C (\bar \bE + \bA^2)\, .
\end{equation}
Now, the second estimate in \eqref{e:lip_osc_est} and \eqref{e:height-bound-pf} give immediately
\[
{\rm osc}\, (v) \leq  C (\bar{\bE} + \bA^2)^{\sfrac{1}{2}}\, .
\]
However observe also that 
\[
\sup \{|\mathbf{p}_{\pi_0^\perp} (q)|: q\in \spt (T) \cap \bB_{15/16} \cap \bC_{3/4}\} \leq C (\bar{\bE} + \bA^2)^{\sfrac{1}{2}}\,,
\]
while $\spt (T) \cap \spt (\bG_v) \cap \bB_{15/16} \cap \bC_{3/4}$ is certainly nonempty. We thus conclude the estimate $\|v\|_{L^\infty}^2 \leq C (\bar{\mathbf{E}} + \bA^2)$.
\end{proof}
\begin{proof}[Proof of Proposition \ref{p:L2Alm-piece-2}]
In order to prove \eqref{eq.cheap Hardt-Simon}, let $q_0=(z_0,w_0)$ be as in the statement of the proposition and let $r_0 < \frac{1}{4}$. Observe first that, since $\bB_{r_0}(q_0) \subset \bB_{15/16}\cap \bC_{3/4}$, it follows rather easily from Proposition \ref{p:L2Alm} that 
\begin{align*}
\frac{\norm{T}(\bB_{r_0}(q_0))}{\omega_m r_0^m} - \Theta_T(q_0)
    &\le \frac{\norm{\bG_v}(\bB_{r_0}(q_0))}{\omega_m r_0^m} - Q + \frac{\norm{T-\bG_v}(\bB_{r_0}(q_0))}{\omega_m r_0^m}  \\
    & \le C\left( \frac{1}{r_0^m} \int_{B_{3/4}} \abs{Dv}^2 + \frac{1}{r_0^m} (\bar \bE+\bA^2) \right) \stackrel{\eqref{e:Dirichlet}}{\le}  \frac{C}{r_0^m}\, (\bar \bE+\bA^2)\,.
\end{align*}
We then use the monotonicity formula and the fact that $\Theta_T(q_0)=Q$ to infer that
\begin{align*}
    \int_{\bB_{r_0}(q_0)} \frac{\abs{(q-q_0)^\perp}^2}{\abs{q-q_0}^{m+2}}\, d\norm{T} & \le \frac{\|T\| (\bB_{r_0} (q_0))}{\omega_m r_0^m} - \Theta_T (q_0) + C \bA^2 r_0^2\\
    &\le C r_0^{-m} (\bar{\bE} + \bA^2)\, . 
\end{align*}
Note that the usual monotonicity formula for varifolds with bounded mean curvature (as in \cite{Allard72}) would give an error term of type $C\bA$ in the first line. In order to get a quadratic error it suffices to invoke the argument in \cite[Appendix A]{DLS_Lp}, which uses the stronger fact that $\|T\|$ is a stationary varifold in the Riemannian manifold $\Sigma$, cf. also \cite[Remark A.2]{DLS_Lp}. 

Using the Lipschitz continuity of $v$, the same argument as in the proof of \cite[Proposition 8.3]{DHMSS} shows that
\begin{align*}
    \int_{B_{r_0}(z_0)\cap K} \frac{1}{\abs{z - z_0}^{m-2}}\sum_{i=1}^Q \left|\partial_r \frac{(v_i(z)- w_0)}{\abs{z-z_0}}\right|^2
    & \leq C \int_{\bB_{r_0}(q_0)}  \frac{\abs{(q-q_0)^\perp}^2}{\abs{q-q_0}^{m+2}}\, d\norm{T} \\
    &\leq C r_0^{-m}\, (\bar \bE+\bA^2)\,,
\end{align*}
thus proving \eqref{eq.cheap Hardt-Simon}. 

We now come to \eqref{eq:graphical_madness}. Recall that $\bS (\rho) :=\bS\cap \bB_{7/8} \cap (\bC_{\frac12} \setminus B_{\frac18+\rho}(V))$. Upon introducing the set $\bC_K:= K \times \pi_0^\perp$, we note that it is enough to show
\begin{equation}\label{eq:graph_mad_1}
 \int_{\bS (\rho_0)} \dist^2(q, \spt(T\res \bC_K) \cap \bB_{15/16})\, d\Ha^m\leq C \int_{\bS (\rho_0)} \dist^2(q,\spt (T))\,d\Ha^m + C\left(\bar\bE +\bA^2\right)^{1+\beta}\,,
\end{equation}
since then \eqref{eq:graphical_madness} follows from
\begin{align*}
    \int_{ \bS (\rho_0) } \dist^2(q, \spt(\bG_v))\, d\Ha^m
        & \leq \int_{\bS (\rho_0)} \dist^2(q, \spt(\bG_v\res \bC_K) \cap \bB_{15/16})\, d\Ha^m \\
        & = \int_{\bS (\rho_0)} \dist^2(q, \spt(T\res \bC_K) \cap \bB_{15/16})\, d\Ha^m \,.
\end{align*}
Towards the proof of \eqref{eq:graph_mad_1}, define the set
\[
U:=\left\{q \in \bS (\rho_0) \, \colon\, \dist(q,\spt(T)) < \frac12\, \dist(q,\spt(T\res \bC_K)\cap \bB_{15/16}) \right\}\,.
\]
It is clear that \eqref{eq:graph_mad_1} holds true provided we can show that, for a suitable choice of $\rho_0$,
\begin{equation} \label{eq:the Vitali claim}
    \mathcal{H}^m (U) \leq C (\bar\bE + \bA^2)^{1+\beta}\,.
\end{equation}

To this aim, first observe that if we set \[\delta_q:=\frac12\min\left\{\dist(q, \spt(T)), \rho_0,\frac{1}{8}\right\} \qquad \mbox{for $q \in \bS({\rho_0})$}\,,\]
then one has
\begin{equation}
 \bE(\bS,T,0,1)\geq \int_{\bB_{\delta_q}(q) \cap \bS} \dist^2(x, \spt(T))\, d\Ha^m(x) \geq \omega_m\,\delta_q^{m+2}\,,
\end{equation}
where in the second inequality we have used that $\dist(x,\spt(T)) \geq \delta_q$ for all $x \in \bB_{\delta_q}(q)$ due to the definition of $\delta_q$. It follows that if $1/8 \geq \rho_0 \geq 10\,\left( \omega_m^{-1}\,\bE(\bS,T,0,1) \right)^{\frac{1}{m+2}}$, then 
\begin{equation}\label{eq:vitali}
\dist(q, \spt(T))=2\delta_q\leq 2\, (\omega_m^{-1}\,\bE(\bS, T, 0,1))^{\frac1{m+2}} \leq \frac{1}{40}\qquad \forall q \in  \bS(\rho_0)\,.
\end{equation}
In particular,  
\begin{equation} \label{e:everything_in_the_right_place}
\dist(q,\spt(T)) = \dist(q,\spt(T) \cap \bC_{5/8} \cap \bB_{29/32}) \qquad \mbox{for all $q \in \bS(\rho_0)$}\,.
\end{equation}

We will estimate the measure of $U$ by a Vitali covering argument. We apply Vitali's covering theorem to the family of balls $\{ \bB_{2r (q)}(q)\,\colon\, q \in U\}$ with $r (q) := \dist(q,\spt(T))$ to find a disjoint subfamily $\{\bB_{2r (q_i)}(q_i)\}$
such that \[
U \subset \bigcup_{i} \bB_{10r (q_i)}(q_i)\,.
\]
For each $i$, fix $p_i \in \spt(T)$ such that $\abs{q_i-p_i}=\dist(q_i,\spt(T))=r (q_i)$. Notice that $p_i \in \spt(T) \cap \bC_{5/8} \cap \bB_{29/32}$ as a consequence of \eqref{e:everything_in_the_right_place}. Hence 
\begin{equation}\label{eq:vitali2}
   \bB_{r (q_i)}(p_i) \cap \spt(T\res \bC_K) =\bB_{r (q_i)}(p_i) \cap \spt(T\res \bC_K) \cap \bB_{15/16} =  \emptyset\,,
\end{equation}  
for otherwise, given the definition of $r(q_i)$, one would contradict the fact that $q_i \in U$. Notice that, since for every $i$ we have $\bB_{r(q_i)}(p_i)\subset \bB_{2r(q_i)}(q_i)$, then also $\{\bB_{r(q_i)}(p_i)\}_i$ is a disjoint family. We recall next the density lower bound for area minimizing currents mod($p$), that is $\omega_m r^m \le 2\norm{T}(\bB_r(\tilde q))$ for all $\tilde q \in \spt(T)$, which holds provided $\bA$ is smaller than a geometric constant. We then have
\begin{align*}
    \mathcal{H}^m(U) 
        &\le p \sum_i 10^m \omega_m r(q_i)^m \le 2p\cdot 10^m  \sum_i \norm{T}\left( \bB_{r (q_i)}(p_i)\right) \\
        &= 2p \cdot 10^m \norm{T}\left( \bigcup_i \bB_{r(q_i)}(p_i) \right)  \\
        &\le 2p \cdot 10^m \norm{T}\left(\bB_{15/16} \cap (B_{\frac34}\setminus K)\times \pi_0^\perp\right) \leq C (\bar{\bE}+\bA^2)^{1+\gamma}\,,
\end{align*}
where in the last inequality we have used \eqref{e:volume_estimate} and the second last inequality is a consequence of \eqref{eq:vitali}, \eqref{e:everything_in_the_right_place}, and \eqref{eq:vitali2}. We have thus proved that \eqref{eq:the Vitali claim} holds with $\beta=\gamma$, and the proof of \eqref{eq:graphical_madness} is complete.
\end{proof}

\begin{proof}[Proof of Proposition \ref{p:L2Alm-piece-3}] In order to simplify our notation, we set $T':= T\res \bB_{7/8}$. Note first that 
under the additional assumption \eqref{no Q points}
we can apply \cite[Lemma 9.5]{DHMSS} to deduce that $T'$ is a classical area minimizing current in $\bB_{2r}(q)$, and thus, thanks to Lemma \ref{Linfty-L2} and \eqref{e:tamer-cylindrical excess}, in $\bC_{7r/4}(z,\pi_0)$. We can then apply the standard decomposition of codimension $1$ area minimizing currents in sum of area minimizing boundaries with constant multiplicities, and De Giorgi's $\eps$-regularity theorem with $L^2$-excess (see for instance \cite[Theorem 4.5]{CoEdSp}), so to conclude that in $\bC_{3r/2} (z,\pi_0)$ the support $\spt(T')$ coincides with the union of the graphs of finitely many $C^{1,\sfrac{1}{2}}$ functions $\tilde{v}_1, \ldots, \tilde{v}_N$ with the property that $\tilde{v}_i (\zeta) = \tilde{u}_i (\zeta) + \Psi (\zeta+ \tilde{u}_i (\zeta))$ and $\tilde{u}_1\leq \tilde{u}_2 \leq \ldots \leq \tilde{u}_N$. Observe that, because of the assumption \eqref{e:tamer-cylindrical excess} and the estimate \eqref{e:volume_estimate}, $K\cap B_r (z)$ can be assumed to have positive measure, provided $\varepsilon$ is chosen sufficiently small. In particular we conclude that $N=Q$ and that the multifunctions $\sum_i \a{\tilde{v}_i}$ and $\sum_i \a{v_i}$ coincide on a set of positive measure. Because of the constancy lemma we immediately conclude the existence of a constant $\tilde{\varepsilon}\in \{-1,1\}$ such that  
\[
T' \res \bC_{3r/2 } (z, \pi_0) = \mathbf{G}_{\tilde{v}}
\]
for the special multivalued function 
\[
\tilde{v} = (\sum_i \a{\tilde{v}_i}, \tilde{\varepsilon})\, .
\]
We remark in passing that the estimate 
\begin{align}
& \|D\tilde{u}_i\|_ {C^0 (B_{5r/4}(z))} + r^{-\sfrac{1}{2}} [D\tilde{u}_i]_{\sfrac{1}{2}, B_{5r/4} (z)} \leq C\,(\bE(T\res \bB_{7/8}, \pi_0, \bC_{2r} (z,\pi_0)) + \mathbf{A}^2)^{\sfrac{1}{2}}\, \label{eq:C12est-2}
\end{align}
follows from classical elliptic regularity, the reader can for instance see the argument in the proof of \cite[Theorem 6.3]{DHMSS}. 

If we could show that $v$ and $\tilde{v}$ coincide on the domain of definition of $\tilde{v}$, we would be finished. In the remaining argument we will show that we can in fact modify $v$ suitably so to coincide with the map $\tilde{v}$ for all choices of $z$ and $r$, while retaining all the estimates that $v$ satisfies (of course with some larger geometric constants). 

Define the set $\mathcal{P}$ of pairs $(q,r)$ satisfying the assumption of the Proposition, and denote by $\tilde{u}^{q,r}$ and $\tilde{v}^{q,r} = \tilde{u}^{q,r}+\Psi(\cdot + \tilde{u}^{q,r})$ the corresponding maps which we just found. We wish to redefine the map $v$ of Proposition \ref{p:L2Alm} with the following algorithm:
\begin{itemize}
\item[(1)] First of all we restrict $u$ to $K$;
\item[(2)] We then enlarge $K$ by adding $B_{9r/8} (z)$ for every pair $(q,r)\in \mathcal{P}$ (where $q=(z,w)$) and denote by $K^\sharp$ the corresponding set;
\item[(3)] Furthermore we extend $u$ to each such $B_{9r/8} (z)$ by setting it equal to $\tilde{u}^{q,r}$;
\item[(4)] We make a final Lipschitz extension to the whole ball $B_{1/2}$, and then we lift such extension to $\Sigma$ using $\Psi$.
\end{itemize}
We denote by $u^\sharp$ the map defined through the steps (1), (2), and (3), and we set $v^\sharp(\zeta) = u^\sharp(\zeta)+\Psi(\zeta+u^\sharp(\zeta))$. Note that the extension in (3) is well defined because necessarily $\tilde{v}^{q,r} = \tilde{v}^{q',r'}$ on $B_{9r/8} (z) \cap B_{9r'/8} (z')$ whenever the latter is nonempty. This crucial property follows from the fact that over both balls the graphs of the corresponding maps coincide with the restrictions of $T'$ on the corresponding cylinders. 

We next claim that 
\begin{equation}\label{e:Lipschitz-nuova}
{\rm Lip}\, (v^\sharp) \leq C (\bar{\mathbf{E}} + \bA^2)^{\gamma}\, .
\end{equation}
Given that $\|v^\sharp\|_{L^\infty} \leq C (\bar{\mathbf{E}}+ \bA^2)^{\sfrac{1}{2}}$ just because of Lemma \ref{Linfty-L2}, we can use the extension theorem \cite[Corollary 5.3]{DLHMS_linear} to extend $u^\sharp$ to $B_{1/2}$ by enlarging the Lipschitz constant and the $L^\infty$ bound by a constant geometric factor and then lift such extension to $\Sigma$ using $\Psi$. All the remaining conclusions of Proposition \ref{p:L2Alm} will then follow, except for the fact that $K^\sharp$ is not closed. To overcome this issue, we replace $K^\sharp$ with the closed set
\[
K^\star := K \cup \overline{\bigcup_{(q,r)\in\mathcal{P}}B_r(z)}\,.
\]
We observe that all the conclusions of Proposition \ref{p:L2Alm} still hold, since $K \subset K^\star \subset K^\sharp$. The first inclusion is obvious; the second follows from \eqref{e:tamer-cylindrical excess}, which in particular implies that $\overline{\bigcup_{(q,r)\in\mathcal{P}}B_r(z)} \subset \bigcup_{(q,r)\in\mathcal{P}}{B_{9r/8}(z)}$.

We are left with the proof of \eqref{e:Lipschitz-nuova}: we fix $\xi, \zeta\in K^\sharp$ and distinguish several cases. 

\medskip

{\bf Case (a)} $\xi, \zeta \in K$. Then 
\begin{equation}\label{e:evident}
\mathcal{G}_s (v^\sharp (\xi), v^\sharp (\zeta)) 
= \mathcal{G}_s (v (\xi), v (\zeta))
\leq \Lip (v) |\xi-\zeta| \leq C
(\bar{\mathbf{E}} + \bA^2)^\gamma |\xi-\zeta|\, .
\end{equation}

\medskip

{\bf Case (b)} $\xi\in K$, $\zeta\in K^\sharp\setminus K$. Consider then $(q,r)\in \mathcal{P}$ such that 
$\zeta \in B_{9r/8} (z)$. We distinguish further two situations:
\begin{itemize}
\item[(b1)] $\xi \in B_{5r/4} (z)$. Then we obviously have
\begin{align}
\mathcal{G}_s (v^\sharp (\xi), v^\sharp (\zeta)) 
&= \mathcal{G}_s (\tilde{v}^{q,r} (\xi), \tilde{v}^{q,r} (\zeta))
\leq \Lip (\tilde v^{q,r}) |\xi-\zeta|\nonumber\\
&\leq C
(\mathbf{E} (T', \pi_0, \bC_{2r} (z,\pi_0)) + \bA^2)^{1/2} |\xi-\zeta|\, .\label{e:obvious}
\end{align}
Note however that, by \eqref{e:tamer-cylindrical excess}
\[
\mathbf{E} (T', \pi_0, \bC_{2r} (z,\pi_0))
\leq \frac{1}{r^{m+2}} \bar{\mathbf{E}} 
\leq \bar{\mathbf{E}}^{2\gamma}\, .
\]
Therefore we again conclude
\begin{equation}\label{e:obvious-2}
\mathcal{G}_s (v^\sharp (\xi), v^\sharp (\zeta)) 
\leq C
(\bar{\mathbf{E}} + \bA^2)^{\gamma} |\xi-\zeta|\, .
\end{equation}
\item[(b2)] $\xi\not \in B_{5r/4} (z)$. We then select $\xi'\in K \cap B_{9r/8} (z)$. Since $|\xi-\zeta|> \frac{r}{8}$ we certainly have 
\begin{align}
|\xi'-\zeta| &\leq 4r \leq 32 |\xi-\zeta|\, ,\label{e:sciocca1}\\
|\xi-\xi'| &\leq |\xi-\zeta| + |\xi'-\zeta| \leq 33 |\xi-\zeta|\, .\label{e:sciocca2}
\end{align}
We can then use the estimates in cases (a) and (b1) to conclude
\begin{align}
\mathcal{G}_s (v^\sharp (\xi), v^\sharp (\zeta))
\leq \mathcal{G}_s (v^\sharp (\xi), v^\sharp (\xi'))
+ \mathcal{G}_s (v^\sharp (\zeta), v^\sharp (\xi'))
\leq C
(\bar{\mathbf{E}} + \bA^2)^{\gamma} |\xi-\zeta|\, .\label{e:trivial}
\end{align}
\end{itemize}

\medskip

{\bf Case (c)} $\xi, \zeta\in K^\sharp\setminus K$. As above we choose a pair $(q,r)\in \mathcal{P}$ such that $\zeta \in B_{9r/8} (z)$. As in case (b) we distinguish two corresponding cases, which we call (c1) and (c2). In case (c1), namely if $\xi \in B_{5r/4} (z)$, we argue as in case (b1) to conclude \eqref{e:obvious-2}. If instead $\xi\not \in B_{5r/4} (z)$ we then choose $\xi' \in K \cap B_{9r/8} (z)$. The two inequalities \eqref{e:sciocca1} and \eqref{e:sciocca2} are still valid. We can now proceed as in the proof of \eqref{e:trivial}, using, this time, case (b) and case (c1). 
\end{proof}

\section{Proof of Proposition \ref{p:decay-1}} \label{s:decay-1}

In this section we prove the first decay Proposition \ref{p:decay-1}. This will be achieved via a suitable linearization over a plane using the theory of special multivalued functions.

\subsection{Preliminary decay estimate on harmonic multifunctions}

The main reason behind Proposition \ref{p:decay-1} is an analogous decay estimate for Dir-minimizing functions $u$ taking values in $\Iqspec$.

\begin{lemma}\label{l:special-decay}
For every $\delta>0$ there is a constant $\bar r (Q, m, \delta)>0$ with the following property. Let $B_1\subset \mathbb R^m$ and let $u\in W^{1,2} (B_1, \Iqspec)$ be Dir-minimizing and such that $u (0) =Q \a{0}$. Then there is a $1$-homogeneous Dir-minimizing $\bar u\in W^{1,2} (B_1, \Iqspec)$ such that
\begin{equation}\label{e:special-decay}
\frac{1}{r^{m+2}} \int_{B_r} \mathcal{G}_s (u(x), \bar u (x))^2\, dx \leq \delta \int_{B_1} |Du|^2\, 
\qquad \forall r\leq \bar r.
\end{equation}
\end{lemma}

\begin{proof} 
We first claim that it suffices to prove that the lemma holds under the additional assumption that $\etab \circ u \equiv 0$ and with the additional conclusion that $\etab \circ \bar u \equiv 0$. Towards proving the claim, notice first that, if $f \colon B_1 \to \R$ is a harmonic function with $f(0)=0$, then it is elementary that, denoting $\bar f$ the linear term in the Taylor series representing $f$, one has
\begin{equation} \label{e:est_average}
     \frac{1}{r^{m+2}}\,\int_{B_r} |f-\bar f|^2 \, dx \leq C r^2 \int_{B_1} |Df|^2\,.
\end{equation}
Next, for a general $u \in W^{1,2}(B_1, \Iqspec)$ which is $\Dir$-minimizing and such that $u(0)=Q\a{0}$, one denotes $f := \etab \circ u$ and $v:= u \ominus f$; if $\bar v$ denotes the $1$-homogeneous and $\Dir$-minimizing function with zero average obtained applying the lemma to $v$, and $\bar f$ denotes the linear term in the Taylor series of $f$ then, setting $\bar u := \bar v \oplus \bar f$ one has, for all $r \leq \bar r (Q,m,\delta)$
\[
\begin{split}
    \frac{1}{r^{m+2}} \int_{B_r} \mathcal{G}_s(u,\bar u)^2 &= \frac{1}{r^{m+2}} \int_{B_r} \mathcal{G}_s(v,\bar v)^2 + \frac{Q}{r^{m+2}} \int_{B_r} |f-\bar f|^2 \\ 
& \leq \delta \int_{B_1} |Dv|^2 + C\, Q\, r^2 \int_{B_1} |Df|^2 \\
&\leq \delta \int_{B_1} (|Dv|^2 + Q |Df|^2) \\
&=\delta \int_{B_1} |Du|^2\,,
\end{split}
\]
where the second to last inequality holds up to choosing a possibly smaller value for $\bar r$ so that $C r^2 \leq \delta$. 

Next, we prove the validity of the lemma under the additional assumption that $\etab \circ u \equiv 0$ and with the additional conclusion that $\etab \circ \bar u \equiv 0$. We denote by $\mathscr{I}_1 \subset W^{1,2}(B_1,\Iqspec)$ the space of $1$-homogeneous Dir-minimizing functions with zero average.
We argue by contradiction and assume therefore that for every choice of $\bar {r} = \frac{1}{k}$ there is a Dir-minimizing function $u_k\in W^{1,2} (B_1, \Iqspec)$ such that $u_k (0) = Q \a{0}$, $\etab \circ u_k \equiv 0$, and for which there is a positive radius $r_k < \frac{1}{k}$ such that
\begin{equation}\label{e:special-decay-contradiction}
\inf_{\bar u \in \mathscr{I}_1}\frac{1}{r_k^{m+2}} \int_{B_{r_k}} \mathcal{G}_s (u_k(x), \bar u (x))^2\, dx \geq \delta \int_{B_1} |Du_k|^2\, .
\end{equation}
By rescaling we can, w.l.o.g., assume that $\int_{B_1} |Du_k|^2=1$ and thus up to subsequences we can assume that $u_k$ converges to some Dir-minimizing $u\in W^{1,2} (B_1, \Iqspec)$, while statement \eqref{e:special-decay-contradiction} becomes
\begin{equation}\label{e:contradiction-2}
\inf_{\bar u \in \mathscr{I}_1}\frac{1}{r_k^{m+2}} \int_{B_{r_k}} \mathcal{G}_s (u_k(x), \bar u (x))^2\, dx \geq \delta\, .
\end{equation}
Clearly
\[
\int_{B_1} |Du|^2\leq 1\, .
\]
Moreover, by \cite[Theorem 3.1]{DHMSS_final}, $u_k$ is equilipschitz on each compact subset of $B_1$, and thus the convergence is uniform. In particular $u(0) = Q \a{0}$ and $\etab \circ u \equiv 0$. Recall next that the convergence is strong in $W^{1,2} (B_r)$ for every $r<1$ (see \cite{DLHMS_linear}). Therefore $\int_{B_1} |Du|^2>0$. Otherwise we would have
\begin{equation}\label{e:impossible-1}
\lim_{k\to\infty} \int_{B_{3/4}} |Du_k|^2 = 0\, .
\end{equation}
Combined with the Lipschitz estimate of \cite[Theorem 3.1]{DHMSS_final}, the latter would imply
\[
\lim_{k\to\infty} \|Du_k\|_{L^\infty (B_{1/2})} = 0\, .
\]
Observe however that, given the information $u_k (0) = Q \a{0}$, from this we would easily infer 
\[
\lim_{k\to\infty} \frac{1}{r_k^{m+2}} \int_{B_{r_k}} \mathcal{G}_s (u_k (x), Q\a{0})^2 \, dx \leq C (m)
\lim_{k\to\infty} \|Du_k\|^2_{L^\infty (B_{1/2})} = 0\, ,
\]
which is incompatible with \eqref{e:contradiction-2} because the function identically equal to $Q\a{0}$ is certainly $1$-homogeneous and with zero average. We thus conclude 
\begin{equation}\label{e:conclusion-1}
\eta:= \int_{B_1} |Du|^2 > 0\, .
\end{equation}
Consider next the frequency $I = I_{u, 0} (0)$. By \cite[Theorem 3.6]{DHMSS_final} we know that $I$ is a positive integer. If $I\geq 2$, it then follows from the monotonicity of the frequency function that
\[
\int_{B_r} |Du|^2 \leq M_0 r^{m+2} \qquad \mbox{for every $r>0$}\, .
\]
In particular, for any fixed positive $\bar r$ there would be $K:= K(\bar r)\in \mathbb N$ such that
\[
\int_{B_{2 \bar r}} |Du_k|^2 \leq 2^{m+3} M_0 \bar r^{m+2}\, \qquad \forall k\ge K\,.
\]
By the Lipschitz estimate of \cite[Theorem 3.1]{DHMSS_final} we then conclude 
\[
\|Du_k\|_{L^\infty (B_{\bar r})} \leq C M_0^{\sfrac12} \bar r \qquad \forall k\geq K
\]
for a geometric constant $C$. In particular, again using $u_k (0) = Q \a{0}$ we get 
\[
\frac{1}{r^{m+2}} \int_{B_r} \mathcal{G}_s (u_k (x), Q \a{0})^2\, dx \leq C M_0 \bar r^2
\qquad \forall k\geq K, \forall r < \bar r\, .
\]
We thus choose first $\bar r$ sufficiently small so that $C M_0 \bar r^2 \leq \frac{\delta}{2}$ and $k$ sufficiently large so that $k\geq K$ and $\frac{1}{k} < \bar r$. Again our conclusion would be in contrast with \eqref{e:contradiction-2} for all such $k$'s and we thus conclude that the frequency $I$ cannot be larger than $1$. It must therefore be $1$. From this, we can draw the following conclusions:
\begin{align}
&\lim_{r\downarrow 0} \frac{r D_{u,0} (r)}{H_{u,0} (r)} = 1 \label{e:freq=1}  \\
&\lim_{r\downarrow 0} \frac{H_{u,0} (r)}{r^{m+1}} =: \gamma >0 \label{e:non_degener}\\
&\lim_{r\downarrow 0} \frac{D_{u,0} (2r)}{D_{u,0} (r)} = 2^m \label{e:energy_doubling}\\
&\lim_{r\downarrow 0} W_{u,0} (r) := \lim_{r\downarrow 0} \left( r^{-m} D_{u,0} (r) - r^{-m-1} H_{u,0} (r)\right) = 0\, . \label{e:Weiss}
\end{align}

In particular, consider the threshold $\bar \varepsilon>0$ given by \cite[Proposition 7.1]{DHMSS_final} for the choice $C= 2^{m+2}$. Next choose $\bar r$ sufficiently small so that
\begin{align}
H_{u,0} (\bar r) &\geq \frac{\gamma}{2} \bar r^{m+1}\\ 
\bar r D_{u,0} (2\bar r) &\leq 2^{m+1} H_{u,0} (\bar r)\\
 \bar r^{m+1} W_{u,0} (\bar r) &\leq \frac{\bar \varepsilon}{2} H_{u,0} (\bar r)\, .
\end{align}
For any sufficiently large $k$ we then have 
\begin{align}
\bar r D_{u_k,0} (2\bar r) &\leq 2^{m+2} H_{u_k,0} (\bar r)\\
\bar r^{m+1} W_{u_k,0} (\bar r) &\leq \bar \varepsilon  H_{u_k,0} (\bar r)\, .
\end{align}
We then can apply \cite[Proposition 7.1]{DHMSS_final} to each rescaled function $v_k (x) := u_k (\bar r x)$. We thus conclude the existence of a constant $\beta$ (which is geometric) and a constant $\bar C$ (which depends on $\bar r$) such that there exists a $1$-homogeneous Dir-minimizing function $\bar u_k$ with $\etab \circ \bar u_k \equiv 0$ (although this property is not claimed in the statement of \cite[Proposition 7.1]{DHMSS_final}, it can be easily concluded by a rapid inspection of the proof) and with the property that
\[
\|\mathcal{G}_s (u_k, \bar u_k)\|_{C^0 (B_r)} \leq \bar C r^{1+\beta} \qquad \forall r\leq \bar r\, .
\]
In particular, for every $k$ sufficiently large, we would infer 
\[
\int_{B_r} \mathcal{G}_s (u_k (x), \bar u_k (x))^2 \, dx \leq \bar C r^{m+2+2\beta} \qquad \forall r\leq \bar r\, .
\]
Choosing $k$ large enough we also ensure $r_k \leq \frac{1}{k} \leq \bar r$ and we thus can write 
\begin{equation}\label{e:contradiction-3}
\frac{1}{r_k^{m+2}} \int_{B_{r_k}} \mathcal{G}_s (u_k (x), \bar u_k (x))^2 \, dx \leq \bar C r_k^{2\beta}\, .
\end{equation}
Since $r_k\downarrow 0$, for $k$ large enough we have $\bar C r_k^{2\beta} \leq \frac{\delta}{2}$. In particular, given that $\bar u_k\in \mathscr{I}_1$, \eqref{e:contradiction-2} and \eqref{e:contradiction-3} are incompatible. 
\end{proof}

\subsection{Proof of Proposition \ref{p:decay-1}} 

Let $\delta_1>0$ be given and fix a small constant 
$$
r_1 (\delta_1, m,n, Q)< \frac{1}{2}
$$ 
whose choice will be specified later. We will argue by contradiction and assume that, for the choice of $\varepsilon_1 = \frac{1}{k}$, there is a current $T_k$ which satisfies the assumptions of the Proposition but for which \eqref{e:decay-1} fails. We set $\bar \bE_k := \bar \bE (T_k, 0,1)$, denote by $\Sigma_k$ the corresponding Riemannian manifolds and let $\bA_k$ be the $L^\infty$ norms of their second fundamental forms. We can further assume to rotate the currents and the ambient manifolds $\Sigma_k$ so that $\mathbb R^m \times \{0\} = \pi_0 \subset T_0 \Sigma_k = \mathbb R^{m+1}\times \{0\}$ is a plane minimizing the excess $\bar \bE_k$. We then let $v_k$ be the Lipschitz approximation of $T'_k := T_k \res \bC_{1/2}(0,\pi_0)$ given by Proposition \ref{p:L2Alm}, and $\bar v_k := \bar\bE_k^{-1/2} v_k$ their normalizations. By \eqref{e:W12} we conclude that 
\begin{equation}
\int_{B_{1/2}} (|\bar v_k|^2 + |D\bar v_k|^2)\leq C\, .
\end{equation}
We can therefore appeal to the extension of the classical Sobolev space theory to $\mathscr{A}_Q (\mathbb R^n)$-valued maps to conclude that $\bar v_k$ converges to a map $v\in W^{1,2} (B_{1/2}, \mathscr{A}_Q (\mathbb R^n))$ strongly in $L^2$, up to extraction of a  subsequence (not relabeled). Moreover we observe that:
\begin{itemize}
    \item[(a)] Since $\bA_k \to 0$, $\Sigma_k$ converge to $\mathbb R^{m+1}\times \{0\}$ and thus, by \eqref{e:inclusion}, $v$ takes values in $\pi_0^\perp \cap (\mathbb R^{m+1}\times \{0\})$, i.e. it can be regarded as a $\mathscr{A}_Q (\mathbb R)$-valued map;
    \item[(b)] by \cite[Theorem 13.3]{DLHMS} $v$ is ${\rm Dir}$-minimizing;
    \item[(c)] since $\Theta (T'_k, 0) = \Theta (T_k, 0) \geq Q$, by \cite[Theorem 23.1]{DLHMS} we have that
    \[
    \lim_{s\downarrow 0} \frac{1}{s^m} \int_{B_s} \mathcal{G}_s (v (y), Q \a{\etab \circ v (0)})^2 = 0\, ;
    \]
    in fact, the validity of \cite[Theorem 23.1]{DLHMS} is stated under the assumption that $\Theta(T_k',0)=Q$, but an inspection of the proof (which is as in \cite[Proof of Theorem 2.7]{DLS_Lp}) shows that the same result also holds when $\Theta(T_k',0) \geq Q$ (precisely, the condition on the density is crucial in \cite[Formula (9.9)]{DLS_Lp}, and the latter inequality is valid also in our setting). Furthermore, one sees that, since the origin is a point of density  at least $Q$ for $T'_k$, $v(0)=Q\a{\etab \circ v (0)} = Q\a{0}$.
\end{itemize}
 We are then in a position to apply Lemma \ref{l:special-decay} and conclude that there are $\bar r>0$ and a $1$-homogeneous Dir-minimizing function $h$ such that 
\begin{equation}\label{e:decaduto-1}
\int_{B_r} \mathcal{G}_s (v, h)^2 \leq \frac{\delta_1}{4 Q} r^{m+2} \qquad \forall r \leq \bar r\, .
\end{equation}
Our choice of $r_1$ is then given by the above $\bar r$. 

We next consider the rescaled functions $h_k := \bar\bE_k^{1/2} h$ and observe that the supports of their graphs are open books $\bS_k$ which belong to $\mathscr{B} (0)$. In particular we must have
\[
\mathbb{E} (T_k, 0, r_1) \leq \frac{1}{ r_1^{m+2}} \left( \int_{\bB_{r_1}} \dist^2 (x, \bS_k)\, d\|T_k\|  + \int_{\bS_k \cap (\bB_{r_1} \setminus B_{r_1/8}(V))} \dist^2(x,\spt(T_k)) \, d\Ha^m \right) \, .
\]
We now claim that, for a sufficiently large $k$,
\begin{equation}\label{e:decaduto-2}
\limsup_{k\to \infty} \frac{1}{\bar \bE_k r_1^{m+2}} \left( \int_{\bB_{r_1}} \dist^2 (x, \bS_k)\, d\|T_k\| + \int_{\bS_k \cap (\bB_{r_1} \setminus B_{r_1/8}(V))} \dist^2(x,\spt(T_k)) \, d\Ha^m \right) \leq \frac{\delta_1}{2}\, ,
\end{equation}
and the latter will give a contradiction since we were assuming $\mathbb{E} (T_k, 0, r_1) > \delta_1 \bar \bE_k$ for every $k$.

Observe first that $\dist (x, \bS_k) \leq r_1$ for every $x\in \bB_{r_1}$ and we can therefore estimate
\begin{align*}
\int_{\bB_{r_1}} \dist^2 (x, \bS_k) \,d\|T_k\| (x) &\leq
\int_{\bB_{r_1}} \dist^2 (x, \bS_k)\, d\|\bG_{v_k}\| (x) + C \|T_k\| ((B_{r_1}\setminus K_k)\times \pi_0^\perp)\\
&\leq \int_{\bB_{r_1}} \dist^2 (x, \bS_k)\, d\|\bG_{v_k}\| (x) + C \bar \bE_k^{1+\gamma}\, .
\end{align*}
Moreover we have
\begin{align*}
    & \int_{\bS_k \cap (\bB_{r_1} \setminus B_{r_1/8}(V))} \dist^2(x,\spt(T_k)) \, d\Ha^m\\
    \le & \int_{\bB_{r_1}} \dist^2(x,\spt(\bG_{v_k}\res \bC_{K_k})) \, d\norm{\bG_{h_k}\res \bC_{K_k}} + C \bar \bE_k^{1+\gamma}\, ,
\end{align*}
so that
\begin{equation}\label{e:decaduto-3}
\begin{split}
&\quad\limsup_{k\to \infty} \frac{1}{\bar \bE_k r_1^{m+2}} \left( \int_{\bB_{r_1}} \dist^2 (x, \bS_k) d\|T_k\| (x) + \int_{\bS_k \cap (\bB_{r_1} \setminus B_{r_1/8}(V))} \dist^2(x,\spt(T_k)) \, d\Ha^m\right) \\& \leq
\limsup_{k\to \infty} \frac{1}{\bar \bE_k r_1^{m+2}} \Bigg( \int_{\bB_{r_1}} \dist^2 (x, \bS_k) d\|\bG_{v_k}\| (x)\\
&\qquad\qquad\qquad\qquad\qquad + \int_{\bB_{r_1} } \dist^2(x,\spt(\bG_{v_k}\res \bC_{K_k})) \, d\norm{\bG_{h_k}\res \bC_{K_k}} \Bigg)\, .
\end{split}
\end{equation}
Consider next that $\bB_{r_1} \subset \bC_{r_1}$ and that, since the Lipschitz constants of $v_k$ and $h_k$ converge to $0$, we conclude
\begin{align}
&\limsup_{k\to \infty} \frac{1}{\bar \bE_k r_1^{m+2}} \int_{\bB_{r_1}} \dist^2 (x, \bS_k)\, d\|\bG_{v_k}\| (x)\nonumber\\
\leq &\limsup_{k\to \infty}\frac{1}{\bar \bE_k r_1^{m+2}} \int_{B_{r_1}} \sum_{i=1}^Q \dist^2 ((y, (v_k)_i (y)), \bS_k)\, dy\, ,
\label{e:decaduto-4}
\end{align}
as well as
\begin{align}
&\limsup_{k\to \infty} \frac{1}{\bar \bE_k r_1^{m+2}} \int_{\bB_{r_1}} \dist^2 (x, \spt(\bG_{v_k}\res \bC_{K_k}))\, d\|\bG_{h_k}\res \bC_{K_k}\| (x)\nonumber\\
\leq &\limsup_{k\to \infty}\frac{1}{\bar \bE_k r_1^{m+2}} \int_{B_{r_1}\cap K_k} \sum_{i=1}^Q \dist^2 ((y, (h_k)_i (y)), \spt(\bG_{v_k}\res \bC_{K_k}))\, dy\, ,
\label{e:deceduto-4}
\end{align}
where $(\sum_i \a{(v_k)_i (y)}, \epsilon_{v_k} (y))$ is the value of the $\mathscr{A}_Q (\mathbb R^n)$-valued function $v_k$ at $y$ and $(\sum_i \a{(h_k)_i (y)}, \epsilon_{h_k} (y))$ is the value of the $\mathscr{A}_Q (\mathbb R^n)$-valued function $h_k$ at $y$.

Now, observe that $h_k=\bar\bE_k^{1/2} h$, and that $\bS_k$ is its support. Thus, if we denote by $(\sum_i \a{h_i (y)}, \epsilon (y))$ the value of $h$ at $y$, then $\eps_{h_k}=\eps$, $(h_k)_i=\bar \bE_k^{1/2} h_i$, and for every $y$, $k$ and $i$ there is a $j=j (k, y,i)$ with the property that
\[
|(v_k)_i (y) - \bar \bE_k^{1/2} h_j (y)|\leq \mathcal{G}_s (v_k (y), \bar \bE_k^{1/2} h (y))\, . 
\]
Since $\dist (\left(y, (v_k)_i (y)\right), \bS_k) \leq |(v_k)_i (y) - \bar \bE_k^{1/2} h_j (y)|$, we can write 
\begin{align}
&\limsup_{k\to \infty}\frac{1}{\bar \bE_k r_1^{m+2}} \int_{B_{r_1}} \sum_{i=1}^Q \dist^2 ((y, (v_k)_i (y)), \bS_k)\, dy\nonumber\\
\leq & \limsup_{k\to \infty} \frac{1}{\bar \bE_k r_1^{m+2}} \int_{B_{r_1}} Q\, \mathcal{G}_s (v_k (y), \bar \bE_k^{1/2} h (y))^2\, dy\nonumber\\
= & \limsup_{k\to \infty} \frac{Q}{r_1^{m+2}} \int_{B_{r_1}} \mathcal{G}_s (\bar v_k (y), h (y))^2\, dy\nonumber\\
=& \frac{Q}{r_1^{m+2}} \int_{B_{r_1}} \mathcal{G}_s (v (y), h (y))^2\, dy
\stackrel{\eqref{e:decaduto-1}}{\leq} \frac{\delta_1}{4}\, . \label{e:decaduto-5}
\end{align}

Arguing analogously, one sees that  $\dist(\left( y, (h_k)_i(y)\right),\spt(\bG_{v_k})) \leq \mathcal{G}_s(\bar\bE_k^{1/2}h(y), v_k(y))$ for every $y \in B_{r_1}\cap K_k$, so that
\begin{align} \label{e:deceduto-5}
    &\limsup_{k\to \infty}\frac{1}{\bar \bE_k r_1^{m+2}} \int_{B_{r_1}\cap K_k} \sum_{i=1}^Q \dist^2 ((y, (h_k)_i (y)), \spt(\bG_{v_k} \res \bC_{K_k}))\, dy
        \leq \frac{\delta_1}{4}\,.
\end{align}
Combining \eqref{e:decaduto-5} and \eqref{e:deceduto-5} with \eqref{e:decaduto-4} and \eqref{e:deceduto-4}, and plugging in \eqref{e:decaduto-3}, we conclude \eqref{e:decaduto-2}, thus completing the proof.
\qed

\section{Proof of Proposition \texorpdfstring{\ref{p:decay-2}}{p:decay2}: propagation lemmas and behavior of \texorpdfstring{$Q$}{Q}-points} \label{s:Q-points}

Many ingredients in the proof of Proposition \ref{p:decay-2} will be borrowed from \cite[Theorem 4.5]{DHMSS}. However, several substantial changes are needed, mostly because the ``optimal open book'' $\bS$ in \cite[Theorem 4.5]{DHMSS} is assumed to be at a fixed distance from a plane, while the one in Proposition \ref{p:decay-2} is not. The first such change is  related to the construction of the graphical parametrization, where we cannot rely solely on White's $\eps$-regularity theorem \cite{White86}, but we will also need to use Proposition \ref{p:L2Alm}. 

Over the next sections we will work under the following set of assumptions.

\begin{ipotesi}\label{ass:decay plane -1}
We let $T$ and $\Sigma$ be as in Assumption \ref{ass:everywhere}, and $\bar \eps, \bar \eta \in \left(0,\frac12\right)$ are fixed positive constants. There are an open book $\bS=\bigcup_{i=1}^{N} \bH_i \subset T_0\Sigma$ and and a plane $\pi_0 \subset T_0\Sigma$ such that
\begin{itemize}
    \item[(i)] $\bA^2 \leq \bar\eps\, \mathbf E (T, \pi_0, 0, 1)\leq \bar\varepsilon^2$; 
    \item[(ii)] $\mathbb{E} (T, \bS, 0,1) \leq \bar \eta \, \bE (T, \pi_0, 0, 1)$;
    \item[(iii)] $\bar \bE (T, 0, 1) \geq (1-\bar \eta) \, \bE (T, \pi_0, 0, 1)$.
\end{itemize}
\end{ipotesi}

Setting $V=V(\bS)$, we write $T_0\Sigma = V^{\perp_0} \oplus V$ with coordinates 
\[
z=(x,y)=(x_1,x_2,y_1,\ldots,y_{m-1})\, .
\]
If additionally $V= V (\bS) \subset \pi_0$, then we set 
\[
\pi_0 = \{x_2=0\}
\qquad \text{and}\qquad 
\pi_0^\pm := \pi_0 \cap \{ \pm x_1 > 0  \}\,.
\]
Coordinates in $T_0 \Sigma^\perp \simeq \R^{n-1}$ are denoted $w$.

\subsection{Angle bound} We start with a lemma bounding the angles formed between the various pages of $\mathbf{S}$ and $\pi_0$.

\begin{definition}\label{def:angles}
For every fixed $q\in T_0\Sigma$, let $\mathscr{H}_q = \left\lbrace \bH \subset \bS \, \colon \, \dist(q,\bH) = \dist(q,\bS) \right\rbrace$, and, for any $m$-dimensional plane $\pi \subset T_0\Sigma$, denote by $\beta_\pi(q)$ the maximal angle between half-planes $\bH \in \mathscr{H}_q$ and $\pi$. More precisely, for any $\bH \in \mathscr{H}_q$ we let $\tau(\bH)$ denote the $m$-plane containing $\bH$, and then we set
\[
\beta_\pi(q) := \max\{\dist_{\mathcal H}(\tau(\bH) \cap \bB_1, \pi \cap \bB_1)\,\colon\, \bH \in \mathscr{H}_q\}\,,
\]
where $d_{\mathcal{H}}$ denotes Hausdorff distance. We record the elementary fact that, when $\tau$ and $\pi$ are two $m$-planes in $T_0\Sigma$,
\[
\dist_{\mathcal H}(\tau \cap \bB_1, \pi \cap \bB_1) = \sup\{|\mathbf{p}_{\pi^\perp}(z)| \, \colon \, z \in \tau \cap \bB_1\}\,.
\]
We also set
\[
\begin{split}
\beta_{\pi}(\bS):&=\max_{q\in \partial \bB_1\cap \bS} \beta_{\pi} (q) \\ &= \max\left\lbrace \dist_{\mathcal H}(\pi \cap \bB_1, \pi_i \cap \bB_1) \, \colon \, \pi_i \supset \bH_i\,, i \in \{1,\ldots,N\} \right\rbrace \,, \\
\beta_{\max}(\bS):&=\max\left\lbrace \max_{q\in \partial \bB_1\cap \bS} \beta_{\pi} (q) \, \colon \, \pi \supset \bH_i \quad \mbox{for some $i \in \{1,\ldots,N\}$} \right\rbrace \\
&= \max\left\lbrace \dist_{\mathcal H}(\pi_i \cap \bB_1, \pi_j \cap \bB_1) \, \colon \, \pi_i \supset \bH_i\,, \pi_j \supset \bH_j\,,i,j \in \{1,\ldots,N\} \right\rbrace\, .
\end{split}
\]
\end{definition}

% we let $\bH_{q}$ be a page of $\bS$ such that $\dist(q,\bH_{q})=\dist(q,\bS)$. For every $m$-dimensional plane $\pi \subset T_0\Sigma$, let $\beta_\pi(q)$ denote the angle between $\bH_{q}$ and $\pi$: more precisely, if $\tau$ is the $m$-dimensional plane containing $\bH_q$, then $\beta_\pi (q) := |\mathbf{p}_\tau - \mathbf{p}_\pi|$. We then set 

$\beta_{\pi} (\bS)$ is the maximal angle formed by pages of $\bS$ and $\pi$, while $\beta_{\max} (\bS)$ is the maximal angle formed by distinct pages of $\bS$. Observe that $\beta_{\max} (\bS)=0$ if and only if $\bS$ is an $m$-dimensional plane.

\begin{lemma}[Angle bound]\label{l:angle-bound}
There are positive constants $\varepsilon_4$, $\eta_4$, and $C$ such that, if $T$ satisfies Assumption \ref{ass:decay plane -1} with $\bar\eps\leq \eps_4$ and $\bar\eta\leq \eta_4$, then the following holds:
    \begin{equation}\label{eq:angle_bound}
         C^{-1} \beta_{\pi_0}^2(\bS) \leq  \bE(T,\pi_0,0,1) \leq C \beta_{\max}^2 (\bS) 
         \leq C^2 \beta_{\pi_0}^2 (\bS)\, .
    \end{equation}
    The same conclusions hold if the first inequality in Assumption \ref{ass:decay plane -1}(i) is replaced by $\bA \leq \bar\eta\beta_{\max} (\bS)$.
\end{lemma}

\begin{proof} 
We first observe that, under the hypotheses (ii) and (iii) of Assumption \ref{ass:decay plane -1}, the book $\bS$ cannot be an $m$-dimensional plane, that is $\beta_{\max}(\bS) > 0$. Indeed, should $\bS$ be a plane, we would have, for $\bar\eta < \frac12$, the contradiction 
\[
\begin{split}
\bE(T,\pi_0,0,1) &\leq \frac{1}{1-\bar\eta}\, \bar\bE(T,0,1) \leq \frac{1}{1-\bar\eta}\, \bE(T,\bS,0,1) \leq \frac{\bar\eta}{1-\bar\eta}\, \bE(T,\pi_0,0,1) \\
& < \bE(T,\pi_0,0,1)\,.
\end{split}
\]

Next, we prove that there exists a positive geometric constant $C$ such that
\begin{equation}\label{eq:angle_bound1}
    \frac1C\,  \bE(T,\pi_0,0,1)\leq \beta_{\max}^2(\bS)\,.
\end{equation}
Indeed, for $q \in \Sigma \cap \bB_1$ set $q' := \p_{T_0\Sigma}(q)$ , let $\bH_{q'} \in \mathscr{H}_{q'}$ be such that $\dist_{\mathcal{H}}(\tau(\bH_{q'}) \cap \bB_1, \pi \cap \bB_1) = \beta_\pi (q')$, and let $q''\in \bH_{q'}$ be such that $\dist (q', \bH_{q'})= \dist(q',\bS) =  |q'-q''|$. Notice that, for $\pi$ an arbitrary plane in $T_0\Sigma$, it holds
\[
\abs{\mathbf{p_{\pi^\perp}}(q)}\le \dist(q', \bH_{q'}) + \beta_{\pi} (q'') + \bA\,.
\]
Therefore we have that
\begin{align*}
    &\int_{\bB_1} \dist^2(q,\pi) \,d\|T\| 
     \le 2 \int_{\bB_1} \dist^2(q,\bS) \,d\|T\|(q) + 2\int_{\bB_1} \beta_{\pi} (q'')^2  \, d\|T\|(q) + C\bA^2\,\\
    \leq & C\,\bE(T,\bS,0,1)+ C \bar\varepsilon \bE (T, \pi_0, 0, 1) + C \bar\eta^2 \beta_{\max} (\bS)^2 + 2\int_{\bB_1} \beta_\pi(q'')^2  \, d\|T\|(q)\\
    \leq & C (\bar\eta + \bar\eps) \bE (T, \pi_0, 0, 1) + C \bar\eta \beta_{\max} (\bS)^2 + 2\int_{\bB_1} \beta_{\pi} (q'')^2 \, d\|T\| (q)\, ,
\end{align*}
where in the second inequality we have used Assumption \ref{ass:decay plane -1}(i) (or the alternative $\bA \leq \bar\eta \beta_{\max} (\bS)$) and in the last inequality we used Assumption \ref{ass:decay plane -1} (ii). The above implies
\[
\bar\bE(T,0,1) \leq \bE (T, \pi, 0, 1) \leq C (\bar\eta+\bar\eps) \bE (T, \pi_0, 0, 1)
+ C \bar\eta \beta_{\max} (\bS)^2 + 2 \int_{\bB_1} \beta_{\pi} (q'')^2 \, d\|T\| (q)\, ,
\]
which coupled with Assumption \ref{ass:decay plane -1}(iii) yields, for $\bar\eta$ and $\bar\eps$ sufficiently small,
\begin{equation}\label{eq.integral lower bound}
\bE(T,\pi_0,0,1) \le 4\,\int_{\bB_1} \beta_\pi(q'')^2\, d\|T\| (q) + C \bar\eta \beta_{\max} (\bS)^2\,.
\end{equation}
We fix next $\pi$ to be some $m$-dimensional plane containing a page $\bH_i$ of $\bS$, so that $\beta_{\pi}(q'') \leq \beta_{\max}(\bS)$ for $\|T\|$-a.e. $q$. We then achieve
\begin{align}\label{eq.lower bound on opening angle}
    \bE(T,\pi_0,0,1)\leq C\,\beta_{\max}^2(\bS)\,.
\end{align}

\medskip

The inequality $\beta_{\max} (\bS) \leq 2 \beta_{\pi_0} (\bS)$ immediately follows from the triangle inequality for the Hausdorff distance.

\medskip

We next claim that there exists a positive geometric constant $C$ such that
\begin{equation}\label{eq:angle_bound2}
    \beta_{\pi_0}^2(\bS) \leq  C\,  \bE(T,\pi_0,0,1)\,.
\end{equation}
We assume by contradiction that for every $C_1>0$ one could have
\[
\beta_{\pi_0}^2(\bS)> C_1\, \bE(T,\pi_0,0,1) \,,
\] 
and we let $\bH_0$ be a half plane realizing $\beta_{\pi_0}(\bS)$, namely $\beta_{\pi_0}(\bS) = \dist_{\mathcal H}(\tau(\bH_0) \cap \bB_1, \pi_0 \cap \bB_1)$. Then, let $V = V(\bS)$ be the spine of $\bS$, set $W:= \pi_0 \cap \bH_0$, and define
\[
    \Omega := \left\lbrace z \in \bH_0 \, \colon \, |z| \leq \frac14,\, \dist(z,V) \geq \frac18,\, \dist(z,W) \geq \frac18 \right\rbrace\,.
\]
Notice that $\mathcal{H}^m(\Omega) \geq c_m$ for some positive geometric constant $c_m$. Notice also that 
\[
\abs{\mathbf{p}_{\pi_0^\perp}(z)}^2 \geq 8^{-2} \dist^2(z,W) \, \beta_{\pi_0} (\mathbf{S})^2 \qquad \mbox{for every $z \in \Omega$}
\]
 Next recall Lemma \ref{Linfty-L2}: 
\begin{align*}
\sup_{\tilde q\in \bB_{15/16}\cap\spt(T)}|\mathbf{p}_{\pi_0^\perp}(\tilde q)|^2 & \leq C_0 \,\bE(T,\pi_0,0,1) + C_0 \bA^2 \\
&\leq C_0 \,\bE(T,\pi_0,0,1) + C_0 \bar\eps \bE (T, \pi_0, 0, 1) + \bar\eta \beta_{\pi_0} (\bS)^2\, ,
\end{align*}
where we have used either Assumption \ref{ass:decay plane -1}(i) or the alternative $\bA^2 \leq \bar\eta^2 \beta_{\max} (\bS)^2 \leq \bar\eta \beta_{\pi_0} (\bS)^2$.
It follows then that for every $z \in \Omega$
\[
\begin{split}
& 2\,\dist(z,\spt(T))^2 \\
&\qquad \qquad\ge 8^{-2} \dist^2(z,W) \beta_{\pi_0} (\mathbf{S})^2 - 2 \sup_{\tilde q\in \bB_{15/16}\cap\spt(T)}|\mathbf{p}_{\pi_0^\perp}(\tilde q)|^2 \\ &\qquad \qquad\geq 
 ((8^{-4}-\bar\eta)C_1-3C_0) \, \bE(T,\pi_0,0,1)\, .
 \end{split}
\]
We then infer that
\[
\bE(\bS,T,0,1) \ge \bE(\bH_0, T,0,1) \ge \frac{((8^{-4}-\bar\eta)C_1-3C_0)}{2} \, c_m\,\bE(T,\pi_0,0,1) \,, 
\]
which, for $\bar\eta$ sufficiently small and $C_1>0$ sufficiently large, is a contradiction with Assumption \ref{ass:decay plane -1} (ii).
\end{proof}

\begin{remark}\label{rmk:on_angle_bound}
A quick inspection of the proof of Lemma \ref{l:angle-bound} shows that in order to prove the inequalities
\[
\bE(T,\pi_0,0,1) \leq C\,\beta_{\max}^2(\bS) \leq C^2\,\beta_{\pi_0}^2(\bS)
\]
only the smallness condition $\bE(T,\bS,0,1) \leq \bar\eta \, \bE(T,\pi_0,0,1)$ on the one-sided excess is needed. The smallness condition in Assumption \ref{ass:decay plane -1}(ii) for the double-sided excess, which involves also a condition on $\bE(\bS,T,0,1)$, is only needed to prove the other bound
\[
\beta_{\pi_0}^2(\bS) \leq C \, \bE(T,\pi_0,0,1)\,.
\]
\end{remark}

\subsection{Propagation of graphicality}

The following lemmas will be the key to achieve a graphical parametrization of the current over $\bS$.

\begin{lemma}[Kick-off lemma] \label{lem.propagation1}
There exists $\eta_5 > 0$ with the following property. Let $T$ and $\Sigma$ be as in Assumption \ref{ass:everywhere}. Assume that there are an open book $\bS = \bigcup_{i=1}^N \bH_i \subset T_0\Sigma$ and a plane $\pi_0 \subset T_0\Sigma$ such that, for some $\bar\eta \leq \eta_5$ 
\begin{enumerate}
    \item[(a1)] $\bE(T,\pi_0,0,1) \le \bar\eta$ and $\bar \bE (T, 0, 1) \geq (1-\bar\eta) \bE (T, \pi_0, 0, 1)$;
    \item[(a2)] $\mathbb{E} (T,\bS, 0,1) \le \bar\eta\, \bE(T,\pi_0,0,1)$; 
    \item[(a3)] $\bA^2 \leq \bar\eta \, \mathbf{E}(T,\pi_0,0,1)$.
\end{enumerate}
Then, there exists a plane $\pi_0'$ such that the rescaled current $T' = (\lambda_{0,\sfrac12})_\sharp T$ satisfies (a1)-(a2)-(a3) with $\pi_0$ replaced by $\pi_0'$ and $\bar\eta$ replaced by some $\eta=\eta(\bar\eta)$ such that $\eta(\bar\eta) \to 0$ when $\bar\eta \to 0$. Furthermore, $\pi_0'$ satisfies the additional properties that
\begin{itemize}
\item[(a4)] $\bE (T',\pi_0',0,1) \leq 2\,\bar\bE (T',0,\sfrac12)$;
\item[(a5)] $V (\mathbf{S})\subset \pi_0'$.
\end{itemize}
\end{lemma}

\begin{proof}
If we choose $\eta_5 \leq \min\{\varepsilon_4,\eta_4\}$, then under the hypotheses of the lemma we can apply Lemma \ref{l:angle-bound}, so that we have
\begin{equation}\label{e:A^2/E-vanishes-after-all}
C\,\beta_{\max}(\bS)^2 \geq \bE (T, \pi_0, 0,1) \geq \bar\eta^{-1} \bA^2\,,
\end{equation}
and
\begin{equation}\label{e:boundedness}
\beta_{\pi_0} (\bS)^2 \leq C \bE (T, \pi_0, 0, 1)\, .
\end{equation}

Consider now sequences $\{T_k\}_{k=1}^\infty$ of currents and $\{\Sigma_k\}_{k=1}^\infty$ of manifolds satisfying, for open books $\bS_k$ and planes $\pi_k$ in $T_0\Sigma_k$, assumptions (a1)-(a2)-(a3) with parameters $\bar\eta=\eta_k \to 0^+$. Up to rotations, we can assume that each $\Sigma_k$ has the same tangent $T_0 \Sigma_k = \tau_0$ and is the graph of a function $\Psi_k: \tau_0\to \tau_0^\perp$ over a region including all points of interest for the rest of the proof. Upon applying a further rotation, we may also assume that the planes $\pi_k$ coincide with a fixed plane $\pi_0$, and also that the spines $V(\bS_k)$ have the same projection onto $\pi_0$, that is $\mathbf{p}_{\pi_0}(V(\bS_k)) = V'$ for every $k$, where $V'$ is an $(m-1)$-dimensional linear subspace of $\pi_0$. We let $\pi_0^\pm$ denote the two halves of $\pi_0$ delimited by $V'$.

Next, we observe that, for all $k$ sufficiently large, we can apply Proposition \ref{p:L2Alm} \footnote{Notice that the conclusions of Proposition \ref{p:L2Alm} still hold true (with a slightly worse constant) if the second part of hypothesis (b) on the optimality of $\pi_0$ is replaced by the almost-optimality condition (a1).} and guarantee the existence of Lipschitz maps $v=v_k \colon B_{3/4}=B_{3/4}(0,\pi_0) \to \mathscr{A}_Q(\pi_0^\perp)$ and closed sets $K = K_k \subset B_{3/4}$ such that \eqref{e:inclusion}-\eqref{e:W12} hold for $T=T_k$ and $\Sigma=\Sigma_k$. Writing $v_k = \left( \sum_i \a{(v_k)_i}, \eps_{v_k} \right)$ and $\bar\bE_k : = \bE(T_k, \pi_0, 0,1)$, we consider the functions $\bar v_k \colon B_{3/4} \to \mathscr{A}_Q(\pi_0^\perp)$ defined by
\[
\bar v_k := \left( \sum_i \a{\frac{(v_k)_i}{\bar\bE_k^{1/2}}}, \eps_{v_k} \right)\,,
\]
and we let $v$ be a subsequential limit (in the weak topology of $W^{1,2}$ over $B_{3/4}$) of the $\bar v_k$'s. Now consider any linear map $\ell_k: \pi_0 \to \pi_0^{\perp_0}$ whose graph describes (on a suitable half of $\pi_0$) an arbitrary page of the book $\bS_k$. The estimate \eqref{e:boundedness} implies that $\bar\bE_k^{-\sfrac{1}{2}} \ell_k$ is uniformly locally bounded and it thus converges, up to subsequences, to some function $\ell$. By (a2)-(a3), the support of the graph of $v$ coincides with the union of the graphs of all linear functions $\ell$ arising as possible limits as above, after restricting each of them to the appropriate half plane $\pi_0^\pm$ (we shall denote $\ell^\pm$ such restriction): in other words, there are positive integers $N^\pm$ and $\kappa_i^\pm$ such that
\[
v= \left(\sum_{i=1}^{N^\pm} \kappa_i^\pm\a{\ell_i^\pm},\pm1\right) \qquad \mbox{on $\pi_0^\pm$}\, ,
\]
which in fact takes values in $\mathscr{A}_Q (\pi_0^{\perp_0})$. By \eqref{e:A^2/E-vanishes-after-all} it follows easily that $v$ cannot be trivial, i.e. the collections $\{\ell_i^\pm\}$ contains at least three distinct linear maps. The support of the graph of $v$ then coincides with a non-flat open book with $(m-1)$-dimensional spine. Notice that if $\bar\bS_k$ denote the rescaled non-flat open books defined as the support of the graphs of $\bar\bE_k^{\frac12}v$ then there exists a rotation $O_k$ in $\Pi_0$ such that $O_k(V(\bar \bS_k))=V(\bS_k)$ and
\begin{equation}\label{e:spine tilt vs excess}
    \lim_{k \to \infty} \bar\bE_{k}^{-\sfrac12} \|O_k - {\rm Id}\| = 0\,.
\end{equation}

We next observe that, by \cite[Theorem 13.3]{DLHMS}, we have in addition that $\bar v_k$ converge to $v$ strongly in $W^{1,2}$ on $B_{1/2}$, and that $v$ is $\Dir$-minimizing. In particular, the averages 
\[
\ell^{\pm} := \frac{1}{N^\pm} \sum_i \kappa_i^\pm \ell_i^\pm
\]
defined on the respective halfplanes $\pi_0^\pm$ form a single harmonic function $\ell$ over $\pi_0$. 

We next consider the planes $\hat\pi_k$ which are the graphs of $\bar{\mathbf{E}}_k^{\frac{1}{2}} \ell$. Using the estimates of Proposition \ref{p:L2Alm}, the strong $L^2$ convergence of the maps above, and the definition of $\ell=\etab\circ v$, it is easy to see that, upon setting $T'_k := (\lambda_{0,\sfrac12})_\sharp T_k$ and $\mathbf{A}'_k := \frac{1}{2} \mathbf{A}_k$, we have 
\begin{align}
&\lim_{k\to\infty} \left(\mathbf{E} (T'_k, \hat\pi_k, 0,1) +\frac{(\mathbf{A}'_k)^2}{\mathbf{E}(T'_k,\hat\pi_k,0,1)} + \frac{\mathbb{E} (T'_k, \mathbf{S}_k, 0,1)}{\mathbf{E} (T'_k, \hat\pi_k, 0, 1)}\right) = 0 \label{e:quasiNatale1}\\
&\lim_{k\to\infty} \frac{\mathbf{E} (T'_k, \hat\pi_k, 0, 1)}{\bar\bE (T'_k, 0,1)}=\lim_{k \to \infty} \frac{\mathbf{E} (T'_k, \hat\pi_k, 0, 1)}{\bar\bE (T'_k, 0,\sfrac12)}=1\, , \label{e:quasiNatale2} \\
\end{align}
Setting now $\pi_k' := O_k(\hat\pi_k)$, we have that $\pi_k' \supset V(\bS_k)$ and, thanks to \eqref{e:spine tilt vs excess}, the conditions in \eqref{e:quasiNatale1}-\eqref{e:quasiNatale2} remain true with $\pi_k'$ in place of $\hat\pi_k$. This completes the proof.
\end{proof}

The following corollary can be easily proved by iterating Lemma \ref{lem.propagation1} (or by following the same proof).

\begin{corollary} \label{cor:kick-off}
For every $r_0 > 0$ there exists $\eta_6 > 0$ such that if $T$, $\Sigma$, $\bS$, $\pi_0$ are as in Lemma \ref{lem.propagation1} and they satisfy (a1)-(a2)-(a3) with $\bar\eta \leq \eta_6$, then, setting $V=V(\bS)$, for every $y \in \bB_{1/4} \cap V$ and $r_0 \leq r \leq 1/4$ there exists a plane $\hat\pi_{y,r}$ so that the rescaled current $T_{y,r} = (\lambda_{y,r})_\sharp T$ satisfies (a1)-(a2)-(a3)-(a4)-(a5) with $\pi_0$ replaced by $\hat\pi_{y,r}$ and $\bar\eta$ replaced by some $\eta=\eta(\bar\eta)$ such that $\eta(\bar\eta) \to 0$ as $\bar\eta \to 0$.
\end{corollary}

\begin{lemma}[Propagation lemma] \label{l:propagation}
For every $\rho>0$, $0 < r_0$, $\delta_0\in (0,\frac{1}{4})$ there exist $\eta_7 > 0$ and $C > 0$ with the following property. Let $T$ and $\Sigma$ be as in Assumption \ref{ass:everywhere}. Assume that there are an open book $\bS = \bigcup_{i=1}^N \bH_i \subset T_0\Sigma$ and a plane $\pi_0 \subset T_0\Sigma$ such that, for some $\bar\eta \leq \eta_7$ 
\begin{enumerate}
    \item[(b1)] $\bE(T,\pi_0,0,1) \le \bar\eta$ and $2\,\bar \bE (T, 0, 1) \geq \bE (T, \pi_0, 0, 1)$;
    \item[(b2)] $\mathbf{E} (T,\bS, 0,1) \le \bar\eta\, \bE(T,\pi_0,0,1)$; 
    \item[(b3)] $\bA^2 \leq \bar\eta\, \mathbf{E}(T,\pi_0,0,1)$;
    \item[(b4)] $\bE(T,\pi_0,0,1) \leq 2\, \bar\bE(T,0,\sfrac12)$;
    \item[(b5)] $V(\bS) \subset \pi_0$.
\end{enumerate}
Then, the following holds.
\begin{enumerate}
    \item Pushing $Q$-points: writing $V=V(\bS)$,
    \begin{equation}\label{eq.replacement_for_graphical2}
    \Theta(T,q)<Q \text{ for all } q \in \spt(T)\cap \bB_{7/8}\cap \bC_{\frac1{8}}\setminus B_\rho(V)\,.
\end{equation}
    \item Propagation estimates: for every $y \in \bB_{\frac1{4}}\cap V$ and $r_0 \leq r \le \frac1{2}$ it holds
    \begin{equation}\label{eq.other-propagation}
        \bE (T, \pi_0, y, r) \leq 2\, \bE (T, \pi_0, 0, 1)\,.
    \end{equation}   
    Furthermore, for $y$ and $r$ as above there exists an $m$-dimensional plane $\hat\pi_{y,r}$ such that, when we write $y$ for the point $(0,y) \in V^\perp \times V = \R^{m+n}$, it holds
\begin{align}
V &\subset \hat \pi_{y,r}\,, \label{eq.spine_inclusion} \\
 2\,\bar\bE(T,y,r) &\geq  \bE (T, \hat\pi_{y,r},y,r)\,, \label{eq.almost_optimality_flat} \\
   \bE(T,\hat\pi_{y,r},y,  r) &\le 2\,\bar\bE(T,y,\sfrac{r}{2}) \,.\label{eq.propagationNonDegneracy12}
\end{align}
    \item No-holes condition:
    \begin{equation}\label{eq.no_holes12}
    \text{for any $y\in \bB_{\frac14}\cap V$ there exists $q\in \bB_{\delta_0}(y)$ such that $\Theta_T(q)\geq Q$}\,.
    \end{equation}
\end{enumerate}
\end{lemma}

\begin{proof}

Notice that if $\eta_7 \leq \min\{\varepsilon_4,\eta_4\}$ then under the hypotheses of the lemma we can apply Lemma \ref{l:angle-bound} (see also Remark \ref{rmk:on_angle_bound}), and conclude
\begin{equation}\label{e:non-degeneracy-again}
C^2\, \beta^2_{\pi_0}(\bS) \geq C \, \beta_{\max} (\bS)^2 \geq \bE (T, \pi_0, 0, 1)\, .
\end{equation}
Nonetheless, since (b2) only provides control on the one-sided conical excess, we can't conclude that the planar excess is controlling $\beta_{\pi_0}(\bS)$. 

Fix now $\rho > 0$, $0 < r_0$, and $\delta_0 \in (0,\frac14)$, and consider sequences $\{T_k\}_{k=1}^\infty$ of currents and $\{\Sigma_k\}_{k=1}^\infty$ of manifolds satisfying, for open books $\bS_k$ and planes $\pi_k$ in $T_0\Sigma_k$, assumptions (b1)-(b2)-(b3)-(b4)-(b5) with parameters $\bar\eta=\eta_k \to 0^+$. Since the sequences are arbitrary, the proof will be complete if we can show that all the conclusions hold true along the given sequence for all sufficiently large $k$. Up to rotations, we can assume that each $\Sigma_k$ has the same tangent $T_0 \Sigma_k = \tau_0$ and is the graph of a function $\Psi_k: \tau_0\to \tau_0^\perp$ over a region including all points of interest for the rest of the proof, with $\Psi_k$ satisfying $\Psi_k(0)=0$ and $D\Psi_k(0) = 0$. Upon applying a further rotation, we may also assume that the planes $\pi_k$ coincide with a fixed plane $\pi_0 \subset \tau_0$, and, thanks to (b5), also that the spines $V(\bS_k)$ coincide with a fixed $(m-1)$-dimensional linear subspace $V \subset \pi_0$. We make the following choice of coordinates: we denote by $y = (y_1,\ldots,y_{m-1})$ the coordinates of $V$, whereas points in $\tau_0$ will be given coordinates $(x,y) = (x_1,x_2,y)$. The plane $\pi_0$ is the subspace $\{x_2=0\}$, and we let $\pi_0^\pm = \{\pm x_1 > 0\} \subset \pi_0$ denote the two halves of $\pi_0$ delimited by $V$. Coordinates in $\tau_0^\perp$ are denoted $w=(w_1,\ldots, w_{n-1})$. We also give an explicit expression of $\bS_k = \bigcup_{\pm}\bigcup_i (\bH_k)_i^\pm$ within this coordinate system. For a half-plane $(\bH_k)_i^\pm$ there exists $((\beta_k^i)^\pm,(\gamma_k^i)^\pm) \in \mathbb{S}^1$ with $\pm(\beta_k^i)^\pm \geq 0$ such that
\[
(\bH_k)_i^\pm = \left \lbrace (t\,(\beta_k^i)^\pm,t\,(\gamma_k^i)^\pm,y) \, \colon \, t \in [0,\infty)   \right\rbrace \subset \tau_0 \,.
\]
Notice that the condition that one of the coefficients $\beta = 0$ corresponds to the parametrization of one of the two half-planes of $\{x_1=0\}$ delimited by $V$.

Next, we observe that, as in the proof of Lemma \ref{lem.propagation1}, for all $k$ sufficiently large, we can apply again Proposition \ref{p:L2Alm} and guarantee the existence of Lipschitz maps $v=v_k \colon B_{3/4}=B_{3/4}(0,\pi_0) \to \mathscr{A}_Q(\pi_0^\perp)$ and closed sets $K = K_k \subset B_{3/4}$ such that \eqref{e:inclusion}-\eqref{e:W12} hold for $T=T_k$ and $\Sigma=\Sigma_k$. Writing $v_k = \left( \sum_i \a{(v_k)_i}, \eps_{v_k} \right)$ and $\bar\bE_k : = \bE(T_k, \pi_0, 0,1)$, we consider the functions $\bar v_k \colon B_{3/4} \to \mathscr{A}_Q(\pi_0^\perp)$ defined by
\[
\bar v_k := \left( \sum_i \a{\frac{(v_k)_i}{\bar\bE_k^{1/2}}}, \eps_{v_k} \right)\,,
\]
and we let $v$ be a subsequential limit (in the weak topology of $W^{1,2}$ over $B_{3/4}$, strong over $B_{1/2}$) of the $\bar v_k$'s. 

Now, we apply the same rescaling (in the coordinates $(x_2,w)$ of the orthogonal complement to $\pi_0$) to the open books $\bS_k$, and we thus obtain 
\[
\bar\bS_k = \bigcup_i (\bar\bH_k)_i^\pm\,,
\]
where 
\[
\begin{split}
(\bar\bH_k)_i^\pm &= \left \lbrace (t\,(\beta_k^i)^\pm,t\,\bar\bE_k^{-\sfrac12}(\gamma_k^i)^\pm,y) \, \colon \, t \in [0,\infty)   \right\rbrace \\
&= \left \lbrace (t\,(\bar\beta_k^i)^\pm,t\,(\bar\gamma_k^i)^\pm,y) \, \colon \, t \in [0,\infty)   \right\rbrace\,,
\end{split}
\]
having defined $$(\bar\beta,\bar\gamma) = \frac{(\beta,\bar\bE^{-\sfrac12}\,\gamma)}{|(\beta,\bar\bE^{-\sfrac12}\,\gamma)|}\,.$$ Upon passing to a (not relabeled) subsequence, the open books $\bar \bS_k$ converge to some open book $\bS_\infty$ with spine $V$. Notice that, as a consequence of \eqref{e:non-degeneracy-again}, $\bS_\infty$ cannot be flat. Next, we claim that $\spt({\bf G}_v) \cap \bC_{\frac12}(0,\pi_0) \cap \bB_{\frac78} \subset \bS_\infty$. To see this, notice first that, for any given $k$, any point $q \in\spt (\bG_{\bar v_k}) \cap {\bf C}_{\frac34}$ has coordinates 
\[
q = \left( x_1 , \bar\bE_k^{-\sfrac12} (u_k)_i (x_1,y), y, \bar\bE_k^{-\sfrac12} \Psi_k (x_1,(u_k)_i (x_1,y), y) \right)\,, \qquad (x_1,y) \in B_{3/4} \,,
\]
for some $i \in \{1,\ldots, Q\}$. Observe that $|\bar\bE_k^{-\sfrac12} \Psi_k| \leq \bar\bE_k^{-\sfrac12}\, \bA \leq \eta_k \to 0$, so that a point $q \in \spt({\bf G}_v) \cap \bC_{\frac34}$ necessarily belongs to $\tau_0$ and it has coordinates
\[
q = (x_1,v_i (x_1,y),y,0) \,, \qquad (x_1,y) \in B_{3/4}\,,
\]
for some $i$. Now, by the strong convergence of $\bar v_k$ to $v$ in $L^2$ on $B_{1/2}$ and the above observation, we have that
\begin{equation} \label{linear containment 1}
\begin{split}
   & \sum_i\int_{B_{1/2}} \dist^2 ((x_1,v_i (x_1,y),y,0),\bS_\infty) dx_1dy \\
   & \qquad \qquad \qquad = \lim_{k \to \infty}  \sum_i \int_{B_{1/2}} \dist^2((x_1 , \bar\bE_k^{-\sfrac12} (u_k)_i (x_1,y), y,0), \bar\bS_k) \, dx_1dy\,;
\end{split}
\end{equation}
on the other hand, for a point $(x_1,x_2,y) \in \tau_0$ we have 
\[
\begin{split}
\dist^2 ((x_1,x_2,y), \bar\bS_k) &= \inf\left\lbrace \abs{x_1-x_1'}^2+\abs{x_2-x_2'}^2+\abs{y-y'}^2 \, \colon \, (x_1',x_2',y') \in \bar\bS_k \right\rbrace \\
&=\inf\left\lbrace \abs{x_1-x_1'}^2+\abs{x_2-\bar\bE_k^{-\sfrac12}x_2'}^2+\abs{y-y'}^2 \, \colon \, (x_1',x_2',y') \in \bS_k \right\rbrace \\
&\leq \bar\bE_k^{-1} \, \inf\left\lbrace \abs{x_1-x_1'}^2+\abs{\bar\bE_k^{\sfrac12} x_2-x_2'}^2+\abs{y-y'}^2 \, \colon \, (x_1',x_2',y') \in \bS_k \right\rbrace \\
&= \bar\bE_k^{-1} \, \dist^2 ((x_1,\bar\bE_k^{\sfrac12} x_2,y),\bS_k)
\end{split}
\]
so that
\[
\begin{split}
&\sum_i\int_{B_{1/2}} \dist^2((x_1 , \bar\bE_k^{-\sfrac12} (u_k)_i (x_1,y), y,0), \bar\bS_k) \, dx_1dy  \\
&\qquad \qquad \qquad \leq \sum_i \bar\bE_k^{-1} \int_{B_{1/2}}\dist^2((x_1,(u_k)_i(x_1,y),y),\bS_k) \, dx_1dy \\
&\qquad \qquad \qquad \leq \bar\bE_k^{-1} \left( \bE(T_k, \bS_k, 0,1) + C \, \bar\bE_k^{1 + \gamma} \right)\,,
\end{split}
\]
where we have used \eqref{e:graph_current} and \eqref{e:volume_estimate}. Since the right-hand side is infinitesimal by (b1) and (b2), \eqref{linear containment 1} concludes the proof of the claim.

Now recall that $v \colon B_{3/4}(0,\pi_0) \to \mathscr{A}_Q(\pi_0^{\perp_0}) \simeq \mathscr{A}_Q(\R)$: the fact that its graph is supported on $\bS_\infty$ implies, in particular, that some of the pages of $\bS_\infty$ are linear graphs over $\pi_0$. 

Next, we claim that the support of the graph of $v$ is not a single hyperplane, but a non-degenerate open book with spine $V$. Suppose, towards a contradiction, that, calling $z=(x_1,y)$ the coordinate on $\pi_0$, $v(z) = (Q\a{\ell(z)},\varepsilon_v)$ for a linear function $\ell \colon \pi_0 \to \pi_0^{\perp_0} \simeq \R$. Then, calling $\bar\pi_k$ the graph of $\bar\bE_k^{\sfrac12}\,\ell$ we would have
\[
\bar\bE_k \leq 2 \, \bar\bE(T_k,0,\sfrac12) \leq 2\, \bE(T_k,\bar\pi_k,0,\sfrac12)\,,
\]
and thus, in particular,
\[
\frac12 \leq \lim_{k \to \infty} \frac{\bE(T_k,\bar\pi_k,0,\sfrac12)}{\bar\bE_k} \leq C\,\lim_{k \to \infty}  \int_{B_{\frac12}} \abs{\bar v_k \ominus \ell}^2 \,dz = 0\,,
\]
a contradiction.

Now that the fundamental properties of the limit $v$ have been established, we proceed with proving the validity of conclusions (1)(2)(3), namely that the corresponding estimates hold true for all sufficiently large $k$.

\medskip

\noindent \textbf{Proof of (1)}. Suppose that \eqref{eq.replacement_for_graphical2} fails for a subsequence (not relabeled), i.e. that there exists a sequence of points $q_k\in \spt(T_k)\cap \bB_{7/8} \cap \bC_{\frac1{8}}\setminus B_\rho(V) $ such that 
$$
\Theta(T_k,q_k)\geq Q \qquad \text{ for all $k$}\,. 
$$
Setting $z_k := \mathbf{p}_{\pi_0}(q_k)$ and $w_k:= \frac{\mathbf{p}_{\pi_0^\perp}(q_k)}{C\bar\bE_k^{\sfrac12}}$ for a suitable geometric constant $C$, Lemma \ref{Linfty-L2} implies once again that 
$$
|w_k|< \frac14\,,
$$
so that, up to subsequences, the sequence $\bar q_k=\left(z_k,w_k\right)$ converges to a point $q_0=(z_0,w_0)\in \bC_{\frac1{8}}\setminus B_\rho(V)$.

Applying estimate \eqref{eq.cheap Hardt-Simon} to $v_k$ we conclude that for every $0 < r_1 < 1/4$ it holds
$$
   \int_{B_{r_1}(z_k)\cap K_k} \frac{1}{\abs{z-z_k}^{m-2}}\sum_{i=1}^Q \left|\partial_r \frac{((\bar v_k)_i(z) - w_k)}{\abs{z-z_k}}\right|^2 \le C \,.
$$
Using the strong convergence of $\bar v_k$ to $v$ in $W^{1,2}$ and the dominated convergence theorem, we conclude that
\[
\int_{B_{r_1}(z_0)} \frac{1}{\abs{z-z_0}^{m-2}}\sum_{i=1}^Q \left|\partial_r \frac{(v_i(z) - w_0)}{\abs{z-z_0}}\right|^2 \le C \,.
\]
Recall now that, since $v$ is Dir-minimizing and takes values in $\mathscr{A}_Q (\mathbb R)$, we can apply \cite[Theorem 3.1]{DHMSS_final} to infer that $v$ is Lipschitz. In particular the real-valued map $|v\ominus w_0|$ is also Lipschitz and we can use Rademacher's theorem (cf. \cite{DLHMS_linear} for its validity in the case of multivalued functions) to get 
\[
\left|\partial_r \frac{|v (z)\ominus w_0|}{|z-z_0|}\right|^2 \leq \sum_{i=1}^Q \left|\partial_r \frac{(v_i(z)\ominus w_0)}{\abs{z-z_0}}\right|^2\, .
\]
In particular
\[
\int_{B_{r_1}(z_0)} \frac{1}{\abs{z-z_0}^{m-2}}\left|\partial_r \frac{|v (z)\ominus w_0|}{\abs{z-z_0}}\right|^2 \le C\, .
\]
But then the Lipschitz map $z\mapsto |v(z)\ominus w_0|$ must vanish in $z_0$, which in turn implies that $v (z_0) = Q \a{w_0}$. Since however $z_0$ does not belong to the spine $V$, the latter fact would contradict the structural description of the map $v$ discussed above (in particular, recall that the graph of $v$ is supported on a non-flat open book with spine $V$).

\medskip

\noindent \textbf{Proof of (2)}. We first prove the estimate in \eqref{eq.other-propagation}. Should it fail, there would be sequences $y_k \in \bB_{\frac14} \cap V$ and $r_0 \leq r_k \leq 1/8$ such that
\[
2 \leq \frac{\bE (T_k,\pi_0,y_k,r_k)}{\bar\bE_k}\,.
\]
Letting $y \in \overline{\bB}_{\frac14} \cap V$ and $r_0\leq r\leq 1/8$ be subsequential limits of $y_k$ and $r_k$ respectively, we would on the other hand have
\[
\lim_{k \to \infty} \frac{\bE (T_k,\pi_0,y_k,r_k)}{\bar\bE_k} = r^{-(m+2)} \int_{\bB_r(y)} \abs{v}^2 = \int_{\bB_1} \abs{v}^2 \leq 1\,,
\]
where we have used the invariance of $v$ with respect to $V$ and the lower semi-continuity of the $L^2$-norm with respect to weak convergence. Now, the last two displayed estimates are in contradiction. 

Next, we show that the set of $m$-dimensional planes $\pi$ for which \eqref{eq.spine_inclusion}-\eqref{eq.almost_optimality_flat} hold with $\pi$ in place of $\hat\pi_{y,r}$ is not empty; then, we show that for some choice of $\hat\pi_{y,r}$ in such set we must have \eqref{eq.propagationNonDegneracy12}. For the first claim, let $\ell \colon \pi_0 \to \pi_0^{\perp_0}$ be the linear function such that $\ell(z) = \etab \circ v (z)$ for $z \in B_{3/4}(0, \pi_0)$, and let $\pi_k$ be the graph of the function $x_2=\bar\bE_k^{1/2} \, \ell(z)$. Since $V$ is the spine of the support of $v$, we have that $\ell(0,y) = 0$ for every $y$, and thus $V \subset \pi_k$ by construction. Moreover, we have that
\[
\lim_{k \to \infty} \frac{\bE(T_k,\pi_k,y,r)}{\bar\bE (T_k,y,r)} = 1 \qquad \mbox{for every $y \in \bB_{1/4} \cap V$ and $r_0 \leq r \leq 1/8$}
\]
as a consequence of the strong convergence $\bar v_k \to v$ in $L^2(B_{3/4}(0,\pi_0))$ and the definition of $\pi_k$. This proves the existence of planes satisfying the conditions in \eqref{eq.spine_inclusion}-\eqref{eq.almost_optimality_flat}.

Assume finally that \eqref{eq.propagationNonDegneracy12} fails along a subsequence, that is there are points $y_k \in \bB_{1/4} \cap V$ and radii $r_0 \leq r_k \leq 1/4$ such that whenever $\pi_k$ is an $m$-dimensional linear subspace of $\tau_0$ with $V \subset \pi_k$ and $2\,\bar\bE (T_k,y_k,r_k) \geq  \bE(T_k,\pi_k,y_k,r_k)$ we get
\begin{equation} \label{cont_non_deg}
    \bE(T_k,\pi_k,y_k,r_k) > 2\, \bE(T_k,\pi_k,y_k,\sfrac{r_k}{2}) \,.
\end{equation}

First, we claim that such a plane $\pi_k$ must be the graph over $\pi_0$ of a linear function $h_k \colon \pi_0 \to \pi_0^{\perp_0}$ with $\left.h_k\right|_V \equiv 0$ satisfying
\begin{equation} \label{linear compactness}
|\nabla h_k|^2 \leq C(r_0)\,\bar\bE_k\,,
\end{equation}
where $\bar\bE_k = \bE(T_k,\pi_0,0,1)$ as usual. Indeed, for every $k$ let $z$ be a point in $K_k \cap \bB_{r_k/2}(y_k)$, and observe that, if $q(z) \in \spt (T_k)$ satisfies $\abs{z-q(z)}= \dist (z,\spt (T_k))$, then by \eqref{e:W12} we have
\[
\abs{z-q(z)}^2 \leq |v_k (z)|^2 \leq C\, \bar\bE_k \,. 
\]
On the other hand, by Lemma \ref{Linfty-L2}  and the almost-optimality of $\pi_k$ we also get
\[
\dist(q(z), \pi_k)^2 \leq C\,\bE (T_k,\pi_k,y_k,r_k) + C \bA^2 \leq C(r_0)\, \bar \bE_k\,.
\]
This shows that there exists a large set of points $z \in\pi_0 \cap \bB_{r_k/2}(y_k)$ such that
\[
\dist(z,\pi_k)^2 \leq C(r_0)\, \bar\bE_k\,,
\]
thus proving the claim.

As a consequence of \eqref{linear compactness}, modulo passing to (not relabeled) subsequences, we have that $y_k \to y\in \bar\bB_{1/4} \cap V$, $r_k \to r \in \left[r_0,1/4\right]$, and the functions $\ell_k = \frac{h_k}{\bar\bE_k^{1/2}}$ converge to a linear function $\ell$ over $\pi_0$. Since $\bB_{r_k}(y_k) \subset \bB_{1/2}$, using the Lipschitz bound on $v_k$ and the fact that the Lipschitz constant of $h_k$ converges to $0$, we get, under the assumption that \eqref{cont_non_deg} holds, 
\begin{align*}
\left(\frac{2}{r}\right)^{m+2} \int_{B_{\frac{r}{2}}(y)} \abs{v \ominus \ell}^2
    &=\lim_{k \to \infty} \frac{\bE(T_k,\pi_k,y_k,\sfrac{r_k}{2})}{\bar\bE_k} \leq \frac12 \lim_{k \to \infty} \frac{\bE(T_k,\pi_k,y_k,r_k)}{\bar\bE_k} \\
    &= \frac{1}{2}\frac{1}{r^{m+2}} \int_{B_{r}(y)} \abs{v \ominus \ell}^2 
 = \frac{1}{2}\left(\frac{2}{r}\right)^{m+2} \int_{B_{\frac{r}{2}}(y)} \abs{v \ominus \ell}^2 \,,
\end{align*}
which contradicts the fact that $v$ is non-flat. 

\medskip

\noindent \textbf{Proof of \eqref{eq.no_holes12}} Finally assume that \eqref{eq.no_holes12} fails, that is there exists a sequence of points $y_k\in \bB_{\frac14}\cap V$ such that $\Theta_{T_k}(q)<Q$ for every $q\in \bB_{\delta_0}(y_k)$. Therefore \eqref{no Q points} is satisfied in the cylinder $B_{\delta_0/2}(y_k,\pi_0)\times \pi_0^\perp$ and so by Proposition \ref{p:L2Alm-piece-3} we have that $v_k|_{B_{\delta_0/2}(y_k,\pi_0)}=(\sum_{i}\a{(v_k)_i},\eps)$ with $\eps\in \{-1,1\}$ a constant, and $(v_k)_1\leq \dots\leq (v_k)_Q$ each satisfying the minimal surfaces equation in $\Sigma_k$. Since up to subsequences we can assume that $y_k\to y\in V\cap \bar\bB_{1/4}$, it follows that in $B_{\delta/4}(y,\pi_0)$ the functions $(\bar v_k)_i = \bar{\bE}_k^{-\sfrac{1}{2}} (v_k)_i$ converge in the $C^1$ topology to harmonic functions. In particular there would be a $C^1$ selection for $v$ in $B_{\delta/4} (y,\pi_0)$, which is not possible, because it would contradict the structural description of $v$. 
\end{proof}

\section{Proof of Proposition \texorpdfstring{\ref{p:decay-2}}{decay2}: Whitney decomposition}

Using the results of the previous section we can now adapt the graphical parametrization constructed in \cite{DHMSS} to our setting. In view of Lemma \ref{lem.propagation1}, we start by updating Assumption \ref{ass:decay plane -1} into

\begin{ipotesi}\label{ass:decay plane}
Assumption \ref{ass:decay plane -1} holds, and in addition $\bE (T,\pi_0,0,1) \leq 2\,\bar\bE (T,0,\sfrac12)$ and $V (\mathbf{S})\subset \pi_0$.
\end{ipotesi}

Recall then that under the above Assumption \ref{ass:decay plane} we set coordinates $(x,y,w)$ in $\R^{m+n}$, where $y=(y_1,\ldots,y_{m-1})$ are the coordinates on the spine $V(\bS)$, $\pi_0$ has coordinates $(x_1,y)$, and $T_0\Sigma$ has coordinates $(x_1,x_2,y)$. The half-planes $\pi_0^\pm$ are defined by $\pi_0^\pm = \{\pm x_1 > 0\} \subset \pi_0$.

Next, we need to identify the domains on which the different graphical approximations of $T$ are going to be defined. These will consist of a union of cubes in a \emph{Whitney}-type decomposition of (a subset of) $\left[0, \infty \right) \times V$ with suitably good properties. Here, the coordinate $t$ on the ``abstract'' closed half-line $[0,\infty)$ will play the role of the distance function from $V$. 

Fix a large positive integer $N_0 \in \mathbb N$, and consider the rectangle \[\texttt{R}_{0}:=[0,2^{-N_0}]\times [-2,2]^{m-1} \subset \left[ 0, \infty \right) \times V\,,\] as well as the collection $\mathcal{L}_{N_0}$ of sub-cubes defined as follows. First, we partition $(0,2^{-N_0}]$ into the dyadic intervals $\{ [2^{-k},2^{-k+1}] \}_{k > N_0}$. Then, we further divide each layer $[2^{-k},2^{-k+1}] \times [-2,2]^{m-1}$ into sub-cubes of side-length $2^{-(k+M)}$, where $M$ is a large integer to be chosen later, cf. Figure \ref{figura-5}. If $L \in \mathcal L_{N_0}$ has side-length $2^{-(k+M)}$, we will say that $L$ has \emph{order $k$}. Notice that 
\begin{equation} \label{e:dist v diam} 
\frac{2^{M+1}}{\sqrt{m}} \,\diam(L) \ge \max_{z\in L} \dist(z,V)\ge \min_{z\in L} \dist (z,V) \ge \frac{2^M}{\sqrt{m}} \,\diam(L) \quad \forall L\in \mathcal{L}_{N_0}\, .
\end{equation} 

For any $L \in \mathcal{L}_{N_0}$, we shall denote $c_L=(t_L, y_L)$ the center of $L$ and $d_L$ the diameter of $L$. In order to ease the notation, we will write $y_L$ in place of the more cumbersome $(0,y_L,0) \in \R^{m+n}$, and we will be interested only in those cubes $L$ for which $|y_L| < 3/4$. For such cubes $L$ we introduce the notation
\[
\bE_L:=\bE\left(T,\bS,y_{L},\bar M d_{L}\right)
\qquad \mbox{and}\qquad \bar \bE_L:= \bar \bE\left(T,y_{L},\bar M d_{L}\right)
\,,
\]
where $\bar M := 2^{M+6} / \sqrt{m}$ and $\bar \bE_L$ is computed by minimizing $\bE(T,\pi,y_L,\bar M d_L)$ among $m$-dimensional planes $\pi \subset T_0 \Sigma$. The parameter $N_0$ is chosen so large that if $L \in \mathcal{L}_{N_0}$ is a cube with $|y_L| < 3/4$ then $\bB_{\bar M d_L}(y_L) \subset \bB_1(0)$.

\begin{figure}
\begin{tikzpicture}
\foreach \x in {1,...,4}
\draw (0, {2+0.5*\x}) -- (6,{2+0.5*\x}) (0,{1+0.25*\x}) -- (6,{1+0.25*\x}) (0,{0.5+0.125*\x}) -- (6,{0.5+0.125*\x}) (0,{0.25+0.0625*\x}) -- (6,{0.25+0.0625*\x}) (0,{0.125+0.03125*\x}) 
-- (6,{0.125+0.03125*\x});
\draw (0,0.125) -- (6,0.125);
\foreach \y in {0,...,12}
\draw ({0.5*\y},2) -- ({0.5*\y},4); 
\foreach \y in {0,...,24}
\draw ({0.25*\y},1) -- ({0.25*\y},2);
\foreach \y in {0,...,48}
\draw ({0.125*\y},0.5) -- ({0.125*\y},1);
\foreach \y in {0,...,96}
\draw ({0.0625*\y},0.25) -- ({0.0625*\y},0.5);
\foreach \y in {0,...,192}
\draw ({0.03125*\y},0.125) -- ({0.03125*\y},0.25);
\fill[black] (0,0) -- (6,0) -- (6,0.125) -- (0,0.125) -- (0,0);
\draw[very thick] (0,0) -- (6,0);
\end{tikzpicture}
\caption{The Whitney decomposition of $[0,2^{-N_0}] \times [-2,2]^{m-1}$. In the above example the parameter $M$ equals $2$.} \label{figura-5}
\end{figure}
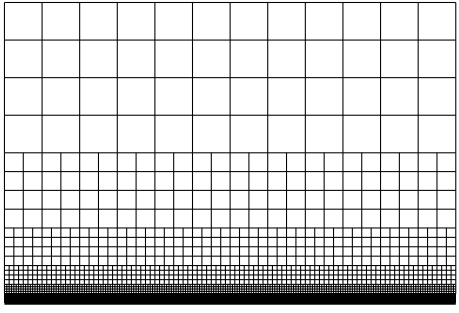

\begin{definition}(Whitney domains) \label{def:whitney}
We establish the following partial order relation in $\mathcal L$: if $L,L' \in \mathcal L$, we say that $L$ \emph{is below} $L'$, and we write $L \preceq L'$, if and only if $\mathbf{p}_V(L) \subset \mathbf{p}_V(L')$.
Let $T$ be as in Assumption \ref{ass:everywhere}, and let $\bS \in \mathscr{B}(0)$. For $\tau, \eta \in \left( 0, 1/2 \right)$, we define the following regions.
\begin{itemize}
    \item[(W)]  The {\em good Whitney domain} of $\texttt{R}_0$ associated with $(T,\bS,\tau,\eta,N_0)$, denoted by $\mathcal{W}=\mathcal{W}(T,\bS,\tau,\eta, N_0)$, is the subfamily of $L \in \mathcal{L}_{N_0}$ with $|y_L| < 3/4$ such that
    \begin{equation} \label{e:the good}
\bE_{L'} < \tau^2 \qquad \mbox{and} \qquad \bE_{L'} < \eta\, \bar\bE_{L'}
\end{equation}
for all $L \preceq L'$.

\item[(B)] The {\em bad Whitney domain} of $\texttt{R}_0$ associated with $(T,\bS,\tau,\eta, N_0)$, denoted by $\mathcal{B}=\mathcal{B}(T,\bS,\tau,\eta,N_0)$ is the subfamily of $L \in \mathcal{L}_{N_0}$ with $|y_L| < 3/4$ such that $L' \in \mathcal W$ for all $L \preceq L'$ with $L' \neq L$ and
    \begin{equation}\label{e: the bad}
\bE_L < \tau^2 \qquad \mbox{and} \qquad \bE_L \ge \eta\, \bar\bE_L\,.
\end{equation}

% {\color{blue}[ I would define the "bad" cubes as the subfamily
% \[\mathcal{B}=\{L \in \mathcal{L} \colon \mathbf E(T,\bS,y_{L},\bar M d_{L}) \ge \eta \bar\bE(T,y_{L},\bar M d_{L})  \text{ and } L' \in \mathcal{W} \quad\forall L\preceq L'   \}\] 
% and the bad region 
% }

%\item[(U)] The {\em ugly Whitney domain} of $[0,2^{-N_0}] \times [-2,2]^{m-1}$ associated to $(T,\bS,\tau,\eta)$, is the family $\mathcal{U}=\mathcal{U}(T,\bS,\tau,\eta, N_0)$ of cubes $L \in \mathcal{L}_{N_0}$ such that $L' \in \mathcal W$ for all $L \preceq L'$ with $L' \neq L$ and
%    \begin{equation} \label{e:the ugly}
%\bE_L \ge \tau^2 \,.
%\end{equation}

\end{itemize}

\end{definition}

Since we will often deal with suitable dilations of the cubes in $\mathcal{L}_{N_0}$, we introduce the following notation. For $1 \leq \lambda \leq 2^M$, $\lambda L$ is the cube with the same center $c_L$ as $L$ and diameter $d_{\lambda L} = \lambda \, d_L$.

We next define the regions where we shall build the graphical parametrization of the current. First, given $1\leq \lambda \leq 2^M$, we set
\begin{equation}
    U_{\lambda \mathcal W} := \bigcup_{L \in \mathcal W} \lambda L \cup \left(\left[2^{-N_0}, 2^{-1}\right]\times B_{\sfrac34}^{m-1}(0)\right)\,,
\end{equation}
and, setting $U_{\mathcal W} = U_{1\,\mathcal W}$, we define the function $\varrho_{\mathcal W} \colon B_{\sfrac34}^{m-1}(0) \to \left[ 0, 2^{-N_0} \right]$ as
\begin{equation} \label{oh dear vrho}
    \varrho_{\mathcal W}(y) := \inf \left\lbrace t \, \colon \, (t,y) \in U_{\mathcal W}  \right \rbrace\,.
\end{equation}
We also define
\begin{equation}
    \label{rotation of Whitney}
    R_{\lambda\mathcal W} := (\p_{\pi_0}^{-1}(U_{\lambda\mathcal W}^+) \cup \p_{\pi_0}^{-1}(U_{\lambda\mathcal W}^-))\cap \bB_{3/4}\,,
\end{equation}
where, for a domain $U \subset \left[0,\infty\right) \times V$, we have denoted
\begin{equation} \label{e:symmetric domains}
U^\pm := \left\lbrace (\pm t,0,y) \in \pi_0^\pm \subset T_0\Sigma \, \colon \, (t,y) \in U \right\rbrace
\end{equation}
the corresponding domains on $\pi_0^\pm$. 

Next, we consider the family 
\[
\mathcal{F} = \left\lbrace \bB_{\bar M d_L}(y_L) \, \colon \, L \in \mathcal B \right\rbrace\,,
\]
and we let $\left\lbrace \bB_{\bar M d_{L_i}}(y_{L_i}) \right\rbrace_{i \in \mathbb{N}}$ be a Vitali covering of $\bigcup \mathcal F$: that is, each $L_i \in \mathcal B$, the balls $\bB_{\bar M d_{L_i}}(y_{L_i})$ are pairwise disjoint, and 
\[
\bigcup_{L \in \mathcal B} \bB_{\bar M d_L} (y_L) \subset \bigcup_{i \in \mathbb N} \bB_{5 \bar M d_{L_i}}(y_{L_i})\,.
\]
To ease the notation, we set $d_i := d_{L_i}$, $y_i := y_{L_i}$, $\bE_i := \bE_{L_i}$, and $\bar \bE_i := \bar \bE_{L_i}$, and, with this notation in place, we define

\begin{equation} \label{rotation of bad Whitney 2}
R_{\mathcal B} := \bigcup_{i
\in \mathbb N} \bB_{5 \bar M d_i} (y_i) \setminus B_{C_\star(\eta^{-1}\bE_i)^{\sfrac12}d_i} (V)\,,
\end{equation}
where $C_\star$ is a geometric constant.

\medskip

Finally, before proceeding we also record the following 
\begin{remark}\label{r:cube-control-eta-tau}
If $T$ satisfies Assumption \ref{ass:decay plane}, then we have 
\begin{equation}\label{eq:W3}
\bar\bE_L\leq \bE(T, \pi_0, y_L, \bar M\,d_L)\leq C \, (\tau^2+ \bar\varepsilon)\qquad \forall  L\in  \mathcal W\cup \mathcal B\, .
\end{equation}
Indeed, the first inequality is trivial, while the second inequality follows from
\begin{align*}
    \bE(T, \pi_0, y_L, \bar M\,d_L)
        &\leq \bE(T, \bS, y_L, \bar M\,d_L)+C\,\beta_{\pi_0}(\bS)^2\\
        &\stackrel{\eqref{eq:angle_bound}}{\leq} \tau^2+ C\, \bE(T, \pi_0, 0,1)\leq C\,(\tau^2+ \bar\varepsilon)\,.
\end{align*}
\end{remark}

\subsection{Graphicality on good cubes} In the following theorem, we are going to represent $T$ as a special multi-valued graph in the region $R_{\mathcal W}$ which ``projects'' onto the good Whitney domain. As it will become apparent in the proof, in this region the hypotheses from Proposition \ref{p:L2Alm-piece-3} will be satisfied, so that the special multi-valued function $u$ which parametrizes $T$ (in the sense that $T$ is the graph of $v$ when $v(z) = u(z) + \Psi(z+u(z))$) will come equipped with a $C^{1,\sfrac12}$ selection as specified in Proposition \ref{p:L2Alm-piece-3}. The latter may then be considered as a ``$p$-multifunction'' on the ``abstract'' domain $U = U_{\mathcal W}$ over the (degenerate) open book $\pi_0$; see \cite[Definition 5.4]{DHMSS}. More precisely, a $p$-multifunction in the present context will be a collection $u=\{u^\pm_{j}\}_{j=1}^Q$ of functions of class $C^{1,\sfrac12}$ defined on domains $U^\pm$ corresponding to some domain $U \subset \left(0,\infty \right) \times V$ as specified in \eqref{e:symmetric domains}. For every $\zeta = (t,y) \in U$, we let $\zeta^\pm := (\pm t, 0, y) \in \pi_0^\pm$, and we set
\begin{equation*}
    \begin{split}
        |u(\zeta)| :&= \max_{\pm} \max_j |u^\pm_j(\zeta^\pm)|\,,\\
         |D u(\zeta)| :&= \max_{\pm} \max_j |D u^\pm_j(\zeta^\pm)|\,,\\
          [D u]_{\sfrac12}(\zeta) :&= \max_{\pm} \max_j [D u^\pm_j]_{\sfrac12}(\zeta^\pm)\,,
    \end{split}
\end{equation*}
where, for $z \in U^\pm$, we have set
\[
[D u^\pm_j]_{\sfrac12}(z) := \inf_{\rho > 0} \sup \left\lbrace 
\frac{|D u^\pm_j(z_1) - D u^\pm_j(z_2)|}{|z_1-z_2|^{\sfrac12}}\, \colon \, z_1 \neq z_2,\, z_k \in U^\pm \cap \bB_\rho(z)\right\rbrace\,.
\]
Finally, we define the weighted $C^{1,\sfrac12}$ norm for a $p$-multifunction $u = \{u^\pm_j\}_j$ setting
\[
\|u\|_{C^{1,\sfrac{1}{2}}_* (U)} := 
\sup_{\zeta = (t,y) \in U} \left( t^{-1} |u (\zeta)| 
+  |D u (\zeta)| + t^{\sfrac12} \, [D u]_{\sfrac12} (\zeta) \right)\,.
\]

\begin{theorem}[Graphical parametrization] \label{thm:graph_v1}
Let $T$ and $\Sigma$ be as in Assumptions \ref{ass:everywhere}. For any $N_0\in \N$ there are $\tau_{8} > 0$, $\eta_{8} > 0$, and $C \geq 1$, depending on $(m,n,p,N_0)$ with the following property. If:
\begin{itemize}
\item[(a)] the values of the parameters $\eta$ and $\tau$ in Definition \ref{def:whitney} are smaller than $\eta_{8}$ and $\tau_{8}$,   
\item[(b)] Assumption \ref{ass:decay plane} holds with $\bar{\eps} < C^{-1} \tau^2$, $\bar{\eta} < C^{-1} \eta$ for some $\bS$ and $\pi_0$, 
\end{itemize}
then there
    is a special $Q$-valued map $u = (\sum_j \a{u_j}, \varepsilon)$ over $U_{4\mathcal W}^+ \cup U_{4\mathcal{W}}^-$, with the following properties:
\begin{itemize}
% \item[(i)] $T\res \mathcal C_{2^{-N_0}}=\bG_v \res \mathcal C_{2^{-N_0}}$;
% % every $L \in \mathcal L$ with $d_L \ge C_2\, \frac{\mathbb E^{1/(m+2)}}{\beta} $ belongs to $\mathcal W \cup \mathcal B$;
\item[(i)] $\varepsilon$ is constant on each of the two domains $U_{4\mathcal W}^{\pm}$, each $u^\pm_{j} \colon U_{4\mathcal W}^{\pm} \to \pi_0^{\perp_0}$ is of class $C^{1,\sfrac{1}{2}}_{{\rm loc}}$, and, regarding $u$ as the $p$-multifunction $\{u_j^\pm\}_{j=1}^Q$, we have $\|u\|_{C_*^{1,\sfrac12}(U_{4\mathcal W})} \leq C\,\tau$;
\item[(ii)] $T \mres R_{4\mathcal W} = \mathbf{G}_v \mres R_{4\mathcal W} $, where $v= (\sum_j \a{v_j}, \varepsilon)$ is the special $Q$-valued function on $U_{4\mathcal W}^+ \cup U_{4\mathcal W}^-$ defined by 
\begin{equation} \label{e:function to manifold}
v^\pm_{j}(z) := u^\pm_{j}(z) + \Psi(z + u^\pm_{j}(z))\,;
\end{equation}
\item[(iii)] if $L\in \mathcal B(T,\bS,\tau,\eta,N_0)$ then there exists $\xi_L \in \bB_{\bar M d_L/2}(y_L)$ with $\Theta_T(\xi_L)\geq Q$;
\item[(iv)] the following estimate holds:
\begin{equation} \label{e:L^2 estimate}
    \int_{\bB_{1/2} \setminus (R_{2\mathcal W} \cup R_{\mathcal B})} \dist (q, V)^2 \, d\|T\| 
    %+ \int_{\bB_{1/2} \setminus (R_{\mathcal W}\cup R_{\mathcal B})} |x|^2 \, d\|\bC\| 
    \leq C \, \tau \eta^{-\sfrac32} \, \mathbf E(T,\bS,0,1)\, ;
\end{equation}
\item[(v)] For every fixed $\eta$, $\tau$, and $\rho$, if $\bar \eta$ and $\bar\varepsilon$ are chosen sufficiently small, then $\varrho_{\mathcal{W}} (y)\leq \rho$ for all $y\in \bB_{1/4} \cap V$.
\end{itemize}
\end{theorem}

\begin{proof} In this proof all constants denoted by $C$ can only depend on $Q,m,n,$ and $N_0$. If the constant does not depend on $N_0$ it will then be denoted by $\bar C$.

First of all, if the constant $C$ in (b) is chosen large enough, and if $tau_8$ and $\eta_8$ are chosen small enough, it follows from Assumption \ref{ass:decay plane} and Corollary \ref{cor:kick-off} that the cubes of order $(N_0+1)$ belong to $\mathcal{W}$. This is important, as it guarantees that every cube $L\in \mathcal{B}$ has a father in $\mathcal{W}$. Moreover, for fixed $\rho > 0$ and $\delta_0 \in \left( 0, 1/4\right)$, the hypotheses of Lemma \ref{l:propagation} are satisfied at the scale of all cubes $L$ of order $(N_0+1)$: that is, conditions (b1) up to (b5) in Lemma \ref{l:propagation} are satisfied with $T$ replaced by $T_L = (\lambda_{y_L,\bar M d_L})_\sharp T$ and $\pi_0$ replaced by $\hat\pi_{L} := \hat\pi_{y_L,\bar M d_L}$ from Corollary \ref{cor:kick-off}.

Next, we claim that if $L$ is a cube in $\mathcal W \cup \mathcal B$ then we can apply Lemma \ref{l:propagation} at the scale of $L$. The proof is by induction on the order $k$ of the cube. The claim is true for $k=N_0+1$. Let us then fix a cube $L$ of order $k+1$ which is in $\mathcal W \cup \mathcal B$, and make the induction hypothesis that Lemma \ref{l:propagation} can be applied to all cubes $L'$ of order $N_0+1 \leq j \leq k$ that are in the ancestry of $L$. We shall prove that the lemma can be applied to $L$. We let $L'$ denote the ``father'' of $L$, i.e. the cube of order $k$ which is closest to $L$: notice that $L' \in \mathcal W$, regardless of whether $L \in \mathcal W$ or $L \in \mathcal B$. Now, we observe that:
\begin{itemize}
    \item[$\bullet$] $\bar\bE (T,y_L,\bar M d_L) \leq \bar C (\tau^2 + \bar\varepsilon) \quad \mbox{by Remark \ref{r:cube-control-eta-tau}}$;
    \item[$\bullet$] the inequalities 
    \begin{align*}
    \bE (T,\bS,y_L,\bar M d_L) \leq 2^{m+2}\, \bE(T,\bS,y_{L'},\bar M d_{L'}) 
     & < 2^{m+2}\,\eta\,\bar\bE(T,y_{L'},\bar M d_{L'})\\
     &\leq 2^{m+2}\,\eta\,\bar\bE(T,y_L,\bar M d_L)
    \end{align*}     
    hold by the definitions of $\mathcal B$ and $\mathcal W$ and \eqref{eq.propagationNonDegneracy12};
    \item[$\bullet$] $(\bar M d_L)^2\bA^2 \leq (\bar M d_L)^2 \, \bar\eta\,\bE(T,\pi_0,0,1)\leq 2(\bar M d_L) \, \bar\eta\,\bar\bE(T,y_L,\bar M d_L)$ by Assumption \ref{ass:decay plane} and an iterative application of \eqref{eq.propagationNonDegneracy12} over the ancestry of $L$;
    \item[$\bullet$] there exists $\hat \pi_L$ so that $V \subset \hat \pi_L$, $2\,\bar\bE(T,y_L,\bar M d_L) \geq\bE(T,\hat\pi_L,y_L,,\bar M d_L)$, and 
    \[
    \bE(T,\hat\pi_L,y_L,\bar M d_L) \leq 2 \, \bar\bE(T,y_L,\sfrac{\bar M d_L}{2})\, ,
    \]
    as a consequence of Lemma \ref{l:propagation} applied at scale $L'$.
\end{itemize}
The above considerations imply that, if $\tau_8$ and $\eta_8$ are chosen small enough, then Lemma \ref{l:propagation} applies indeed. In particular, we conclude that $\Theta (T,q)<Q$ for every $q\in \spt (T) \cap \bC_{\bar M d_L/8} (y_L, \hat\pi_L) \cap \bB_{7 \bar M d_L/8} (y_L) \setminus B_{\rho d_L} (V)$. We can then apply Proposition \ref{p:L2Alm-piece-3} in $\mathbf{p}_{\pi_0}^{-1} (8 L) \cap \bB_{\bar M d_L/2}(y_L)$ (where we used the short-hand notation $\mathbf{p}_{\pi_0}^{-1} (8 L)$ for $\mathbf{p}_{\pi_0}^{-1} (8 L^+) \cup \mathbf{p}_{\pi_0}^{-1} (8 L^-)$) to conclude that the support of $T$ decomposes into smooth minimal surfaces over $4L$. Observe that, as a consequence of Lemma \ref{Linfty-L2} and of the planar excess estimates obtained at the scales of all cubes $L \in \mathcal W \cup \mathcal B$, we have
\[ 
    \spt(T) \cap \bB_{3/4} \cap \bC_{\bar M d_L /4}(y_L, \pi_0) \subset \left\lbrace q \, \colon \, |\mathbf{p}_{\hat\pi_L}^\perp (q)|^2 \leq C (\tau^2 + \bar\varepsilon)\,d_L^2 \right\rbrace \,,
\]
so that
\begin{equation}\label{e:inductive height bound}
    \spt(T) \cap \bB_{3/4} \cap \bC_{\bar M d_L /4}(y_L, \pi_0) \subset \bB_{\bar M d_L / 2}(y_L)\,.
\end{equation}
This guarantees that $\mathbf{p}_{\pi_0}^{-1} (4L) \cap \bB_{\bar M d_L/2}(y_L) \cap \spt (T) =\mathbf{p}_{\pi_0}^{-1} (4 L) \cap \bB_{3/4} \cap \spt (T)$. The graphical representation over $U_{4 \mathcal{W}}$ follows now from noticing that, where cubes $4L$ and $2L'$ coincide, the corresponding functions must agree because they parametrize the same piece of the current. In particular this proves (ii). For the argument leading to the precise estimate claimed in (i) we refer the reader to \cite[Section 5]{DHMSS}

Fix now $L\in \mathcal{B}$ and $L'$ be the ``father'' of $L$ as above. By Lemma \ref{l:propagation}, we have 
 \begin{equation} \label{Q point of father}
 \forall y\in \bB_{\frac{{\bar M d_{L'}}}{4}}(y_{L'})\cap V\quad \exists \xi_{L}\in \bB_{\delta_0 \bar M d_{L'}}(y)\quad 
 \mbox{such that $\Theta_T(\xi_{L})\geq Q$}\,.
 \end{equation}
 We apply \eqref{Q point of father} with $y=y_L$, and thus we guarantee the existence of $\xi_L \in \bB_{\frac{\bar M d_L}{2}}(y_L)$ with $\Theta_T(\xi_L) \geq Q$. This proves (iii). 
 
  We next come to (iv). We first claim that

\begin{equation} \label{covering1}
    \spt (T) \cap \bB_{1/2} \setminus R_{2\mathcal W} \subset \bigcup_{L \in \mathcal B} \bB_{\bar M d_L}(y_L) \subset \bigcup_{i \in \mathbb N} \bB_{5 \bar M d_i} (y_i)\,.
\end{equation}

% \begin{equation} \label{covering1}
% \begin{split}
%   &  \spt (T) \cap \bB_{1/2} \setminus (R_{\mathcal W}\cup R_{2\mathcal B}) \subset  \bigcup_{L \in \mathcal B} U_L\,,\\
% &   U_L :=\left\lbrace q=(x,y,w) \in \R^{m+n} \, \colon \, y \in \p_V ({\rm int}(L)) \mbox{ and } %C\,(\bar\bE^\frac12_L+\bA_L)
%     \sqrt{|x|^2+|w|^2} \leq \frac {C}{\eta^{\frac12}} \,d_L\,\bE_L^{\frac12}  \right\rbrace\,.
%     \end{split}
% \end{equation}

To see this, let $L \in \mathcal{L}_{N_0}$ with $|y_L| < 3/4$ be such that $L \notin \mathcal W$ but $L' \in \mathcal W$ for every $L \preceq L'$ with $L' \neq L$. In particular, let $L'$ be the father of $L$. Since $L' \in \mathcal W$, and assuming $\tau_8$ and $\eta_8$ are sufficiently small, we can apply Lemma \ref{l:angle-bound} and Remark \ref{rmk:on_angle_bound} to conclude that $\bE_{L'} \leq \eta\,\bar \bE_{L'} \leq C\,\eta\,\beta_{\max}^2(\bS)$. Now, since $\bB_{\bar M d_{L}}(y_{L})\subset \bB_{\bar M d_{L'}}(y_{L'})$, and since $d_{L'}=2\,d_L$, we have that
\[
\bE_L \leq 2^{m+2}\, C\, \eta\, \beta_{\max}^2(\bS)\,.
\]
On the other hand, by Lemma \ref{l:angle-bound} we have $\beta_{\max}^2(\bS) \leq \bar C\,\bE(T,\pi_0,0,1) \leq \bar C\bar\varepsilon$. Hence, a suitable choice of $\bar\varepsilon$ guarantees that $\bE_L < \tau^2$, namely that $L \in \mathcal B$. 

Next, let $L \in \mathcal L_{N_0}$ with $|y_L| < 3/4$ be such that $L \notin \mathcal W$, and let $L'$ be the largest ancestor of $L$ such that $L' \notin \mathcal W$. By the considerations above, $L' \in \mathcal B$, and thus $\mathbf{p}_V({\rm int}(L)) \subset \mathbf{p}_{V}({\rm int} (L'))$ for some $L' \in \mathcal B$.

With this in mind, let now $q=(x,y,w) \in \spt (T) \cap \bB_{1/2}$, and let $t=\sqrt{|x|^2+|w|^2} = \dist (q,V)$. If $t \geq \varrho_{\mathcal W}(y)$, then $(t,y)\in \overline{U_{\mathcal W}}$, and $q \in \bB_{\bar M d_L}(y_L)$ for some $L \in \mathcal W$. Applying Lemma \ref{Linfty-L2} and Remark \ref{r:cube-control-eta-tau}, we have, on the other hand,
\[
\abs{{\bf p}_{\pi_0} (q)} \leq C \, d_L \, \left( \bE(T, \pi_0, y_L, \bar M d_L) + d_L^2 \bA^2 \right)^{\sfrac12} \leq C\,d_L \, (\tau + \bar\varepsilon^{\sfrac12})\,,
\]
so that, if $\tau_8$ is sufficiently small, $q \in R_{2\mathcal W}$. If, instead, $t <  \varrho_{\mathcal W} (y)$ then $\varrho_{\mathcal W}(y) > 0$ and, by the considerations above, $y \in \mathbf{p}_V({\rm int}(L))$ for some $L \in \mathcal B$, and thus $q \in \bB_{\bar M d_L}(y_L)$, completing the proof of \eqref{covering1}.

% We start by estimating the first union $U_1$ by applying applying Vitali's covering theorem to find pairwise disjoint balls $\bB_{r_i} (y_i)$ with $y_i = y_{L_i}$ and $r_i = 2\bar M d_{L_i}$ for $L_i \in \mathcal U$ such that $\{\bB_{5r_i} (y_i)\}$ covers $U_1$. Using the monotonicity formula and the fact that by definition of $\mathcal U$
% $$
% \bE(T, \bS, y_i, r_i/2)\geq \tau^2\,,
% $$ 
% we then have
% \begin{align*}
% & \int_{U_1} |x|^2 \, d\|T\| + \int_{U_1} |x|^2 \, d\|\bC\|\\
% \leq & \sum_i C r_i^2 ( \|T\|(\bB_{5r_i}(y_i)) + \|\bC\|(\bB_{5r_i}(y_i))  ) \leq \sum_i C r_i^{m+2}\\
% \leq & \bar C \, \tau^{-2} \sum_i \int_{\bB_{r_i}(y_i)} \dist^2(\cdot,\bS) \, d\|T\| \leq \bar C \,\tau^{-2}\, \bE(T,\bS,0,1)\, .
% \end{align*} 

We can now complete the proof of (iv). We  notice that for each $i \in \mathbb N$, denoting $U_i := \bB_{5 \bar M d_i}(y_i) \cap B_{C_\star (\eta^{-1}\bE_i)^{\sfrac12}d_i} (V)$, we have
\[\begin{split} \int_{U_{i}} \dist^2(q,V)\, d\norm{T}(q) &\le \frac{C_\star^2}{\eta}\, d_i^2\, \bE_{i} \, \norm{T}(U_{i}) \le \frac{C_\star^2}{\eta}\, d_i^{m+2}\, \bE_i\, \bE_{i}^{\sfrac12} \\ &\le \frac{C\tau}{\eta^{\sfrac32}} \int_{\bB_{\bar M d_i}(y_i)} \dist^2(\cdot, \bS)\, d\norm{T}\,,
\end{split}\]
where we used that $\bE_i^{\sfrac12}\le \tau$ by \eqref{e: the bad}, and that, by a simple covering argument and the monotonicity formula for $T$,
    \[
    \|T\|(U_{i})\leq C\, d_{i}^{m-1} \, \eta^{-\sfrac12} \,  d_i \,\bE_L^{\sfrac12}\,.
    \]
By \eqref{covering1} and the definition of $R_{\mathcal B}$ in \eqref{rotation of bad Whitney 2}, \[\spt (T) \cap \bB_{1/2} \setminus (R_{2\mathcal W} \cup R_{\mathcal B}) \subset \bigcup_{i \in \mathbb N} U_i\,,\] and thus \eqref{e:L^2 estimate} follows by summing over $i$, keeping in mind that the balls $\bB_{\bar M d_i}(y_i)$ are pairwise disjoint.

\medskip

Finally, (v) is just a consequence of Corollary \ref{cor:kick-off}.
\end{proof}

\subsection{Improved \texorpdfstring{$L^2$}~ estimate} 
The next results are proved in the same way as in \cite[Sections 6 and 7]{DHMSS}. 

\begin{definition}\label{d:linear-multi}
We let $l^{\pm}_j: \pi_0^\pm \to \pi_0^{\perp_0}$ be the maps whose graphs describe the pages of the open book $\bS$. 
\end{definition}
Note that all of them must vanish on $V = V (\bS)$, the $(m-1)$-dimensional spine of $\bS$. There are $N^+ \geq 1$ functions $l_j^+$ and $N^- \ge 1$ functions $l_j^-$, with $N^+ + N^- \leq 2Q$, and with the possibility that $N^+ \neq N^-$. The key point of this section is that over the two halves of the ``good'' region $U_{4 \mathcal{W}}$, namely $U^\pm_{4\mathcal{W}}$ we will be able to select, for each map $u_j^\pm$ in the collection of maps describing $u$, some linear map $l^\pm_{h^\pm (j)}$ for which the $L^2$ norm of $w_j^\pm := u_j^\pm - l^\pm_{h^\pm (j)}$ can be estimated in terms of the excess with respect to $\bS$, rather than the excess with respect to $\pi_0$. The proof is verbatim that given in \cite[Section 6]{DHMSS} for the corresponding estimate in that situation and it is therefore omitted. 

\begin{theorem}[Improved $L^2$ estimates]\label{thm:graph v2}
Let $T$, $\Sigma$, $\bS$ and $\pi_0$ be as in Theorem \ref{thm:graph_v1}. Let $u$ be the corresponding map, and let $l=\{l^\pm_j\}_{j=1}^{N^\pm}$ be the maps of Definition \ref{d:linear-multi}. There are a geometric constant $C$ and two selection functions $h^\pm:j \in \{1,\ldots,Q\}\mapsto h^\pm(j) \in \{1, \ldots, N^\pm\}$ such that if $\tilde l^\pm_j := l^\pm_{h^\pm (j)}$ and
\begin{equation} \label{e:difference function}
    w^\pm_j := u^\pm_{j} - \tilde l^\pm_{j}\, ,
\end{equation}
then
\begin{align}\label{eqn:Linfty estimate excess}
	 	\sup_{\zeta = (t,y) \in U_{3 \mathcal W}}|t|^{\frac{m}{2}+1}\left( |t|^{-1} \abs{w(\zeta)} + \abs{Dw(\zeta)}+ |t|^{\sfrac12} [Dw]_{\sfrac12}(\zeta) \right)&\le C\, (\mathbf{E}(T,\bS,0,1)^{\sfrac12} + \bA)\,,\\ \label{L2 estimate excess}
 \sum_{\pm}\sum_{j=1}^Q \int_{U_{3\mathcal W}^\pm}    (|w_{j}^\pm (z)|^2 + \abs{x}^2 \abs{Dw^\pm_{j}(z)}^2) \, dz &\le C\, (\mathbf E(T,\bS,0,1) + \bA^2) \,,
\end{align}
where, for $z \in \pi_0$, $|x|$ denotes, as usual, the distance of $z$ from $V$.
\end{theorem}

\section{Proof of Proposition \ref{p:decay-2}: reparametrization on the new book} \label{s:reparametrization}

In this and the next section we assume that the parameters $\eta$ and $\tau$ defining the Whitney decomposition used in the previous section are fixed. The two selection functions $h^\pm$ and the corresponding linear maps $\tilde{l}^\pm_j$ of Theorem \ref{thm:graph v2} identify a new open book $\tilde \bS \subset \bS$ as follows 

\begin{definition}\label{d:new-cone}
We define
\begin{equation}\label{e:new_book}
\tilde{\bS} = \bigcup_{\pm}\bigcup_{j=1}^{N^\pm} ({\rm id} + \tilde l_j^\pm)(\pi_0^\pm) =: \bigcup_{\pm}\bigcup_{j=1}^{N^\pm} \tilde\bH_{j}^\pm\,.
\end{equation}
\end{definition}

\begin{remark} Observe that $\tilde{\bS}$ can be a proper subset of $\bS$. However it certainly contains at least two pages, one on the left and one on the right. 
\end{remark}

In this section we reparametrize a large portion of the current $T$ as graph over the pages of $\tilde{\bS}$. The advantage of $\tilde{\bS}$ over $\bS$ is that we have ``thrown away useless pages'', i.e. pages of $\bS$ which were not really close to $\spt (T)$. In particular the conclusions of this section will be used to prove Lemma \ref{l:compare-books}.

\subsection{Reparametrizing over \texorpdfstring{$\tilde{\bS}$}{tildeS}}\label{ss:reparametrization-books} By Theorem \ref{thm:graph_v1}(ii), in the region $R_{\mathcal W}$ the current $T$ coincides with the graph $\mathbf{G}_v$ of the special $Q$-valued function $v = u + \Psi(\cdot + u)$ over the domains $U_{4\mathcal W}^\pm$. Recall that on each domain $U_{4\mathcal W}^\pm$ of $\pi_0^\pm$ the function $u$ is canonically decomposed into $C^{1,\sfrac{1}{2}}$ functions $u_{j}^\pm$, and, since each domain $U_{4\mathcal W}^\pm$ is connected, the sign function $\varepsilon_u$ is constant on each of them. We will then simply reparametrize the graph of each map $u_j^\pm$ over the domain $U_{4\mathcal{W}}^\pm$ as the graph of a map $\tilde u_j^\pm$ over a subset of the corresponding page $\tilde{\bH}^\pm_j$. Observe that because the orientation of the graph of $u$ is constant on each $U_{4\mathcal{W}^\pm}$, so is the orientation of the graph of $\tilde{u}$ in order for the map $\tilde v = \tilde u + \Psi (\cdot + \tilde u)$ to describe the same current. In particular in this case, with a slight abuse of notation, we can omit to specify such orientation, and the corresponding sheets will be denoted by $\bG_{\tilde{v}_j^\pm}$, while the sum of them will be denoted by $\bG_{\tilde{\bS}} (\tilde{v})$. This is discussed in subsection \ref{ss:reparametrization-good}.

In subsection \ref{ss:reparametrization-bad}. we will instead aim at reaching a similar parametrization for $T$ over the ``bad'' domain $R_{\mathcal{B}}$. Taking advantage of the smallness of planar excess at the scale of each bad cube $L \in \mathcal B$, and of the existence of points of density at least $Q$ in the current at that scale, we may still define Lipschitz approximations $u_L$ on suitable planes $\hat\pi_L$ satisfying the estimates of Proposition \ref{p:L2Alm}. Of course, as specified in Remark \ref{r:special-multi-functions}, for each such function $u_L$ we also have a canonical selection by Lipschitz maps. After carefully estimating the angle between the cone $\tilde \bS$ and the plane $\pi_L$, we will be able to reparametrize the portion of the current described by the graphs of the $v_L = u_L + \Psi(\cdot + u_L)$ over the varying domains in $\hat\pi_L$ for $L \in \mathcal B$ with the union of graphs of functions $\tilde v_{j}^\pm$ defined over suitable domains of $\tilde \bH_j^\pm$. This time we do not have a ``sign function'' which is locally constant on the regions on the left and on the right of the spine. On the other hand, every point $q = \zeta + \tilde{v}^\pm_j (\zeta)$ with $\zeta \in \tilde\bH^\pm_j$ can be rewritten as $z + (v_L)_j (z)$ for a suitable $L \in \mathcal B$ and $z\in \hat\pi_L$, and we orient the approximate tangent to the graph of $\tilde{v}^\pm_j$ at $q$ positively (with respect to the orientation of $\tilde\bH^\pm_j$) if $\varepsilon_{v_L} (z) = 1$ and negatively (with respect to the orientation of $\tilde\bH^\pm_j$) if $\varepsilon_{v_L} (z) = -1$. This defines an integer rectifiable current $\bG_{\tilde{v}^\pm_j}$ and the sum of all these will be denoted by $\bG_{\tilde{\bS}} (\tilde{v})$.

\subsection{Reparametrization over the good domain} \label{ss:reparametrization-good}
As in \cite[Section 7]{DHMSS}, a reparametrization of the graph of $v$ on the slightly smaller good region $R_{2\mathcal W}$ over the new book $\tilde{\bS}$ follows from Theorem \ref{thm:graph v2}. The proof can be taken verbatim from \cite{DHMSS} and it is therefore omitted. In what follows, we adopt, for a given domain $U \subset [0,\infty) \times V$, the notation
\[
\tilde U^\pm_j := \left\lbrace (x,y) \in \tilde \bH^\pm_j \, \colon \, (|x|,y) \in U \right\rbrace\,,
\]
where, since each $\tilde\bH_j^\pm \subset \bS \subset T_0\Sigma$, we can use $(x,y)$ as a short-hand notation to identify the point $(x,y,0)$. We will also adopt the following convention in order to further ease the notation. The reparametrization algorithm will produce precisely $p=2Q$ functions, $Q$ defined on the half planes $\tilde\bH_j^+$ projecting on $\pi_0^+$ and the other $Q$ defined on the half planes $\tilde\bH_j^-$ projecting on $\pi_0^-$. We agree to denote such functions as $\tilde u_j^\pm$, where $j \in \{1,\ldots,Q\}$, and to let $\tilde \bH^\pm_j$ denote the half plane containing its domain. In particular, if two functions, say $\tilde u^+_j$ and $\tilde u^+_{j'}$, are defined on the same half plane, then that half plane will be denoted $\tilde\bH^+_{j}$ or $\tilde \bH^+_{j'}$ depending on whether it is thought of as the domain of $\tilde u^+_j$ or $\tilde u^+_{j'}$. At the price of possibly having $\tilde\bH^\pm_j = \tilde\bH^\pm_{j'}$ for some $j \neq j'$, we can think from now on that
\[
\tilde\bS = \bigcup_{\pm} \bigcup_{j=1}^Q \tilde\bH^\pm_j\,.
\]

\begin{corollary}[Reparametrization on Good cubes] \label{cor:reparametrization}
    Let $T, \Sigma, \bS$, and $\pi_0$ be as in Theorem \ref{thm:graph v2}, and let $\tilde{\bS}$ be the open book in Definition \ref{d:new-cone}. There are $2Q$ functions $\tilde u^\pm_{j} \colon (\tilde U_{2 \mathcal W})^\pm_{j} \subset \tilde{\bH}^\pm_{j} \to (\tilde{\bH}^\pm_{j})^{\perp_0}$ (with $j\in\{1,\ldots,Q\}$) of class $C^{1,\frac12}$ with the following properties. The estimate 
    \begin{equation} \label{e:rep_Holder}
        \| \tilde u \|_{C_*^{1,\frac12}(\tilde U_{2\mathcal W})} \leq C (\bE (T, \bS, 0, 1)^{\sfrac{1}{2}} + \bA)
    \end{equation}
    holds. Moreover, if we set 
    \begin{equation} \label{e:rep_up to manifold}
        \tilde v^\pm_{j}(z) := \tilde u^\pm_{j}(z) + \Psi(z + \tilde u^\pm_{j}(z))\,,
    \end{equation}
    there is an appropriate choice of the orientation of the graphs of $\tilde{v}_j^\pm$ so that, following the notation of Section \ref{ss:reparametrization-books},
    \begin{equation}\label{e:rep_parametrization}
        T \mres R_{\mathcal W} = {\bf G}_{\tilde{\bS}}(\tilde v) \mres R_{\mathcal W}\, .
    \end{equation}
Finally,
    \begin{equation}\label{e:comparison-tilde-u-w}
    \begin{split}
\int_{(\tilde U_{2\mathcal W})^\pm_{j} } (|\tilde{u}^\pm_{j}|^2 + |x|^2 |D \tilde{u}^\pm_{j}|^2) 
&\leq C \int_{U_{3 \mathcal W}^\pm } (|w^\pm_{j}|^2 + |x|^2 |Dw^\pm_{j}|^2)\\
& \leq C \, (\bE(T,\bS,0,1) + \bA^2)\,,
\end{split}
\end{equation}
where $w$ is the multifunction over $\pi_0$ defined in \eqref{e:difference function}.
\end{corollary}

\subsection{Multivalued approximation in bad cubes}\label{ss:reparametrization-bad}
Here we show that over the ``bad'' Whitney region $\cB$ the current $T$ can still be approximated with a multivalued graph over $\tilde{\bS}$, with good estimates, in the following sense.

\begin{remark}[Graphicality in bad cubes]\label{rem:bad_graph}
By virtue of Theorem \ref{thm:graph_v1}(iii) and \eqref{eq:W3}, as soon as $\tau$ and $\bar\eps$ are chosen sufficiently small, for any cube $L \in \mathcal B$ we may apply Proposition \ref{p:L2Alm} in the ball $\bB_{20\bar M d_L}(y_L)$ and conclude from the Propagation Lemma \ref{l:propagation} the existence of:
\begin{itemize}
    \item[$\bullet$] a plane $\hat\pi_L$ so that $V \subset \hat \pi_L$, $2\,\bar\bE(T,y_L,20\bar M d_L) > \bE(T,\hat\pi_L,y_L,20\bar M d_L)$, as well as $\bE(T,\hat\pi_L,y_L,20\bar M d_L) \leq C\,\bar\bE_L$,
    \item[$\bullet$]  a closed set $K_L \subset B_L := \bB_{10\bar M d_L}(y_L) \cap (y_L+\hat\pi_L)$,
    \item[$\bullet$] and a function $
u_L \colon B_L \to \mathscr{A}_Q(\pi_L^{\perp_0})
$
\end{itemize}
such that the corresponding map $v_L = u_L + \Psi (\cdot + u_L)$ satisfies (the rescaled version of) \eqref{e:inclusion}-\eqref{e:W12}, which we record here for future reference, keeping in mind that $(\bar M d_L)^2\bA^2 \leq \bar \eta\,\bar \bE_L$ as shown in the proof of Theorem \ref{thm:graph_v1}:
\begin{align}\label{e:inclusion cube}
&\spt(\bG_{v_L}) \subset \Sigma\,, \\ \label{e:lip_osc_est cube}
&\Lip (v_L)\leq C\bar\bE_L^\gamma \qquad \mbox{and} \qquad {\rm osc}(v_L) \leq C \mathbf{h}(T, \bC_{15 \bar M d_L}(y_L,\hat\pi_L)) + C \, d_L \, \bar{\bE}_L^{\sfrac12}\,,\\\label{e:graph_current cube}
&\bG_{v_L} \res (K_L\times \hat\pi_L^\perp )= T\mres (K_L\times \hat\pi_L^\perp ) \cap \bB_{15\bar M d_L}(y_L) \;\modp\,,\\ \label{e:volume_estimate cube}
&|B_L\setminus K_L|\leq \|T\| (((B_L \setminus K_L)\times \hat\pi_L^\perp)\cap \bB_{15\bar M d_L}(y_L)) \leq C\, d_L^{m}\, \bar{\bE}_L^{1+\gamma}\,,\\
& \Abs{\|T\|(\bC_{10\bar M d_L}(y_L,\hat\pi_L)\cap \bB_{15\bar M d_L}(y_L)) - Q |B_L| - \frac{1}{2} \int_{B_{L}} \abs{Dv_L}^2} \leq C\,d_L^m\,\bar{\bE}_L^{1+\gamma}\,,\label{e:Dirichlet cube} \\
& d_L^{-2} \|v_L\|^2_{L^\infty(B_L)} + d_L^{-(m+2)} \int_{B_L} \left(|v_L|^2 + |z-y_L|^2\,|D v_L|^2\, dz\right) \leq C\, \bar{\bE}_L\, .\label{e:W12 cube}\
\end{align}
\end{remark}

% {\color{red} \begin{remark}(Tilting of the minimizing plane compared to $\bS$)
% The aim of this remark is to show that there is a half-plane $\bH$ in $\bS$ such that the angle between the plane $\pi_L$ and $\bH$ is proportional to the excess i.e. 
% \[ \beta^2(\pi_L, \bH) \lesssim \bar{\bE}_L + \bE_{L'} + \bA^2 \]
% where $L \preceq L'$ is an adjacent cube of $L$.

% \end{remark}}

As in Corollary \ref{cor:reparametrization}, we can reparametrize the function $u_L$ on the cone $\tilde \bS$ outside of a small region around the spine of $\tilde \bS$. Recall the notation $R_{\mathcal B}$ introduced in \eqref{rotation of bad Whitney 2}.

% \begin{proposition}[Parametrization matching on good and bad cubes] \label{p:Almgren conical}
% There exists a constant $C=C(m,n,Q)$ with the following property. For every cube $L \in \mathcal B$ there exist half-planes $\bH_L^{\pm} \subset \tilde \bS$, closed sets $K_L^\pm \subset \Omega_L^\pm := \bB_{\frac{\bar M d_L}{4}}(y_L) \cap \bH_L^\pm \setminus \mathcal C_L (V)$, and functions $u_L^\pm \colon \Omega_L^\pm \to \mathscr{A}_Q((\pi_L^\pm)^{\perp_0})$ (where $\pi_L^\pm$ is the plane containing $\bH_L^\pm$) such that, setting
% \begin{equation} \label{lifting up cube}
% v_L^\pm(z) := \left( \sum_{i=1}^Q \a{(v_L^\pm)_i(z)}, \eps_{u_L^\pm}(z) \right)\,, \qquad (v_L^\pm)_i(z) := (u_L^\pm)_i(z) + \Psi(z+(u_L^\pm)_i(z))\,,
% \end{equation}
% then
% \begin{align}\label{e:inclusion cube tilted}
% &\spt(\bG_{v_L^\pm}) \subset \Sigma\,, \\ \label{e:lip_osc_est cube tilted}
% &\Lip (v^\pm_L)\leq C(\bar \bE_L + \bA_L^2)^\gamma\,,\\\label{e:graph_current cube tilted}
% &\bG_{v_L^\pm} \res (K_L^\pm\times (\pi_L^\pm)^\perp \setminus \mathcal C_L (V))= T\mres (K_L^\pm\times (\pi_L^\pm)^\perp \setminus \mathcal C_L (V)) \;\modp\,,\\ \label{e:volume_estimate cube}
% &\sup_\pm|\Omega_L^\pm\setminus K_L^\pm|\leq \sup_\pm\|T\| ((\Omega_L^\pm\setminus K_L^\pm)\times (\pi_L^\pm)^\perp) \leq C d_L^{m} (\bar \bE_L + \bA^2_L)^{1+\gamma}\,,\\
% & \sum_{\pm}\int_{\Omega_L^\pm} \left(|v_L^\pm|^2 + |x|^2\,|D v^\pm_L|^2\right) \, dz \leq C\,d_L^{m+2} (\bar \bE_L + \bA^2_L)\,.\label{e:W12 cube}\
% \end{align}
% \end{proposition}

\begin{corollary} [Reparametrization on Bad cubes]\label{cor:reparametrization on Bad cubes}
    Let $T, \Sigma, \bS$, and $\pi_0$ be as in Theorem \ref{thm:graph v2}, and let $\tilde{\bS}$ be the open book in Definition \ref{d:new-cone}. There exist $2Q$ functions $\tilde u^\pm_{j} \colon  R_{ \cB}\cap \tilde{\bH}^\pm_j \to (\tilde{\bH}^\pm_{j})^{\perp_0}$ and an appropriate choice of orientations $\varepsilon^\pm_j$ with the following properties. The estimates
    \begin{equation} \label{e:rep_Lipschtiz_bad}
       \Lip( \tilde u^\pm_j ) \leq C \tau^{2\gamma}\,, \qquad \qquad \|\tilde u^\pm_j \|_{L^\infty} \leq C \tau^2
    \end{equation}
    holds. Moreover, if we set 
    \begin{equation} \label{e:rep_up to manifold_bad}
        \tilde v^\pm_{j}(z) := \tilde u^\pm_{j}(z) + \Psi(z + \tilde u^\pm_{j}(z))\,,
    \end{equation}
    and denote by $\bG_{\tilde{\bS}} (\tilde{v})$ the current described in Section \ref{ss:reparametrization-books}, then 
    \begin{equation}\label{e:rep_parametrization_bad}
    \left| \int_{R_{\mathcal B}} \dist^2(q,V) \, d\|T\|(q) - \int_{R_{\mathcal B}} \dist^2(q,V) \, d\|\bG_{\tilde \bS}(\tilde v)\|(q) \right| \le \frac{C}{\eta} \,\bE(T, \bS, 0,1)
    \end{equation}
  and 
    \begin{equation}\label{e:W12_rep_bad}
\int_{R_{\mathcal{B}}\cap \tilde{\bH}^\pm_j}  (|\tilde{u}^\pm_{j}|^2 + |x|^2 |\nabla \tilde{u}^\pm_{j}|^2)  \leq \frac{C}{\eta} \, \bE(T,\bS,0,1) \,.
\end{equation}
\end{corollary}

Before coming to its proof we register the following further corollary. 

\begin{corollary}\label{cor:comparrison excess}
    Let $T, \Sigma, \bS$, and $\pi_0$ be as in Theorem \ref{thm:graph v2}, and let $\tilde{\bS}$ be the open book in Definition \ref{d:new-cone}, then we have 
\begin{equation}\label{eq:tildeexcesvsexcess}
    r^{m+2}\mathbb{E}(T,\tilde\bS, 0, r)\leq \frac{C}{\eta^{\sfrac32}}\,\bE(T,\bS,0,1) + C \bA^2\,,\qquad 0<r\leq\frac12\,.
\end{equation}    
Moreover, if $\bar \eta$ and $\bar \varepsilon$ are sufficiently small, compared to $\eta$, then 
\begin{equation}\label{eq:angle_bound_new_book1}
         C (\eta)^{-1} \beta_{\pi_0}^2(\tilde \bS) \leq  \bE(T,\pi_0,0,1) \leq C (\eta) \beta_{\max}^2 (\tilde \bS) 
         \leq C (\eta) \beta_{\pi_0}^2 (\tilde \bS)\, ,
    \end{equation}
where the constant $C (\eta)$ depends additionally only upon $Q$, $m$, and $n$. Combining this with \eqref{eq:angle_bound} we get
    \begin{equation}\label{eq:angle_bound_new_book}
        \beta_{\pi_0}(\bS)\leq C(\eta)\,\beta_{\pi_0}(\tilde \bS)\,.
    \end{equation}
\end{corollary}

\begin{proof}[Proof of Corollary \ref{cor:comparrison excess}] The proof of \eqref{eq:tildeexcesvsexcess} follows as in \cite{DHMSS} using \eqref{e:L^2 estimate}, \eqref{L2 estimate excess}, \eqref{e:rep_parametrization_bad}, and \eqref{e:W12_rep_bad}. 

To obtain \eqref{eq:angle_bound_new_book1} it is sufficient to check that the new book $\tilde{\bS}$ satisfies the assumptions of Lemma \ref{l:angle-bound} at scale $1/2$. As a consequence of \eqref{eq:tildeexcesvsexcess}, we have an upper bound of the double-sided excess between $T$ and $\tilde\bS$ at scale $1/2$ in terms of the control quantity $\bE(T,\bS,0,1)+ \bA^2$. Since $\bar\bE(T,0,\frac12) \geq \frac12 \bE(T,\pi_0,0,1) \geq \frac14 (\bar\eta^{-1} \bE(T,\bS,0,1) + \bar\varepsilon^{-1}\bA^2)$, this is sufficient to prove Assumption \ref{ass:decay plane -1}(ii) at the scale $1/2$ and conclude.
\end{proof}

\begin{proof}[Proof of Corollary \ref{cor:reparametrization on Bad cubes}]

\noindent \emph{Step 1.} Here we estimate the tilt of the plane $\hat\pi_L$ at the scale of cubes $L\in\mathcal B$ with respect to the reference plane $\pi_0$. Let $L$ be any cube in $\mathcal B$, and let $L' \in \mathcal W$ be its father. As noticed in the proof of Theorem \ref{thm:graph_v1}, Lemma \ref{l:propagation} can be applied at the scale of $L'$, so that, in particular,
\begin{equation} \label{eq:father-son excess comparison}
    \bar\bE_L \leq 2 \cdot 2^{m+2}\,\bar\bE_{L'} \,, \qquad  \bar \bE_{L'} \leq 2 \, \bar\bE_L\,.
\end{equation}
This implies that we can estimate
\begin{equation} \label{eq: tilt of planes}
\begin{split}
    |\hat \pi_L - \hat \pi_{L'}|^2 &\leq C \left(\bE^{no}(T,\hat \pi_L, \bB_{\bar M d_L/2}(y_L)) + \bE^{no}(T,\hat \pi_{L'}, \bB_{\bar M d_{L'}/2}(y_{L'}))\right) \\
    & \leq C \, \bar\bE_L\,,
\end{split}
\end{equation}
where $\bE^{no}$ is the unoriented tilt excess defined in Definition \ref{d:E-cessi}, and where we have used the classical tilt-excess inequality together with the condition $(\bar M d)^2 \bA^2 \leq \bar \bE$ at the scales of $L$ and $L'$ and \eqref{eq:father-son excess comparison}.

Next, denoting by $\beta(\pi,\bH) = \dist_{\mathcal H}(\pi \cap \bB_1, \bH \cap \bB_1)$ the ``angle'' between a plane $\pi$ and a half plane $\bH$, we claim that: 
\begin{equation}\label{eq:angle_est_L}
    \beta (\hat \pi_L, \bH)^2 \le C(\bar{\bE}_L + \beta^2_{\max}(\bS) )\, \qquad \mbox{for every $\bH \subset \bS$}\,.
\end{equation}

To see this, first apply Lemma \ref{l:propagation} at the scale of $L$ and with $\rho$ sufficiently small to conclude that $\Theta(T,q) < Q$ for every $q \in \spt (T) \cap \bC_{\bar M d_{L}/8}(y_{L},\hat\pi_{L}) \cap \bB_{7\bar M d_{L}/8}(y_{L}) \setminus B_{d_{L}}(V)$. In particular, if we let $q \in \hat\pi_{L}$ denote any of the two points with $\mathbf{p}_V(q)=y_{L}$ and $\dist(q,V) = \frac{\bar M d_{L}}{16}$ then for every half plane $\tilde\bH_j^\pm \subset \tilde \bS$ 
\[
\bB_{4d_L}(q) \cap \tilde\bH_j^\pm \subset (\tilde U_{2\mathcal W})_j^\pm \cap \bC_{\bar M d_{L}/8}(y_{L},\hat\pi_{L}) \cap \bB_{7\bar M d_{L}/8}(y_{L}) \setminus B_{d_{L}}(V)\,,
\]
and, as a consequence of Proposition \ref{p:L2Alm-piece-3}, 
\[
B_{2d_L}(q,\hat\pi_L) := \bB_{2d_L}(q) \cap (q + \hat \pi_L) \subset K_L\,,
\]
so that
\[
\bG_{v_L} \mres \bC_{2d_L}(q,\hat\pi_L) = T \mres \bC_{2d_L}(q,\hat\pi_L) \cap \bB_{7 \bar M d_L /8} (y_L)= \bG_{\tilde \bS}(\tilde v)\mres \bC_{2d_L}(q,\hat\pi_L)\,.
\]
As a further consequence of Proposition \ref{p:L2Alm-piece-3}, there is a $C^{1,\sfrac12}$-selection for $\left.u_L\right|_{B_{2d_L}(q,\hat\pi_L)}$, which we denote $(u_L)_1 \leq \ldots \leq (u_L)_Q$. Now, by standard arguments one immediately concludes that there exists $\bH_\star \subset \bS$ such that
\[
\mathcal{L}^m \left( \underbrace{ \left\lbrace z \in B_{2d_L}(q,\hat\pi_L) \, \colon \, \dist(z + (v_L)_1(z),\bS) = \dist(z+(v_L)_1(z), \bH_\star) \right\rbrace}_{=:O_L} \right) \geq c\,d_L^m\,,
\]
where $c=c(m,Q)$. For every $z \in O_L$, we then have
\[
\beta(\hat\pi_L,\bH_\star) \, d_L \leq \dist(z,\bH_\star) \leq \dist(z + (v_L)_1(z),\bS) + |(v_L)_1(z)|\,,
\]
so that after squaring and integrating over $O_L$ we reach
\[
\beta(\hat\pi_L,\bH_\star)^2\,d_L^{m+2} \leq C \left( \int_{\bB_{\bar M d_{L'}}(y_{L'})} \dist^2(\cdot,\bS) \, d\|T\| + \int_{B_L} |v_L|^2 \,dz \right) \,,
\]
which in turn implies, by \eqref{e:W12 cube}, $L' \in \mathcal W$ (whence $\bE_{L'} \leq \eta \, \bar\bE_{L'}$), and \eqref{eq:father-son excess comparison} that
\[
\beta(\hat\pi_L,\bH_\star)^2 \leq C\,\bar\bE_L\,,
\]
and \eqref{eq:angle_est_L} follows by triangle inequality.

With \eqref{eq:angle_est_L} at our disposal, and recalling that $\beta_{\pi_0}(\bS) \leq C\,\beta_{\max}(\bS)$ as a consequence of Lemma \ref{l:angle-bound}, we may conclude that 
\begin{equation} \label{eq:angle_estimate L with ref}
    \dist_{\mathcal H}(\hat\pi_L\cap\bB_1,\pi_0\cap\bB_1) \leq C\,(\bar\bE_L + \beta^2_{\max}(\bS)) \leq C\,(\tau^2+\bar\varepsilon)\,.
\end{equation}
In turn, the $L^\infty$ estimate \eqref{e:W12 cube}, the condition $\bar\bE_L \leq \eta^{-1} \bE_L$ for a cube $L \in \mathcal B$, and \eqref{eq:angle_estimate L with ref} imply that for a suitable choice of the geometric constant $C_\star$ in \eqref{rotation of bad Whitney 2}, for each $L\in \mathcal{B}$ the graph $\bG_{v_L} \res R_{\mathcal{B}}$ splits into two disjoint parts, one that projects on $\pi_0^+$ and the other on $\pi_0^-$, that is
\begin{equation} \label{graph decomposition pm}
\bG_{v_L} \res R_{\mathcal{B}} = \bG_{v_L^+}\res R_{\mathcal{B}} + \bG_{v_L^-}\res R_{\mathcal{B}}\,.
\end{equation}

\medskip

\noindent\emph{Step 2:} Here we exploit the conclusions drawn in Step 1 in order to produce the claimed reparametrization over $R_{\mathcal B}$. Let $L$ be any cube in $\mathcal B$. By \eqref{graph decomposition pm}, it will be sufficient to seek a reparametrization for the function $v_L^+$ over the half-planes $\tilde \bH_j^+$ projecting onto $\pi_0^+$, because the argument for the function $v_L^-$ is going to be the same. As noticed above, Proposition \ref{p:L2Alm-piece-3} implies that, for $q$ as in Step 1 which projects onto $\pi_0^+$, $B_{2d_L}(q)\subset K_L$ and therefore \[\bG_{v^+_L}\res \bC_{2d_L}(q, \hat\pi_L)= T\res\bC_{2d_L}(q, \hat\pi_L) \cap \bB_{7 \bar M d_L/8}(y_L)= \bG_{\bar{\bS}}(\tilde{v})\res \bC_{2d_L}(q, \hat\pi_L)\,.\]
Hence, we can consider the function $u_L^+$ so that $v_L^+ = u_L^+ + \Psi(\cdot + u_L^+)$, and, recalling from Remark \ref{r:special-multi-functions} that we have a fixed selection $u^+_L= \sum_{j=1}^Q \a{(u_L^+)_j}$ of Lipschitz functions $(u_L^+)_1 \le (u_L^+)_2 \le \ldots \le (u_L^+)_Q$, we proceed as follows. For every $j \in \{1,\ldots,Q\}$, we let $h^+(j)$ be the index such that the half-plane $\tilde \bH^+_{h^+(j)} \subset \tilde \bS$ hosts the domain of the function in $\tilde v$ which reparametrizes $(v_L^+)_j$ over $B_{2d_L}(q,\hat\pi_L)$. Then, we call $(\tilde{u}_L^+)_j$ the reparametrization of $(u_L^+)_j$ over $\tilde \bH^+_{h^+(j)}$. By \eqref{eq:angle_estimate L with ref} and the $L^\infty$ bound \eqref{e:W12 cube}, the domain of $(\tilde{u}_L^+)_j$ contains $D_{L,j}:=\tilde \bH^+_{h^+(j)} \cap \bB_{8 \bar M d_L}(y_L) \setminus B_{C(\tau^2+\bar\varepsilon)(\eta^{-1}\bE_L)^{\sfrac12}d_L} (V)$, and the currents associated to the graphs of $(u^+_L)_j$ and $(\tilde u^+_L)_j$ agree, provided the choice of the orientation of the tangent plane to the graph of $(\tilde{u}^+_L)_j$ is made following the algorithm detailed in Section \ref{ss:reparametrization-books}. Furthermore, using the estimates produced in Step 1 we can calculate
\[
\Lip((\tilde{u}^+_L)_j)\le C( \Lip((u^+_L)_j) + \beta(\hat\pi_L,\tilde\bH^+_{h^+(j)})) \leq C ( \bar\bE_L^\gamma + \bar\bE_L^{\sfrac12} )  \leq C (\tau^2+\bar\varepsilon)^\gamma \leq C \tau^{2\gamma} \,,
\]
as well as
\[
d_L^{-2}\|(\tilde u^+_L)_j\|^2_{L^\infty(D_{L,j})} \leq C(d_L^{-2} \|v_L\|^2_{L^\infty(B_L)} + \beta(\hat\pi_L, \tilde\bH^+_{h^+(j)})^2) \leq C \bar\bE_L \leq C \tau^2\,.
\]
By the above argument, for each $j \in \{1,\ldots,Q\}$ the page $\tilde \bH^+_j$ hosts a domain $D_j$ such that
\[
D_j \supset \bigcup_{L\in\mathcal B} \left\lbrace (x,y) \in \tilde \bH^+_j \, \colon \, \mbox{ $y \in \mathbf{p}_V({\rm int}(V))$ and $C (\tau^2+\bar\varepsilon) (\eta^{-1}\bE_L)^{\sfrac12}d_L \leq |x| < 5 \bar M d_L $ }\right\rbrace\,.
\]
In particular, since cubes $L \in \mathcal B$ have interiors with disjoint projection onto $V$, for a suitable choice of the geometric constant $C_\star$ in \eqref{rotation of bad Whitney 2} we can define a global function $\tilde u^+_j$ over $\tilde \bH^+_j \cap R_{\mathcal B}$.  We can now estimate, using that the balls $\bB_{\bar M d_i}(y_i)$ are disjoint:
\begin{align*}
   & \int_{R_{\mathcal B}\cap\tilde \bH_j^+} (|\tilde u_j^+|^2 + |x|^2|D\tilde u_j^+|^2) \, dz \leq \sum_i \int_{(\bB_{5 \bar M d_i} (y_i) \setminus B_{C_\star(\eta^{-1}\bE_i)^{\sfrac12}d_i} (V) ) \cap \tilde \bH_j^+}(|\tilde u_j^+|^2 + |x|^2|D\tilde u_j^+|^2) \, dz \\
   &\qquad \quad \leq \sum_i \int_{B_{L_i}} (|v_{L_i}|^2 + |z-y_i|^2 |Dv_{L_i}|^2 + d_i^2 \,\bar\bE_{L_i}) \, dz \\
   & \qquad \qquad \leq C \, \sum_i d_i^{m+2} \bar\bE_{L_i} \leq C \eta^{-1} \sum_i d_i^{m+2} \bE_{L_i} \leq C \eta^{-1} \bE(T,\bS,0,1)\,,
\end{align*}
thus completing the proof of \eqref{e:W12_rep_bad} for the part over $\pi_0^+$. Similarly, we obtain the missing estimate \eqref{e:rep_parametrization_bad} summing the errors in the region where $T$ does not agree with the graph $\bG_{\tilde \bS}(\tilde v)$ by means of \eqref{e:graph_current cube}, \eqref{e:volume_estimate cube}, and \eqref{e:W12 cube}.
\end{proof}

\subsection{Proof of Lemma \texorpdfstring{\ref{l:compare-books}}{compare-books}}

Fix $\gamma  > 0$, and let $T$ and $\Sigma$ be as in Assumption \ref{ass:everywhere}. Correspondingly, choose the parameters $\bar\varepsilon$ and $\bar\eta$ depending on $\gamma$ so that whenever Assumption \ref{ass:decay plane} holds for some open book $\bS$ and some plane $\pi_0$ with $\bar\varepsilon$ and $\bar\eta$ then one can conclude graphicality of $T$ over a suitable subset of $\tilde\bS \subset \bS$ up to distance $\sfrac{\gamma}{8}$ from $V (\bS)$. Now fix two open books $\bS$ and $\bS'$ in $\mathscr{B}(0)$. Since the Hausdorff distance between any two open books in $\mathscr{B}(0)$ is bounded by a universal constant $\bar C$, if $r \geq \gamma$ and $\bA^2+\mathbb{E}(T,\bS,0,1)+\mathbb{E}(T,\bS',0,r)\geq \bar\alpha(\gamma)$ then the conclusion in \eqref{e:compare-books} is trivially true with the choice $C(\gamma) = \bar C^2 \, \bar\alpha(\gamma)^{-1}$. Hence, we assume
\begin{equation} \label{the-interesting-case}
    \bA^2+\mathbb{E}(T,\bS,0,1)+\mathbb{E}(T,\bS',0,r) < \bar\alpha(\gamma)\,,
\end{equation}
where $\bar\alpha(\gamma)$ will be chosen momentarily, and we consider the following two cases:
\begin{itemize}
    \item[(a)] Assumption \ref{ass:decay plane} holds for $\bS$ and a plane $\pi_0$ with $\bar\varepsilon$ and $\bar\eta$ as specified above;
    \item[(b)] Assumption \ref{ass:decay plane} for $\bS$ with $\bar\varepsilon$ and $\bar\eta$ as above fails. 
\end{itemize}
In case (a), after denoting $\tilde\bS \subset \bS$ the open book in Definition \ref{d:new-cone}, we estimate for $r \ge \gamma$
\[
\begin{split}
d_{\mathcal H}(\tilde \bS \cap \bB_1 , \bS' \cap \bB_1)^2 &= r^{-2}\,d_{\mathcal H}(\tilde \bS \cap \bB_r , \bS' \cap \bB_r)^2 \leq C (\mathbb{E}(T,\tilde\bS,0,r) + \mathbb{E}(T, \bS',0,r)) \\
& \leq C\,\gamma^{-(m+2)} (\bA^2 + \mathbb{E}(T,\bS,0,1) + \mathbb{E}(T, \bS',0,r))\,,
\end{split}
\]
where in the last inequality we have used \eqref{eq:tildeexcesvsexcess}. On the other hand, it is immediately seen that $d_{\mathcal H}(\tilde\bS\cap\bB_1,\bS\cap\bB_1)^2 \leq C (\mathbb{E}(T,\bS,0,1))$. This proves \eqref{e:compare-books} when (a) holds, and \eqref{e:almost-monotonicity-book} is also a consequence of \eqref{eq:tildeexcesvsexcess}.

Now, suppose we are in case (b). Of course, we can assume that the failure of Assumption \ref{ass:decay plane} is due to the failure of Assumption \ref{ass:decay plane -1}, for otherwise Lemma \ref{lem.propagation1} would imply that case (a) holds for the current $T'=(\lambda_{0,\frac12})_\sharp T$ (possibly upon reducing the values of $\bar\varepsilon$ and $\bar\eta$). Hence, in this case at least one of the following holds:
\begin{itemize}
    \item[(b1)] $\bA^2 > \bar\varepsilon\, \bar\bE(T,0,1)$,
    \item[(b2)] $\bar\bE(T,0,1) > \bar\varepsilon$,
    \item[(b3)] $\mathbb{E}(T,\bS,0,1) >\bar\eta \,\bar\bE(T,0,1)$,
\end{itemize}

If (b2) holds, then \eqref{the-interesting-case} for $\bar\alpha(\gamma) \ll \bar\varepsilon$ implies a lower bound on the opening angle $\beta_{\max}(\bS)$. Then by choosing $\bar\alpha (\gamma)$ much smaller, if necessary, depending on $\bar\varepsilon$, we would once again deduce graphicality of $T$ over a suitable subset $\tilde\bS \subset \bS$ due to \cite[Corollary 6.6]{DHMSS} and conclude as in case (a).

In case (b1) or (b3) holds, \eqref{e:compare-books} is a simple consequence of the fact that \[d_{\mathcal{H}} (\bS \cap \bB_1, \pi_0 \cap \bB_1)^2 \leq C \, \left(\bar\bE (T,0,1) + \bA^2 \right)\] 
when $\bS$ is optimal for the conical excess and $\pi_0$ is optimal for the planar excess in $\bB_1$, as a consequence of the height bound Lemma \ref{Linfty-L2}. Analogously, Lemma \ref{Linfty-L2} also implies that
\[
\mathbb{E}(T,\bS,0,\sfrac12) \leq C \left( \bar\bE(T,0,1) + \bA^2 \right)\,,
\]
which in turn implies \eqref{e:almost-monotonicity-book} if (b1) or (b3) holds. \qed

\section{Proof of Proposition \texorpdfstring{\ref{p:decay-2}}{decay2}: Simon type estimates}\label{sec:Hardt-Simon}
In this section we use one of the crucial ideas of Simon's work \cite{Simon} (cf. also \cite{CoEdSp}): close to points of high density, the monotonicity formula gives an improved $L^2$ estimate, see \eqref{e:Hardt-Simon}; in particular, such points of high density are bound to lie close to the spine $V$ at the scale of the excess $\bE$. An analogous estimate was proved in \cite{DHMSS}, but here the situation is more subtle as we have to take into account that the ``opening angle'' of $\bS$ might be relatively small.

\begin{theorem}\label{t:Hardt-Simon-main}
 There are positive constants $C $, $\eta_{9}$, and $\varepsilon_{9}$ depending upon $(m,n,p)$ such that if $T,\Sigma, \bS, \pi_0$ are as in Assumptions \ref{ass:everywhere} and \ref{ass:decay plane} with $\bar\eps <\eps_{9}$ and $\bar \eta < \eta_{9}$ then the following conclusion holds. Assume that
\begin{itemize}
\item[(a)] $\tilde{\bS}$ denotes the open book introduced in Definition \ref{d:new-cone} 
\item[(b)] and $q_0=(x_0,y_0)\in (V^\perp \times V)\cap \bB_{\sfrac{1}{4}}$ is a point with $\Theta_T (q_0) \geq Q$.
\end{itemize}
Then
\begin{equation} \label{e:Hardt-Simon}
\beta_{\pi_0}(\bS)^2\, |x_0|^2+|x_0^\perp|^2 + \int_{\bB_{1/4}} \frac{{\rm dist}^2\, (q-q_0, \tilde{\bS})}{|q-q_0|^{m+\frac74}}\, d\|T\| (q) \leq C (\bE(T, \bS, 0,1) + \bA^2)\, ,
\end{equation}
where $x_0^\perp = \mathbf{p}_{\pi_0}^\perp (x_0) = \mathbf{p}_{\pi_0}^\perp (q_0)$.
\end{theorem}

The proof follows the same argument as in \cite[Section 8]{DHMSS}, however we need to suitably modify \cite[Proposition 8.4]{DHMSS} to take into account the presence of the bad Whitney region $\mathcal B$. This will be done by taking advantage of the fact that in this region the planar and conical excess are comparable, and using once more the multivalued approximation to estimate the required errors. 

\subsection{Consequences of the monotonicity formula} We start with an improved version of the first part of \cite[Lemma 8.2]{DHMSS}. More precisely, the bound \eqref{e.h_k monotonicity} differs from the corresponding one in \cite[Lemma 8.1]{DHMSS} in the dependence upon $\bA$. The proof is given in the appendix. 

\begin{lemma}\label{l:monot-with-A^2}
Let $T$ and $\Sigma$ be as in Assumption \ref{ass:everywhere}, and assume that $g (q) = |q|^k \,\hat g (\frac{q}{|q|})$ for some $k\geq 1$ and some Lipschitz non-negative function $\hat g$ on the unit sphere. Then, for every $2>\alpha>0$ and $R \leq 1$ we have
\begin{align}\label{e.h_k monotonicity}
	\frac{\alpha}{2} \int_{\bB_{R}} \frac{g^2 (q)}{\abs{q}^{m  + 2k -\alpha}}\, d\norm{T} (q) \leq & \frac{m+2k}{R^{m+2k-\alpha}} \int_{\bB_{R}} g^2 \,d\norm{T} + \frac{2}{\alpha} \int_{\bB_{R}}  \frac{\abs{\nabla g (q)}^2\abs{q^\perp}^2}{\abs{q}^{m+2k-\alpha}}\, d\norm{T} (q)\nonumber\\
	&+ C \bA^2 \|\hat g\|_\infty^2 \frac{\norm{T}(\bB_{R})}{R^{m-\alpha}}\,,
\end{align}
where $q^\perp := q - \mathbf{p}_{\vec{T}}(q)$ at $\Ha^m$-a.e. $q \in \spt(T)$ (here, $\mathbf{p}_{\vec{T}} = \mathbf{p}_{\vec{T}(q)}$ is the orthogonal projection onto ${\rm span}(\vec{T}(q))$).
\end{lemma}

As a simple corollary we then conclude the following.

\begin{corollary}\label{c:first-Hardt-Simon}
Let $T$ and $\Sigma$ be as in Assumption \ref{ass:everywhere}.
%{\color{red}[Jonas: I would remove that part] andits domain, respectively.} 
Then, for every $r<1$ and any open book $\bar\bS$, 
	\begin{align}
	\int_{\bB_r} \frac{\dist(q,\bar \bS)^2}{\abs{q}^{m+\frac74}} \,d\|T\|
	\leq & C \int_{\bB_r} \frac{|q^\perp|^2}{|q|^{m+2}}\,d\|T\|  + C(\bE (T, \bar\bS, 0, r) + \bA^2)\, .\label{eq:est_prelim}
	\end{align}
\end{corollary}
\begin{proof}
Observe that $g (q) := \dist(q,\bar \bS)$ is $1$-homogeneous function and that $\tilde{g}$ is $1$-Lipschitz. The inequality follows therefore applying Lemma \ref{l:monot-with-A^2} with $k=1$ and $\alpha = \frac{1}{4}$.
\end{proof}

We next use the refined Lipschitz approximation of the previous sections to suitably bound the first summand in the right-hand side of \eqref{eq:est_prelim}.

\begin{proposition}\label{prop:density-est}
There are positive $\eta_{10}$ and $\eps_{10}$ such that the following holds. Let $T$, $\Sigma$, and $\bS$ be as in Assumption \ref{ass:everywhere} and \ref{ass:decay plane} with $\bar\eta \leq \eta_{10}$ and $\bar\eps \leq \eps_{10}$. Denote by $\mathbf{p}_{V}$ the orthogonal projection on the spine $V$ of $\bS$, and for $\|T\|$-a.e. $q$ denote by $\mathbf{p}_{\vec{T} (q)^\perp}$ the projection on the orthogonal complement of the tangent plane to $T$ at $q$. Also set $q^\perp:= q-\p_{\vec{T}(q)}(q)$ for $\Ha^m$-a.e. $q \in \spt (T)$. Then 
\begin{gather}\label{e.reverse poincare}
\int_{\bB_{\sfrac{1}{3}}}\left(\left|\mathbf{p}_V\cdot \mathbf{p}_{\vec{T} (q)^\perp}\right|^2 \, + \frac{|q^\perp|^2}{\abs{q}^{m+2}}\right) \,d\|T\| (q) \leq C \,(\mathbf E + \bA^2 )\, ,
\end{gather}
where $|\cdot|$ is the Hilbert-Schmidt norm and the constant $C$ depends upon $(m,n,p)$.
\end{proposition}

\begin{proof}
Let $g \in C^\infty_c(\bB_{1})$, and, denoting $\mathbf{p}_{V^\perp}$ the orthogonal projection onto the complement $V^\perp$ to the spine $V$ of $\bS$, proceed as in the proof of \cite[(8.14) Proposition 8.4]{DHMSS} to estimate
\begin{align}
&\int \left| \p_V \cdot \p_{\vec{T}^\perp} \right|^2 g^2 d\|T\| 
+ 2 \left(\int g^2 d\|T\| - \int g^2 d\|\tilde{\bC}\|\right) \nonumber \\
&\qquad \leq \underbrace{ -2 \int g^2  x \cdot H_T \, d\|T\|}_{=: {\rm(A)}}+ \underbrace{4 \int \abs{x^\perp}^2\, \abs{\nabla_V g}^2\, d\|T\|}_{=: {\rm (B)}} \nonumber \\
 &\qquad\;+ \underbrace{4 \int g \left(x\cdot \nabla_{V^\perp} g\right) d \|\tilde{\bC}\|
 - 4 \int g\, (\p_{\vec{T}}(x) \cdot \nabla_{V^\perp}g)\, d\|T\|}_{=: {\rm (C)}} \,. \label{e:FV-2}
\end{align}
Here recall that $\tilde{\bC}$ is a suitable representative $\modp$ supported in the book $\tilde{\bS}$: for its definition we refer to \cite{DHMSS}. 

Notice that when calculating ${\rm (B)}$ and ${\rm (C)}$ we can replace $\|T\|$ with $\|\bG_{\bar{\bS}}(\tilde{v})\|$ up to an error of size $\bE$ thanks to \eqref{e:rep_parametrization_bad}, since in all instances the integrand can be bounded by $|x|^2$ (see \cite[Proof of Proposition 8.4]{DHMSS}). The rest of the proof now proceeds as in \cite[Proposition 8.4]{DHMSS} using Corollaries \ref{cor:reparametrization} and \ref{cor:reparametrization on Bad cubes} in place of Corollary 6.6 therein.
\end{proof}

We can now combine Corollary \ref{c:first-Hardt-Simon} and Proposition \ref{prop:density-est} to infer the following 

\begin{corollary}\label{c:second-Hardt-Simon}
Let $T$, $\Sigma$, and $\bS$ be as in Proposition \ref{prop:density-est}. Then
\begin{align}
	\int_{\bB_{\sfrac{1}{3}}} \frac{\dist(q, \bS)^2}{\abs{q}^{m+\frac74}} \,d\|T\|
	\leq & C(\bE (T, \bS, 0, 1) + \bA^2)\, .\label{eq:est_second}
	\end{align}
\end{corollary}

\subsection{Shifted cones} In the next two steps to prove Theorem \ref{t:Hardt-Simon-main} we will make a fundamental use of the following geometric lemma.

\begin{lemma} \label{l:geometry-cones}
\begin{itemize}
    \item[(a)] Assume $\bS$ is an open book and $q,z\in \mathbb R^{m+n}$ and $O\in {\rm SO} (m+n)$. Then:
\begin{align}
\dist (z, q+ O (\bS)) &\leq \dist (z, q + \bS) + 2
|O-{\rm Id}| |z-q|\label{e:rotating_books}
\end{align}
\item[(b)] There is a geometric constant $C$ such that the following inequality holds for any $q,q',z\in \mathbb R^{m+n}$, any $m$-dimensional plane $\pi$ with $\pi\supset V= V (\bS)$ and under the additional assumptions that  $\mathbf{p}_{\pi} (\bS) = \pi$ and $\beta_\pi (\bS)\leq \frac{1}{2}$:
\begin{align}
\dist (z, q + \bS) &\leq \dist (z, q' + \bS) + |\mathbf{p}_{\pi^\perp} (q-q')| + C \beta_\pi (\bS) |\mathbf{p}_{\pi\cap V^\perp} (q-q')|\label{e:shifting-books}.
\end{align}
\item[(c)] For any constant $C_0$ there is a constant $C_1$ such that the following holds under the assumption that $C_0^{-1} \beta_{\pi} (\bS) \leq \beta_{max} (\bS)\leq \beta_\pi (\bS)$. For every $q\in \mathbb R^{m+n}$ there is a page $\bH\subset \bS$ such that 
\begin{equation}\label{e:shifting-books-2}
|\mathbf{p}_{\pi}^\perp (q)| + \beta_\pi (\bS) |\mathbf{p}_{V^\perp} (q)|\leq C \dist (x-q, \bS) \quad
\mbox{whenever $|\mathbf{p}_{V^\perp} (x)| \geq 2 |\mathbf{p}_{V^\perp} (q)|$}\, .
\end{equation}
\end{itemize}
\end{lemma}
\begin{proof} {\bf Proof of (a).} We can assume without loss of generality that $q=0$. Fix $z$ and let $y$ be a point in $\bS$ such that $\dist (z, \bS) = |z-y|$. Observe that certainly $|y|\leq 2 |z|$, otherwise $0\in \bS$ would be closer to $z$ then $y$. On the other hand $O (y)\in O (\bS)$ and thus we can estimate
\begin{align*}
\dist (z, O (\bS)) &\leq |z- O (y)|\leq |z-y| + |y- O(y)|
\leq \dist (z, \bS) + |O-{\rm Id}| |y|\\
&\leq \dist (z, \bS) + 2 |O-{\rm Id}| |z|\, .
\end{align*}

\medskip

{\bf Proof of (b).} Observe that we can write 
\[
q = q' + \underbrace{\mathbf{p}_V (q-q')}_{q_1} + \underbrace{\mathbf{p}_{\pi\cap V^\perp} (q-q')}_{q_2} + \underbrace{\mathbf{p}_{\pi^\perp} (q-q')}_{q_3}\, .
\]
Evidently, it suffices to prove the three claims
\begin{align*}
\dist (z, q'+q_1+\bS) &= \dist (z, q'+\bS)\\
\dist (z, q'+q_1+q_2+\bS) &\leq \dist (z, q'+q_1+\bS) + C \beta_\pi (\bS) |q_2|\\
\dist (z, q'+q_1+q_2+q_3+\bS) &\leq \dist (z, q'+q_1+q_2+\bS) + |q_3|\, .
\end{align*}
This amounts to show the inequality \eqref{e:shifting-books} in three particular cases in which $q-q'\in V$, $q-q'\in \pi\cap V^\perp$, and $q-q'\in \pi^\perp$. In all of these cases we can assume, without loss of generality, that $q'=0$. The third case is the trivial estimate, while the first one is obvious because $q+\bS=\bS$ when $q\in V$. We are thus left with the second case. 

Fix thus $z\in \mathbb R^{m+n}$ and $q\in \pi\cap V^\perp$. Denote by $\tau$ the $(m+1)$-dimensional plane which contains $\bS$ and $\pi$ and observe that it contains $q+\bS$ as well. Without loss of generality we can assume therefore that $z\in \tau$. The assumption $\mathbf{p}_{\pi} (\bS) = \pi$ implies the following geometric property:
\begin{itemize}
\item[(P)] for every $\xi \in \pi$, the line $\xi+ \tau\cap \pi^\perp$ intersects $\bS$.
\end{itemize}
Consider now $y$ such that $\dist (z, \bS) = |z-y|$ and let $\bH$ be the page containing it. We further set $y' := \mathbf{p}_{\pi \cap V^\perp} (y)$. Since $\pi \cap V^\perp$ is $1$-dimensional, we can distinguish two cases:
\begin{itemize}
\item[(a)] $y'$ is not contained in the segment $[0,q]$; in this case $\mathbf{p}_\pi (y) + \pi^\perp\cap \tau$ intersects $q+\bH$ in some point $y_q$ and $|y-y_q|\leq \tan \beta |q|$, where $\beta$ denotes the angle between $\bH$ and $\pi$. Since $\beta \leq C \beta_\pi (\bS)$ the desired inequality follows. 
\item[(b)] $y'$ is contained in the segment $[0,q]$. In this case $|\mathbf{p}_{\pi^\perp} (y)|\leq |q| \tan \beta$. The geometric property (P) guarantees that $y+\tau\cap \pi^\perp$ intersects $q+\bS$ at some point $y_q$. If $\bH'$ is the page of $\bS$ such that $y_q\in q+\bH'$, this time we get $|\mathbf{p}_{\pi^\perp} (y_q)|\leq \tan \beta' |q|$. Since $|y-y_q|\leq |\mathbf{p}_{\pi^\perp} (y)| + |\mathbf{p}_{\pi^\perp} (y_q)|$ the desired inequality follows again.
\end{itemize}

\medskip

{\bf Proof of (c).} Let $\bH_i$ be the pages of $\bS$ and denote by $\pi_i$ the $m$-dimensional plane which contains $\bH_i$. We will show below the following fact
\begin{itemize}
\item[(F)] For every $x\in \bH_i$ with $|\mathbf{p}_{V^\perp} (x)|\geq 2 |\mathbf{p}_{V^\perp} (q)|$, we have $\dist (x, q+\bS) = \dist (x, q+\bH_i) = |\mathbf{p}_{\pi_i^\perp} (q)|$.
\end{itemize}
From (F) we conclude as follows. We select a page $\bH_i$ with the property that $|\mathbf{p}_{\pi_i^\perp} (q)|$ is maximal. We then have to show that 
\begin{align}
|\mathbf{p}_{\pi^\perp} (q)|&\leq C |\mathbf{p}_{\pi_i^\perp} (q)|\\
\beta_{max} (\bS) |\mathbf{p}_{V^\perp} (q)| &\leq C |\mathbf{p}_{\pi_i^\perp} (q)|\, .
\end{align}
Consider the plane $\tau$ which contains $\pi$ and $\bS$ and observe that 
\begin{align*}
|\mathbf{p}_{V^\perp} (q)|^2 &= |\mathbf{p}_{\tau\cap V^\perp} (q)|^2 + |\mathbf{p}_{\tau^\perp} (q)|^2\\
|\mathbf{p}_{\pi_i^\perp} (q)|^2 &= |\mathbf{p}_{\pi_i^\perp} (\mathbf{p}_{\tau\cap V^\perp} (q))|^2 + |\mathbf{p}_{\tau^\perp} (q)|^2\\
|\mathbf{p}_{\pi^\perp} (q)|^2&= |\mathbf{p}_{\pi^\perp} (\mathbf{p}_{V^\perp \cap \tau} (q))|^2 + |\mathbf{p}_{\tau^\perp} (q)|^2\, .
\end{align*}
Since moreover $\bS$ and $q+\bS$ are invariant under translations along $V$, we can just reduce to the situation in which $q\in V^\perp\cap \tau$. Moreover, by dilation, we can assume it has unit length. Note therefore that we are reduced to prove the following claim. We have $2Q$ lines in $\mathbb R^2$ with the property that the maximal angle between them is $\beta_{\max} (\bS)$ and the maximal angle between one of them and the horizontal axis is $\beta_{\pi} (\bS)$. $\xi = (\cos \theta, \sin \theta)$ is a unit vector in $\mathbb R^2$ and $\ell$ is the one among the $2Q$ lines which is further away from $\xi$, while we wish to show that 
\begin{align}
|\sin \theta|&\leq C \dist (\xi, \ell)\\
\beta_{max} (\bS) &\leq C \dist (\xi, \ell)\, .
\end{align}
Pick the two lines $\ell_1$ and $\ell_2$ which form the largest angles $\beta_{max} (\bS)$ and let $\ell_1$ be the one further away from $\xi$ of the two. The angle between $\xi$ and $\ell_1$ is thus at least half of $\beta_{max} (\bS)$, but it is also smaller than the angle between $\xi$ and $\ell$ and so the second inequality is trivial. For the other inequality we notice that $\theta$ is the angle between $\xi$ and the horizontal axis, which is bounded by the sum of the angle between $\xi$ and $\ell_1$ (controlled by $\beta_{max} (\bS)$ and so by $\dist (\xi, \ell)$) and the angle between $\ell_1$ and the horizontal line, which is bounded by $\beta_{\pi} (\bS)$. Since the latter is also bounded by $\beta_{max} (\bS)$, which in turn is bounded by $C\dist (\xi, \ell)$, we have proved our claim. 

We now come to the proof of (F). Let $\pi_i$ be the plane containing $\bH_i$ and observe that, since $|\mathbf{p}_{V^\perp} (x)|\geq 2 |\mathbf{p}_{V^\perp} (q)|$, we easily see that $\mathbf{p}_{\pi_i} (x-q)=x- \mathbf{p}_{\pi_i} (q)$ belongs to $\bH_i$. Thus $\dist (x, q+\bH_i) = |\mathbf{p}_{\pi_i^\perp} (q)|$. We hence just need to show that 
\[
\dist (x, q+ \bH_j) 
\geq |\mathbf{p}_{\pi_i^\perp} (q)| = \dist (x, q+ \bH_i)
\]
for every other page $\bH_j$. We will in fact show that $\dist (x, q+ \pi_j)\geq \dist (x, q+ \pi_i)$ (which is enough because $\dist (x, q+ \bH_j) \geq \dist (x, q+\pi_j)$). Summarizing, we are left with the task of proving 
\[
\dist (x-q, \pi_i) \leq \dist (x-q, \pi_j)
\]
for every $x\in \pi_i$ such that $|\mathbf{p}_{V^\perp} (x)|\geq 2 |\mathbf{p}_{V^\perp} (q)|$. 

Arguing as above, we can ignore the components of $q$ and $x$ along $V$ and along $\tau^\perp$. 
We thus reduce the claim to a statement about pairs of lines. More precisely, given two lines $\ell, \ell'\subset \mathbb R^2$, a point $x\in \ell$ and a point $q$ with $2 |q|\leq |x|$, we wish to show that $\dist (x-q, \ell)\leq \dist (x-q, \ell')$. By scaling we can assume $|q|=1$. We thus fix coordinates on the plane in such a way that $\ell = \{(s, 0): s\in \mathbb R\}$, $x= (\sigma, 0)$ with $\sigma\geq 2$, $q= (\cos \alpha, \sin \alpha)$, and $(\cos \beta, \sin \beta)$ is a unit vector ortogonal to $\ell'$. The claim then amounts to the inequality
\[
\sin^2 \alpha\leq (\sigma-\cos \alpha)^2 \cos^2 \beta + \sin^2 \alpha \sin^2 \beta\, .
\]
Notice however that, since $\sigma\geq 2$, $\sin^2 \alpha\leq 1 \leq (2-\cos \alpha)^2$ and the desired inequality follows easily.
\end{proof}

\subsection{Shifted \texorpdfstring{$Q$}{Q}-points} Consider now any point $q\in \bB_{1/16}$ with $\Theta_T (q)\geq Q$. For each such $q$ we fix a rotation $O_q$ of the ambient space, with the properties that 
\begin{itemize}
\item[(i)] $O_q (T_0 \Sigma) = T_q \Sigma$;
\item[(ii)] $|O_q-{\rm Id}|$ is minimal among all rotations which satisfy condition (i).
\end{itemize}
Clearly 
\begin{equation}\label{e:rotation_simple_est}
|O_q-{\rm Id}|\leq C_0 |q| \bA\, ,
\end{equation}
for some geometric constant $C_0$. The point of this Section is to show that, provided $\bar \eta$ and $\bar\varepsilon$ are small enough, we achieve an estimate as in \eqref{eq:est_second} with $q$ replacing the origin and $O_q (\tilde{\bS})$ replacing $\bS$. 

\begin{proposition}\label{prop:shifting}
Let $T$, $\Sigma$, $\pi_0$, and $\bS$ be as in Assumption \ref{ass:everywhere} and \ref{ass:decay plane}, with parameters $\bar \eta$ and $\bar\varepsilon$ small enough to apply Theorem \ref{thm:graph v2} and define the book $\tilde{\bS}$. Then there are $\eta_{11}$ and $\varepsilon_{11}$ such that, if $\bar\eta < \eta_{11}$ and $\bar\varepsilon < \varepsilon_{11}$, then Corollary \ref{c:second-Hardt-Simon} applies with $T_{q,1/3}$ in place of $T$, $O_q (\tilde{\bS})$ in place of $\bS$, and $O_q (\pi_0)$ in place of $\pi_0$, whenever $q\in \bB_{1/16}$ satsfies $\Theta_T (q) \geq Q$. In particular, for any such point we gain the estimate
\begin{equation}\label{e:Hardt-Simon-3}
\int_{\bB_{\sfrac{1}{9}} (q)} \frac{\dist(z-q, O_q (\tilde{\bS}))^2}{\abs{z-q}^{m+\frac74}} \,d\|T\|
	\leq  C(\bE (T, O_q (\tilde{\bS}), q, \sfrac{1}{3}) + \bA^2)\, .
\end{equation}
\end{proposition}

\begin{proof} In order to show that Corollary \ref{c:second-Hardt-Simon} applies with $T' = T_{q,1/3}$ in place of $T$, $\bS' = O_q (\tilde{\bS})$ in place of $\bS$ and $\pi'=O_q (\pi_0)$ in place of $\pi_0$ we need to show the following conditions:
\begin{itemize}
    \item[(i)] $9^{-1} \bA^2 \leq \eps_{10}\, \bE (T', \pi', 0, 1)\leq \varepsilon_{10}^2$; 
    \item[(ii)] $\mathbb{E} (T', \bS', 0,1) \leq \eta_{10} \bar\bE (T', 0, 1)$;
    \item[(iii)] $\bar \bE (T', 0, 1) \geq (1-\eta_{10}) \bE (T', 0, 1)$.
\end{itemize}
This will be shown assuming that 
\begin{itemize}
    \item[(a)] $\bA^2 \leq \bar\varepsilon\, \bE (T, \pi_0, 0, 1)\leq \bar\varepsilon^2$; 
    \item[(b)] $\mathbb{E} (T, \bS, 0,1) \leq \bar\eta \bE (T, \pi_0, 0, 1)$;
    \item[(c)] $\bar \bE (T, 0, 1) \geq (1-\bar \eta) \bE (T, \pi_0, 0, 1)$;
\end{itemize}
where $\bar\varepsilon$ and $\bar\eta$ are two much smaller parameters. 

\medskip

{\bf Step 1. Height of $q$.} We first prove that, for every fixed positive $\rho$, no matter how small, 
\begin{equation}\label{e:small-height}
|\mathbf{p}_{\pi_0}^\perp (q)|^2 \leq \rho^2 \bE (T, \pi_0, 0, 1)
\end{equation}
provided $\bar\eta$ and $\bar\varepsilon$ are chosen small enough. 

Assume by contradiction this is not the case. Then there are sequences $T_k$, $\bS_k$, $\Sigma_k$ satisfying Assumption \ref{ass:everywhere} and (a), (b), and (c) above with vanishing $\bar \eta = \eta_k$ and $\bar\varepsilon = \varepsilon_k$, and a sequence of points $q_k\in \bB_{\sfrac{1}{16}}$ with $\Theta_{T_k} (q_k) \geq Q$ such that 
\begin{equation}\label{e:small-height-2}
|\mathbf{p}_{\pi_0}^\perp (q)|^2 \geq \rho^2 \bE (T_k, \pi_0, 0, 1)\, .
\end{equation}
By applying a rotation, we can assume that all $\bS_k$ have the same spine $V$.
Let now $y_k = \mathbf{p}_V (q_k)$ and recall that, by Lemma \ref{l:propagation} $|q_k-y_k|\to 0$. Up to subsequences we can also assume that $y_k \to y$. We argue as in the Proof of Lemma \ref{lem.propagation1} and in particular introduce the maps $\bar v_k$ and study their limit $v$, which is a Dir-minimizing map, and is a strong $L^2$ limit. By Proposition \ref{p:L2Alm} and because $v$ is $1$-homogeneous and invariant by translation along the spine $V$, we see that, for every fixed $r$, 
\begin{equation}\label{e:1-homogeneous}
\lim_{k\to \infty} \frac{\bE (T_k, \pi_0, y, r)}{\bE (T_k, \pi_0, 0, 1)} = 1\, .
\end{equation}
In particular we also see that 
\begin{equation}
\lim_{k\to \infty} \frac{\bE (T_k, \pi_0, y_k, r)}{\bE (T_k, \pi_0, 0, 1)} = 1\, .
\end{equation}
So, for $k$ large enough, we have 
\[
\bE (T_k, \pi_0, y_k, r) \leq 2 \bE (T_k, \pi_0, 0, 1)\, .
\]
Since $q_k$ converges towards $y_k$, we can, for a sufficiently large $k$, apply Lemma \ref{Linfty-L2} to conclude 
\[
|\mathbf{p}_{\pi_0^\perp} (q_k)|
\leq C r \bE (T_k, \pi_0, 0, 1)^{\sfrac{1}{2}}\, .
\]
The constant $C$ is independent of $r$. Therefore, by choosing $r$ smaller than $\frac{\rho}{C}$ we contradict \eqref{e:small-height-2}.

\medskip

{\bf Step 2.} We now wish to prove (i). We argue again by contradiction. This time we have, however, either
\begin{equation}\label{e:annoying-1}
\bE (T_k, O_{q_k} (\pi_0), q_k, \sfrac{1}{3})\geq \varepsilon_{10} \, ,
\end{equation}
or 
\begin{equation}\label{e:annoying-2}
\bA_k^2 \geq 9 \varepsilon_{10} \bE (T_k, O_{q_k} (\pi_0), q_k, \sfrac{1}{3}) \, .
\end{equation}
Observe that
\begin{align*}
\bE (T_k, O_{q_k} (\pi_0), q_k, \sfrac{1}{3})
& \leq \left(1+ 3 |q_k - y_k|\right)^{m+2} \bE (T_k, \pi_0, y_k, \sfrac{1}{3}+ |q_k -y_k|)\\
&\qquad\quad +
C |\mathbf{p}_{\pi_0}^\perp (q_k)|^2 + C \bA_k^2\, ,
\end{align*}
but also 
\begin{align*}
\bE (T_k, O_{q_k} (\pi_0), q_k, \sfrac{1}{3})
& \geq \left(1- 3 |q_k - y_k|\right)^{m+2} \bE (T_k, \pi_0, y_k, \sfrac{1}{3}- |q_k -y_k|)\\
&\qquad\quad-
C |\mathbf{p}_{\pi_k}^\perp (q_k)|^2 - C \bA_k^2\, ,
\end{align*}
Recalling \eqref{e:1-homogeneous} we conclude that 
\begin{equation}
\lim_{k\to\infty} \frac{\bE (T_k, O_{q_k} (\pi_0), q_k, \sfrac{1}{3})}{\bE (T_k, \pi_0, 0, 1)} = 1\, .
\end{equation}
Since however $\bE (T_k, \pi_0, 0, 1)$ and $\bE (T_k, \pi_0, 0, 1)^{-1} \bA_k^2$ are both infinitesimal, clearly we contradict either \eqref{e:annoying-1} or \eqref{e:annoying-2}.

\medskip

{\bf Step 3.} We next prove (iii). Assume by contradiction that there is a sequence of planes $\pi_k \subset T_q \Sigma_k$ such that 
\[
\bE (T_k, \pi_k, q_k, \sfrac{1}{3}) 
\leq (1-\eta_{10}) \bE (T_k, O_q (\pi_k), q_k, \sfrac{1}{3})\, . 
\]
Using again the estimate $|O_{q_k} - {\rm Id}|\leq C \bA_k$ and the estimates of the previous steps, we conclude that 
\[
\bE (T_k, \pi_k, y, \sfrac{1}{3} - |q_k -y|)\leq 
(1-\eta_{10}/2) \bE (T_k, \pi_0, y, \sfrac{1}{3}+ |q_k-y|)\, .
\]
Observe also that $|\pi_k - \pi_0|\leq C \bE (T_k, \pi_0, 0, 1)^{\sfrac{1}{2}} =: \bar \bE_k^{\sfrac{1}{2}}$. Consider the linear maps $l_k: \pi_0 \to \pi_0^{\perp_0}$ whose graph give $\pi_k$ and let $l$ be their limit, up to subsequences.

If $v$ is the limiting function found in the proof of Lemma \ref{lem.propagation1}, observe that 
\[
\lim_{k\to \infty} \bar\bE_k^{-1} \bE (T_k, \pi_0, y, \sfrac{1}{3} + |q_k -y|)  = \frac{1}{3^{m+2}}
\int_{B_{1/3} (y)} |v|^2\, .
\]
On the other hand we also get 
\[
\lim_{k\to\infty} \bar\bE_k^{-1} \bE (T_k, \pi_k, y, \sfrac{1}{3} - |q_k -y|) = 
\frac{1}{3^{m+2}} \int_{B_{1/3}} \sum_i |v_i - l|^2\, . 
\]
In particular we would conclude that there is a linear function $l$ such that 
\[
\int_{B_{1/3}} \sum_i |v_i-l|^2 < \int_{B_{1/3}} |v|^2\, .
\]
Recall however that, because $\bE (T_k, \pi_0, 0, 1) \leq (1-\eta_k)^{-1} \bar \bE (T_k, 0, 1)$, $\etab\circ v = Q^{-1}\,\sum_i v_i$ vanishes identically. In particular 
\[
\sum_i |v_i - l|^2 = |v|^2 + Q |l|^2\, ,
\]
which in turn shows 
\[
\int_{B_{1/3}} \sum_i |v_i-l|^2 \geq \int_{B_{1/3}} |v|^2\, .
\]

\medskip

{\bf Step 4.} It remains to show (ii). By Lemma \ref{l:geometry-cones} we have
\begin{align*}
\dist (z, O_q(\tilde{\bS})+q)&
\leq C \bA |q| |z-q| + \dist (z, \tilde{\bS}+q)\\
&\leq C \bA |q| |z-q| + \dist(z, \tilde{\bS}) + 
|\mathbf{p}_{\pi_0^\perp}  (q)| + C \beta_{\pi_0} (\bS) |q- \mathbf{p}_V (q)|\, .
\end{align*}
In particular, we can estimate
\begin{align}
\bE (T', \bS', 0, 1)
&\leq C (\bE (T, \tilde{\bS}, 0, \sfrac{1}{2}) + 
\bA^2 + |\mathbf{p}_{\pi_0^\perp}  (q)|^2 +
C \beta_{\pi_0} (\bS) |q- \mathbf{p}_V (q)|^2)\nonumber\\
&\leq C (\bE (T, \bS, 0, 1) + 
\bA^2 + |\mathbf{p}_{\pi_0^\perp}  (q)|^2 +
C \beta_{\pi_0} (\bS) |q- \mathbf{p}_V (q)|^2)\, ,
\end{align}
where in the last line we have used \eqref{eq:tildeexcesvsexcess}. 

From the previous steps it follows that each of the summands on the right hand side can be made arbitrarily small with respect to $\bE (T, \pi_0, 0, 1)$, provided $\bar\eta$ and $\bar\varepsilon$ are taken small enough. Since in turn $\bE (T, \pi_0, 0, 1)$ can be bounded by $2 \bE (T', \pi', 0, 1)$ by possibly chosing the two parameters even smaller, we conclude the proof. 
\end{proof}

\subsection{Proof of Theorem \texorpdfstring{\ref{t:Hardt-Simon-main}}{Hardt-Simon}} By Proposition \ref{prop:shifting} we have
\[
\int_{\bB_{\sfrac{1}{9}} (q_0)} \frac{\dist(q, O_{q_0} (\tilde{\bS}) +q_0)^2}{\abs{q-q_0}^{m+\frac74}} \,d\|T\|
	\leq  C(\bE (T, O_{q_0} (\tilde{\bS}), q, \sfrac{1}{3}) + \bA^2)\, ,
\]
provided the parameters are small enough. Using Lemma \ref{l:geometry-cones} and \eqref{eq:tildeexcesvsexcess} we then get 
\begin{align}
\int_{\bB_{\sfrac{1}{9}} (q_0)} \frac{\dist(q-q_0, \tilde{\bS})^2}{\abs{q-q_0}^{m+\frac74}} \,d\|T\|
\leq & C\int_{B_{\sfrac{1}{3}}} \dist (q, \tilde{\bS})^2 d\|T\| + C (\bA^2 + |x_0^\perp|^2 + \beta_{\pi_0} (\tilde\bS)^2 |x_0|^2)\nonumber\\
\leq & C\bE (T, \tilde{\bS}, 0, \sfrac{1}{2}) +  C (\bA^2 + |x_0^\perp|^2 + \beta_{\pi_0} (\tilde\bS)^2 |x_0|^2)\nonumber\\
\leq & C \bE (T, \bS, 0, 1) + C (\bA^2 + C\, |x_0^\perp|^2+C \beta_{\pi_0} (\tilde\bS)^2 |x_0|^2)\, .\label{e:HS-estimate-3}
 \end{align}
 From now on in order to simplify our notation we use $\bE$ in place of $\bE (T, \bS, 0, 1)$. Fix $\rho>0$. We next wish to show that, provided the parameters $\varepsilon_9$ and $\eta_9$ are small enough, then
 \begin{equation}\label{e:using-rho}
|x_0^\perp|^2 + \beta_{\pi_0} (\tilde{\bS})^2 |x_0|^2 \leq 
C \rho^{\sfrac{7}{4}} \int_{\bB_{\sfrac{1}{9}} (q_0)} \frac{\dist(q-q_0, \tilde{\bS})^2}{\abs{q-q_0}^{m+\frac74}} \,d\|T\| +
C\rho^{-m} (\bE + \bA^2)\, ,
 \end{equation}
 where the constant $C$ is independent of $\eta_9$ and $\varepsilon_9$. In particular, for $\rho$ sufficiently small we can combine \eqref{e:using-rho} and \eqref{e:HS-estimate-3} to get 
 \begin{equation}\label{e:using-rho-2}
|x_0^\perp|^2 + \beta_{\pi_0} (\tilde{\bS})^2 |x_0|^2\leq 
C\rho^{-m} (\bE + \bA^2)\, .
 \end{equation}
 We fix such a $\rho$ and gain therefore 
 \[
\beta_{\pi_0}(\tilde\bS)^2\, |x_0|^2+|x_0^\perp|^2 + \int_{\bB_{1/4}} \frac{{\rm dist}^2\, (q-q_0, \tilde{\bS})}{|q-q_0|^{m+\frac74}}\, d\|T\| (q) \leq C (\bE(T, \bS, 0,1) + \bA^2)\, 
 \]
 (where we are treating the fixed $\rho$ as a geometric constant). 
 Since however $\beta_{\pi_0} (\bS) \leq C \beta_{\pi_0} (\tilde{\bS})$ by \eqref{eq:angle_bound_new_book}, we achieve our desired conclusion.

 It remains to show \eqref{e:using-rho}. First of all, by assuming the parameters small enough, Lemma \ref{lem.propagation1} implies that $2 |\mathbf{p}_{V^\perp} (q_0)|\leq \rho$. Thus we can apply Lemma \ref{l:geometry-cones} (c) and select a page $\bH_i$ of $\tilde{\bS}$ with the property that 
 \[
\beta_{\pi_0}(\tilde\bS)^2\, |x_0|^2+|x_0^\perp|^2 
\leq C \dist (x-q_0, \tilde{\bS})^2 \qquad \forall x\in \bH_i \setminus B_\rho (V)\, .
 \]
 We next apply Theorem \ref{thm:graph_v1}(v) and assume the parameters $\bar \varepsilon$ and $\bar\eta$ are small enough so that $\rho_{\mathcal{W}} (y) \leq \rho$ for all $y\in \bB_{\sfrac{1}{4}}$. Since $\bH_i \in \tilde{\bS}$, it follows that there is a function $\tilde{v}_j$ as in Corollary \ref{cor:reparametrization} and that $\Omega := (B_{2\rho} (V) \setminus B_\rho (V))\cap \bB_\rho (q_0)$ belongs to the domain of $\tilde{v}_j$. For each point $x$, consider the point $q=x+\tilde{v}_j (x)\in \spt (T)$. We then have 
 \begin{align}
\beta_{\pi_0}(\tilde\bS)^2\, |x_0|^2+|x_0^\perp|^2 
&\leq C \dist (q-q_0, \tilde{\bS})^2 + C |\tilde{v}_j (x)|\nonumber\\
&\leq C \dist (q-q_0, \tilde{\bS})^2 + C |\tilde{u}_j (x)|^2 + C \bA^2\qquad \forall \Omega\, .\label{e:will-integrate}
 \end{align}
 Moreover, given the Lipschitz and $L^\infty$ bounds on $\tilde{v}_j$, it follows that $q\in \bB_{1/9} (q_0)$ and that $|q-q_0|\geq \frac{\rho}{4}$. We thus average \eqref{e:will-integrate} over the set $\Gamma := \{x+\tilde{v}_j (x): x\in \Omega\}$ and use \eqref{e:comparison-tilde-u-w} to achieve 
\begin{align*}
\beta_{\pi_0}(\tilde\bS)^2\, |x_0|^2+|x_0^\perp|^2
&\leq C \rho^{-m} \int_{\Gamma} \dist (q-q_0, \tilde{\bS})^2 d\mathcal{H}^m + C \rho^{-m} \int_\Omega |\tilde{u}_j|^2 + C \rho^{-m} \bA^2\\
&\leq C \rho^{\sfrac{7}{4}} \int_\Gamma \frac{\dist (q-q_0, \tilde{\bS})^2}{|q-q_0|^{m+\sfrac{7}{4}}} \,d\Ha^m
+ C \rho^{-m} (\bE + \bA^2)\\
&\leq C \rho^{\sfrac{7}{4}} \int_{\bB_{\sfrac{1}{9}}} \frac{\dist (q-q_0, \tilde{\bS})^2}{|q-q_0|^{m+\sfrac{7}{4}}}\, d\|T\| (q) + C \rho^{-m} (\bE + \bA^2)\, .
\end{align*}
This completes the proof of \eqref{e:using-rho} and hence the proof of Theorem \ref{t:Hardt-Simon-main}.

\section{Proof of Proposition \texorpdfstring{\ref{p:decay-2}}{decay2}: binding functions}

Following the blueprint of Simon's work on cylindrical tangent cones, in the form used in \cite{DHMSS} in this section we prove the existence of suitable ``binding functions'', which in the final blow-up proof of Proposition \ref{p:decay-2} will be crucial to show the compatibility of the harmonic sheets. The central Proposition of this section has its counterpart in \cite[Theorem 9.3]{DHMSS}. The crucial difference is that we are not able to really estimate the ``binding function'' $\xi$ in terms of the excess $\bE$ (as it is the case for \cite[Theorem 9.3]{DHMSS}). We will instead be able to estimate separately its vertical portion $\mathbf{p}_{\pi_0^\perp} (\xi)$ and the horizontal portion $\mathbf{p}_{\pi_0} (\xi)$: it's in the estimate for the latter part that we ``lose''. 

\begin{definition} \label{def:binding functions}
A \emph{binding function} is any Borel measurable function $\xi \colon R_{\mathcal W} \to V^\perp$ with the property that
     $\xi(q) = \xi(q')$ for all $q=(0,x,y)$ and $q'=(0,x',y')$ such that $(|x|,y)$ and $(|x'|,y')$ belong to the interior of the same Whitney cube.
\end{definition}

\begin{theorem}\label{t:binding}
 There are positive constants $C$, $\eta_{12}$, and $\varepsilon_{12}$ depending upon $(m,n,p)$ such that the following holds. If
 \begin{itemize}
     \item[(i)] $T,\Sigma, \bS, \pi_0$ are as in Assumptions \ref{ass:everywhere} and \ref{ass:decay plane},
     \item[(ii)] $\bar\eps <\eps_{12}$ and $\bar \eta < \eta_{12}$, 
     \item[(iii)] $\tilde{\bS}$ denotes the open book introduced in Definition \ref{d:new-cone},
     \item[(iv)] and $\varrho_\infty:= \|\varrho_{\mathcal{W}}\|_\infty$,
    \end{itemize}
then
\begin{equation}\label{e:est spine}
\int_{\bB_{\sfrac18}} \frac{\dist(q,\tilde\bS)^2}{\max\{\varrho_\infty, |x|\}^{1/2}} \, d\|T\|(q) \leq C (\bE (T, \bS, 0, 1)+ \bA^2) =: C (\bE + \bA^2)\, .
\end{equation}
Moreover, there exists a binding function $\xi: R_{\mathcal{W}} \to \mathbb R^{m+n}$ such that the following estimates hold for every $j$:
\begin{align}
&\int_{\bB_{\sfrac{1}{8}}\cap U_{\mathcal{W}}^\pm} \frac{|u^\pm_j (q) - l_{h(j)} (q)) - (\mathbf{p}_{\pi_0^\perp} (\xi (q)) - l_{h(j)} (\mathbf{p}_{\pi_0} (\xi (q)))|^2}{|x|^{\sfrac{5}{2}}} d\mathcal{H}^m (q) \leq 
C (\bE  + \bA^2)\, ,\label{e:binding-1}\\
&\int_{\bB_{\sfrac{1}{8}}\cap U_{\mathcal{W}}^\pm}
\frac{|\nabla u^\pm_j(q) - \nabla l_{h(j)} (q)|^2}{|x|^{\sfrac{1}{2}}}\leq C (\bE + \bA^2)\, , \label{e:binding-2}\\
&\|\mathbf{p}_{\pi_0^\perp} (\xi)\|^2_{\infty}
\leq C (\bE  + \bA^2)\, .\label{e:binding-3}\\
&\|l_{h(j)}\circ \mathbf{p}_{\pi_0} (\xi)\|_\infty^2
\leq C \beta_{\pi_0} (\bS)^2 \|\mathbf{p}_{\pi_0} (\xi)\|_\infty^2
\leq C (\bE + \bA^2)\, .\label{e:binding-4}
\end{align}
\end{theorem}

\begin{proof}
The proof of \eqref{e:est spine} follows verbatim the one given in \cite[Section 9.2]{DHMSS} for the analogous estimate \cite[(9.5)]{DHMSS}: in this case the argument would substitute Lemma \ref{lem.propagation1} to \cite[Proposition 9.4]{DHMSS} and Theorem \ref{t:Hardt-Simon-main} to \cite[Theorem 8.1]{DHMSS}. Note that since the left hand side of \eqref{e:Hardt-Simon} has a quadratic dependence on $\bA$ rather than the linear one of \cite[(8.1)]{DHMSS}, \eqref{e:est spine} gains the quadratic dependence on $\bA$ on its right hand side as well. 

As for \eqref{e:binding-2} we can follow the argument in \cite[Section 9.2]{DHMSS} in order to show the following partial statement. For every cube $L\in \mathcal{W}$ we find a suitable point $\xi_L\in \spt (T) \cap \bB_{1/4}$ with $\Theta_T (\xi_L)\geq Q$ such that, for every $j$
\begin{equation}\label{e:binding-5}
\sum_{L\in \mathcal{W}^\pm: L\cap \bB_{1/8}\neq \emptyset} \int_{2L} \frac{|(u_j (z) - l_{h(j)} (z))- (\mathbf{p}_{\pi_0^\perp} (\xi_L) - l_{h(j)} (\mathbf{p}_{\pi_0} (\xi_L))|^2}{|x|^{\sfrac{5}{2}}}\, dz
\leq C (\bE + \bA^2)
\end{equation}
and
\begin{equation}\label{e:binding-6}
\sum_{L\in \mathcal{W}^\pm: L\cap \bB_{1/8}\neq \emptyset} \int_{2L}\frac{|\nabla u_j (z) - \nabla l_{h (j)} (z)|^2}{|x|^{\sfrac{1}{2}}}\, dz
\leq C (\bE + \bA^2)\, . 
\end{equation}
Again in this case the gain of a quadratic estimate on $\bA$, compared to the linear dependence of the analogous estimates in \cite{DHMSS}, is due to the quadratic dependence on $\bA$ of the right hand side of \eqref{e:Hardt-Simon}.

We now set the binding function to be equal to $\mathbf{p}_{V^\perp}(\xi_Q)$ in each cube $Q\in \mathscr{W}$.
 Summing over all the cubes we then reach \eqref{e:binding-1} and \eqref{e:binding-2}. At the same time \eqref{e:binding-3} and \eqref{e:binding-4} follow immediately from \eqref{e:Hardt-Simon}.
\end{proof}

\section{Proof of Proposition \texorpdfstring{\ref{p:decay-2}}{p:decay-2}: final blow-up} 

In this section we introduce a suitable blow-up sequence which will be used to prove Proposition \ref{p:decay-2}. 

\subsection{Blow-up sequence} The main argument is by contradiction. We therefore fix $p=2Q$ and fix sequences $\Sigma_k$, $T_k$, $\pi_k$, and $\bS_k$ with the following properties:
\begin{itemize}
\item[(a)] $T_k$ and $\Sigma_k$ satisfy Assumption \ref{ass:everywhere};
\item[(b)] $T_0 \Sigma_k = \mathbb R^{m+1}\times \mathbb \{0_{n-1}\}=: \tau_0$ and $\bA_k = \|A_{\Sigma_k}\|_\infty$;
\item[(c)] $\pi_k\in \mathscr{P} (0, \Sigma_k)$ and $\bar \bE_k := \bar \bE_k (T_k, 0, 1)= \bE (T_k, \pi_k, 0, 1)$;
\item[(d)] $\bS_k\in \mathscr{B} (0, \Sigma_k)$ and $\bbE_k := \bbE (T_k, 0, 1) = \bbE (T_k, \bS_k, 0, 1)$;
\item[(e)] the following holds:
\begin{align}\label{e:bu-assumption-1}
\lim_{k\to \infty} \left(\bar\bE_k + \frac{\bbE_k}{\bar \bE_k} + \frac{\bA_k^2}{\bbE_k}\right) = 0\, .
\end{align}
\end{itemize}
From Lemma \ref{lem.propagation1} if we pass to $(\lambda_{0,\frac{1}{2}})_\sharp T_k$ and change the optimality of $\pi_k$ in (c) to ``almost optimality'', we can additionally assume that $V (\mathbf{S}_k) \subset \pi_k$. Since passing to the rescaled currents just leads to a slightly different radius $r_2$ in the conclusion of Proposition \ref{p:decay-2}, we will keep the notation $T_k$. Moreover, by possibly applying a rotation, without loss of generality we can assume in addition to (a)-(e) the following two facts:
\begin{itemize}
    \item[(f)] $\pi_0 = \mathbb R^m\times \{0_{n}\}$, and $|\pi_k-\pi_0|\leq C \bar{\mathbf{E}}_k^{\sfrac{1}{2}}$, and $V = V (\bS_k) = \{0_1\}\times \mathbb R^{m-1} \times \{0_n\}$;
    \item[(g)] $\pi_0$ is almost optimal, namely
\begin{equation}\label{e:bu-assumption-2}
\lim_{k\to\infty} \frac{\bE (T_k, \pi_0, 0, 1)}{\bar\bE_k} = 1\, .
\end{equation}
\end{itemize}

\begin{definition}\label{d:scoppia}
A blow-up sequence is a sequence of quadruples $(T_k, \Sigma_k, \pi_k, \bS_k)$ together with linear subspaces $\tau_0 = T_0 \Sigma_k \supset \pi_0 \supset V = V (\bS_k)$ satisfying (a), (b), (c), (d), (e), (f), and (g). 
\end{definition}

We are now in a position of applying Theorem \ref{thm:graph_v1}, Theorem \ref{thm:graph v2}, Corollary \ref{cor:reparametrization}, Theorem \ref{t:Hardt-Simon-main}, and Theorem \ref{t:binding}, for any $k$ sufficiently large. In particular we can introduce 
\begin{itemize}
\item[($\alpha$)] The Whitney decompositions $\mathcal{W}_k$, the good regions $R_{\lambda \mathcal{W}_k}$ with the corresponding functions $u_j^{k,\pm}$ as in Theorem \ref{thm:graph_v1}, and the radii $\rho^k_\infty$ as in Theorem \ref{t:binding};
\item[($\beta$)] The new books $\tilde{\bS}_k$ and the linear maps $\tilde{l}_j^{k,\pm}  = l_{h_k (j)}^{k,\pm}: \pi_0^\pm \to \pi_0^{\perp_0}$ parametrizing their pages $\tilde{\bH}^{k, \pm}_j$;
\item[($\gamma$)] The binding functions $\xi^k:R_{\mathcal{W}_k}\cap \bB_{1/8} \to \mathbb R^{m+n}$. 
\end{itemize}
The following is then an easy corollary of the estimates in Theorem \ref{t:binding}, whose proof is left to the reader.  

\begin{corollary}\label{c:scoppia}
Consider a blow-up sequence $(T_k, \Sigma_k, \pi_k, \bS_k)$ and a plane $\pi_0$ as in Definition \ref{d:scoppia}. Consider books $\tilde{\bS}_k$, with pages $\tilde{\bH}^{k, \pm}_j$, and maps $\xi_k$ and $\tilde{u}_j^{k,\pm}$ as in ($\alpha$)-($\gamma)$. Hence set $\bar w^{k,\pm}_j:= \bbE_k^{-\sfrac{1}{2}} (u^{k,\pm}_j-\tilde{l}^{k, \pm}_j)$, $\xi^k_v:= \bbE_k^{-\sfrac{1}{2}} \mathbf{p}_{\pi_0^\perp} (\xi^k)$, and 
$\xi^k_o := \bbE_k^{-\sfrac{1}{2}} \beta_{\pi_0} (\tilde\bS) \mathbf{p}_{\pi_0} (\xi^k)$. Then, up to subsequences, the following holds:
\begin{itemize}
    \item[(i)] For each $j$ the sequence $\bar w^{k, \pm}_j$ converges locally in $C^1$ to a map $\bar{w}^\pm_j : \bB_{1/2}\cap \pi_0^\pm \to \pi_0^{\perp_0}$; 
    \item[(ii)] $\bar \xi^k_o$ and $\bar \xi^k_v$ converges locally uniformly to a pair of bounded functions 
    \begin{align}
    &\bar \xi_v: \bB_{1/8}\cap \pi_0 \to \pi_0^{\perp_0}\label{e:vertical-binding}\\
    &\bar \xi_o: \bB_{1/8} \cap \pi_0 \to V^\perp\cap \pi_0\label{e:horizontal-binding}
    \end{align}
    which are even with respect to $V$, namely $\xi_v (t, y)=\xi_v (-t,y)$ and $\xi_o (t,y)= \xi_o (-t, y)$ for every $(t,y)\in (V^\perp\cap \pi_0) \times V$ on their domain of definition;
    \item[(iii)] The normalized linear functions $\bar{l}^{k,\pm}_j := (\beta_{\pi_0} (\tilde{\bS}))^{-1} \tilde{l}^{k,\pm}_j$ converge smoothly to linear functions $\bar{l}_j^\pm$;
    \item[(iv)] The following estimates hold (for a geometric constant $C = C (Q,m,n)$): 
    \begin{align} \label{bu:c_onealpha}
        &\sup_{\zeta = (t,y) \in \pi_0^\pm} |t|^{\frac{m}{2}+1}\left( |t|^{-1} \abs{\bar w^\pm_j (\zeta)} + \abs{D\bar w^\pm_j (\zeta)}+ |t|^{\sfrac12} [D\bar w^\pm_j]_{\sfrac12}(\zeta) \right)\le C\\ \label{bu:radial}
        &\int_{\bB_{1/8}\cap \pi_0^\pm} |z|^{2-m}\left|\partial_r \frac{\bar w_j^\pm (z)}{|z|}\right|^2\, dz \leq C\\ \label{bu:nonconcentration w12}
        &\sum_{j}\int_{\bB_{1/8}\cap \pi_0^\pm} \frac{|\bar w^\pm_j - (\bar \xi_v - \bar l_j^\pm \circ \bar \xi_o)|^2}{|x|^{\frac{5}{2}}} +\frac{|\nabla \bar w^\pm_j|^2}{|x|^{\frac{1}{2}}}\, dz \leq C\, .
    \end{align}
\end{itemize}
\end{corollary}

\subsection{Strong convergence} Again following the blueprint of Simon's work, the estimates of the previous sections will allow us to conclude that the convergence of the $w^{\pm, k}_j$ is in fact strong, that the conical excess in $\bB_{1/8}$ can be controlled in terms of the limiting $\bar{w}^\pm_j$, and that the $\bar w^\pm_j$ are indeed harmonic. 

\begin{proposition}\label{p:strong-convergence}
Let $T_k, \Sigma_k, \pi_k, \bS_k$, $\bar w^{\pm,k}_j$, and $\bar w^\pm_j$ be as in Corollary \ref{c:scoppia}. Then thew following holds.
\begin{itemize}
\item[(i)] The convergence of $\bar{w}^{k, \pm}_j$ to $\bar{w}^\pm_j$ is strong in the sense that 
\begin{equation}\label{e:strong}
\int_{\pi_0^\pm \cap \bB_{1/8}} (|\bar{w}^\pm_j|^2 + |x|^2 |\nabla \bar{w}^\pm_j|^2) = 
\lim_{k\to\infty} \int_{\bB_{1/8}\cap U^\pm_{\mathcal{W}_k}} (|\bar{w}^{k,\pm}_j|^2 + |x|^2 |\nabla \bar{w}^{k,\pm}_j|^2)\, .
\end{equation}
\item[(ii)] The following estimate holds:
\begin{equation}\label{e:reverse-control}
\limsup_{k\to \infty} \bbE_k^{-1} \bE (T, \bS_k, 0, \sfrac{1}{8}) \leq \sum_j \left(\int_{\pi_0^+\cap \bB_{1/8}} |\bar w^+_j|^2 + \int_{\pi_0^-\cap \bB_{1/8}} |\bar w^-_j|^2 \right)\, .
\end{equation}
\item[(iii)] Each $\bar{w}^\pm_j$ is smooth and harmonic in its domain of definition.
\end{itemize}
\end{proposition}

The proof is verbatim the same of (i), (ii), and (iii) of \cite[Proposition 10.5]{DHMSS}. 

\subsection{Simon's and Wickramasekera's variational identities} We next introduce two important functions, which will be crucial to show that in fact the functions $\bar{w}^\pm_j$ can be suitably extended to harmonic functions over $\pi_0\cap \bB_{1/8}$. The first function is considered by Simon in his original work and it is simply the ``average'' of the $\bar{w}^\pm_j$ in the following sense:
\begin{equation}\label{e:average}
\omega (t,y) := \sum_j (\bar{w}^+_j (t,y)+ \bar{w}^-_j (- t,y))\,, \qquad \mbox{for $(t,y) \in \bB_{1/8} \cap \pi_0^+$}\,.
\end{equation}
The second one is instead introduced by Wickramasekera in \cite{Wic}. We start by recalling that $\pi_0^{\perp_0}$ is one-dimensional, and can therefore be identified with $\mathbb R$. After fixing such identifcation, there exists coefficients $\mu_j^\pm$ with the property that 
\begin{equation}\label{e:coefficients}
\bar l^{\pm}_j (t,y) = \mu_j^\pm t\, .
\end{equation}
Wickramasekera's weighted average takes then the form
\begin{equation}\label{e:weight-average}
\varpi (t,y) = \sum_j (\mu_j^+\bar{w}^+_j (t,y)+ \mu_j^- \bar{w}^-_j (-t,y))\,, \qquad \mbox{for $(t,y) \in \bB_{1/8} \cap \pi_0^+$}\,.
\end{equation}
We note in passing the following obvious consequence of the estimate in Corollary \ref{cor:comparrison excess}.

\begin{lemma}\label{l:opened-and-bounded}
There is a positive constant $C$ depending on $m$ and $Q$ such that 
\begin{align}
C^{-1} \leq &\max \{|\mu^\alpha_j - \mu^{\alpha'}_{j'}|\,\colon\,\alpha, \alpha' \in \pm, 1 \leq j, j' \leq Q \}\nonumber\\
\leq &
2 \max \{|\mu^\alpha_j|\,\colon\,\alpha \in \pm, 1\leq j \leq Q \} \leq 2 C\, .\label{e:opened-and-bounded}
\end{align}
\end{lemma}

Both functions $\omega$ and $\varpi$ satisfy then the same variational identity.

\begin{proposition}\label{p:variational-averages}
Let $\bar{w}^\pm_j$ and $\bar l^\pm_j$ be as in Corollary \ref{c:scoppia} and consider the functions $\omega$ and $\varpi$ introduced in \eqref{e:average} and \eqref{e:weight-average}. Then the following identities hold for every $w\in C^\infty_c (\bB_{1/8} \cap \pi_0, \pi_0^{\perp_0})$ which is even in the variable $t\in V^\perp \cap \pi_0$ and for every direction $v\in V$:
\begin{align}\label{e:variational-averages1}
&\int_{\bB_{1/8}\cap \pi_0^+} \nabla \omega \cdot \nabla \frac{\partial w}{\partial v} = 0\, ,\\\label{e:variational-averages2}
&\int_{\bB_{1/8}\cap \pi_0^+} \nabla \varpi \cdot \nabla \frac{\partial w}{\partial v}= 0\, .
\end{align}
\end{proposition}

\begin{proof} The proof of \eqref{e:variational-averages1} is the same as the proof of (iv) in \cite[Proposition 10.5]{DHMSS}. In particular, if we fix a unit vector $e_{m+1}\in \pi_0^{\perp_0}$ and let  
\begin{align*}
W :=w\, e_{m+1}\, ,
\end{align*}
we observe that $W$ is cylindrical in the sense of \cite[Definition 10.4]{DHMSS}, while the identity \cite[(10.7)]{DHMSS} is equivalent to \eqref{e:variational-averages1}. 

\medskip

The proof of \eqref{e:variational-averages2} follows a slightly different argument. We can definitely argue as in the proof of (iv) in \cite[Proposition 10.5]{DHMSS} to assume, without loss of generality, that $w$ depends only on the $y\in V$ variable in a neighborhood $B_\rho (V)$ of $V$. Hence we let $e$ be a unit vector which spans $V^\perp\cap \pi_0$, we fix a direction $v\in V$ and we consider the vector field 
\begin{equation}
\bar W := \frac{\partial w}{\partial v} e\, .
\end{equation}
Proceeding as in the proof of (iv) in \cite[Proposition 10.5]{DHMSS} we first choose an orientation for $V$, fix a corresponding orientation for the pages of $\tilde{\bS}_k$ so that  $\partial \a{\tilde{\bH}^{k,\pm}_j} = \partial \a{V}$ and hence introduce the cylindrical current 
\[
\bC_k:= \sum_j \a{\tilde{\bH}^{k,+}_j} +\a{\tilde{\bH}^{k,-}_j}\, .
\]
Because $\bar W$ is a derivative along a direction $v\in V$, while $\bC_k$ is invariant under translations in the $v$ direction, we have $\delta \bC_k (\bar W) = 0$. On the other hand we have 
\[
\delta T_k (\bar W) = - \int \vec{H}_T (q)\cdot \bar W (q) d\|T\| (q)\, .
\]
As already argued several times, $\|\vec{H}_T\|_\infty\leq \bA$, while $\vec{H}_T (q) \cdot \bar{W} (q) = \vec{H}_T (q) \cdot \mathbf{p}_{(T \Sigma_q)^\perp} (\bar W (q))$ and $\|\mathbf{p}_{(T \Sigma_q)^\perp} \circ \bar W\|_\infty \leq C \bar A \|\bar W\|_\infty$, so that we reach
\begin{equation}\label{e:variation-A2-again}
|\delta T_k (\bar W) - \delta \bC_k (\bar W)|\leq C \|\bar W\|_\infty \bA_k^2\, .
\end{equation}
Our goal is to show next that 
\begin{align}\label{e:linearization-again}
\lim_{k\to \infty} \frac{1}{\beta_{\pi_0} (\bS_k) \bbE_k^{\sfrac{1}{2}}} (\delta T_k (\bar W) - \delta \bC_k (\bar W)) = -
\int_{\bB_{1/8}\cap \pi_0^+} \nabla \varpi \cdot \nabla \frac{\partial w}{\partial v}\ ,
\end{align}
which, given \eqref{e:bu-assumption-1} and the bound in Lemma \ref{l:angle-bound}, implies \eqref{e:variational-averages2}.

In order to show \eqref{e:linearization-again}, we subsequently fix a $r>0$ and $k$ sufficient large such that  $\rho^k_{\infty} < r$ and introduce the currents 
\begin{align}
T^g_k &:= T_k\res (B_r (V))^c\\
\bC^g_k&:= \bC_k \res (B_r (V))^c\\
T^r_k &:= T_k\res B_r (V)\\
\bC^r_k&:= \bC_k \res B_r (V)\, .
\end{align}
Note in particular that $T_k^g$ is a multigraph over $\pi_0$

We will then split our proof of \eqref{e:linearization-again} in two separate parts, in particular we will show that 
\begin{equation}\label{e:close-to-spine}
\limsup_{k\to \infty} \beta_{\pi_0} (\bS_k)^{-1} \bbE_k^{-\sfrac{1}{2}} \left|\int {\rm div}_{T_k} \bar W d\|T_k^r\| - \int {\rm div}_{\bC_k} \bar W d\|\bC_k^r\|\right| \leq C r^{\sfrac{1}{2}}
\end{equation}
for a constant $C$ independent of $r$, and that 
\begin{align}
\lim_{k\to \infty} \beta_{\pi_0} &(\bS_k)^{-1} \bbE_k^{-\sfrac{1}{2}} \left(\int {\rm div}_{T_k} \bar W d\|T_k^g\| - \int {\rm div}_{\bC_k} \bar W d\|\bC_k^g\|\right)\nonumber\\ 
= & - \int_{\bB_{1/8}\cap \pi_0^+\setminus B_r (V)} \nabla \varpi \cdot \nabla \frac{\partial w}{\partial v}\, .\label{e:away-from-spine}
\end{align}
From \eqref{e:close-to-spine},\eqref{e:away-from-spine} and the facts that $\nabla \varpi\in L^2$ and $r$ is arbitrary, we conclude \eqref{e:linearization-again}.

Recall that $\bar W$ is directed along $e\in \pi_0\cap V^\perp$, while, in the region $B_\rho (V)$, it does not depend on directions orhogonal to $V$. In particular, on the latter region we have $\tr({ {\bf p}_{\pi_0} D\bar{W} {\bf p}_V })=0$. We can thus estimate
\begin{align*}
|\Div_{\pi}(\bar{W}) (q)| &= |\tr({{\bf p}_{\pi} {\bf p}_{\pi_0} D\bar{W} (q) {\bf p}_V })| = |\tr({{\bf p}^\perp_{\pi} {\bf p}_{\pi_0} D\bar{W} (q) {\bf p}_V })|\\
&\le |\tr({ {\bf p}_{\pi}^\perp {\bf p}_{\pi_0})}|\,|\tr({{\bf p}_{\pi}^\perp {\bf p}_{V})}|\,|D \bar W| (q)
\le C |{\bf p}_{\pi_0} - {\bf p}_\pi| |{\bf p}_V \cdot {\bf p}_{\pi^\perp}|\, |D\bar W (q)| \,,
\end{align*}
for every $q\in B_\rho (V)$ and for every $m$-dimensional plane $\pi$. Recall that $r<\rho$. In particular, since $V$ is a subset of any tangent plane to $\bC_k$, we immediately conclude 
\begin{equation}\label{e:invariance}
\int {\rm div}_{\bC_k} \bar W d\|\bC_k^g\| = 0\, .
\end{equation}
Moreover we can use Proposition \ref{prop:density-est} to estimate
\begin{align}
& \left| \int \Div_{T_k} \bar{W} \, d\norm{T^r_k}\right| \nonumber\\
    \le & C \left(\int_{B_\frac14 \cap B_r(V)} |{\bf p}_{\vec T_k} - {\bf p}_{\pi_0}|^2 d\norm{T^r_k} \right)^\frac12 \left(\int_{B_\frac14 \cap B_r(V)} |{\bf p}_V\cdot {\bf p}_{\vec T_k}^\perp|^2 d\norm{T^r_k} \right)^\frac12 \nonumber\\
    \le & C \bbE_k^{\frac12} \left(\int_{B_\frac14 \cap B_r(V)} |{\bf p}_{\vec T_k} - {\bf p}_{\pi_0}|^2 d\norm{T^r_k} \right)^\frac12\,\,.\label{e:not-yet-gained-r}
\end{align} 
Next, using \eqref{eq.other-propagation} in Lemma \ref{lem.propagation1} and Lemma \ref{l:angle-bound}, for $k$ large enough we have 
\[
\bE (T_k, \pi_0, y, 2r) \leq 2 \bE (T_k,\pi_0, 0, 1) \leq C \beta_{\pi_0} (\bS_k)^2
\]
for every $y\in \bB_{1/4} \cap V$. Subsequently, we can use Allard's tilt-excess estimate \cite[Proposition 4.1]{D-Allard} to conclude that
\[
\int_{\bB_r (y)} |{\bf p}_{\vec T_k} - {\bf p}_{\pi_0}|^2 d\norm{T^r_k} \leq C r^m (\beta_{\pi_0} (\bS_k)^2 + \bA_k^2)
\]
for every $y\in \bB_{1/4} \cap V$ (provided $k$ is large enough). Since we can cover $\bB_{1/4} \cap V$ with $C r^{-m+1}$ balls of radius $r$ centered at points $y\in V \cap \bB_{1/4}$, we clearly conclude that 
\begin{equation}\label{e:gained-r}
\limsup_{k\to\infty} \beta_{\pi_0} (\bS_k)^{-2}\int_{B_\frac14 \cap B_r(V)} |{\bf p}_{\vec T_k} - {\bf p}_{\pi_0}|^2 d\norm{T^r_k} 
\leq C r\, .
\end{equation}
Combining \eqref{e:invariance}, \eqref{e:not-yet-gained-r}, and \eqref{e:gained-r}, we then get \eqref{e:close-to-spine}.

In order to prove \eqref{e:away-from-spine} we observe that, for $r$ large enough, $T^r_g\res \bB_{1/8} = T \res (\bB_{1/8} \setminus B_r (V))$ is the union of the $2Q$ graphs over $\pi_0^\pm\cap \bB_{1/8}\setminus B_r (V)$ of the functions 
\[
q\mapsto v^{k,\pm}_j (q) = \big(w^{k,\pm}_j (q) + \tilde{l}^{k,\pm}_j (q) , \underbrace{\Psi_k (q, w^{k,+\pm}_j (q) + \tilde{l}^{k,\pm}_j (q))}_{=: \psi^{k,\pm}_j (q)}\big)\in \pi_0^{\perp_0}\times T_0 \Sigma^\perp\, .
\]
while $\bC_k$ is the union of the graphs, over the same domains, of the functions
\[
q\mapsto \ell^{k, \pm}_j (q) = \big( \tilde{l}^{k, \pm}_j (q), 0\big)\in \pi_0^{\perp_0}\times T_0 \Sigma^\perp\, .
\]
In particular we can write 
\[
\delta T_k^r (\bar W) - \delta \bC_k^r (\bar W) = \sum_j \big(\delta \bG_{v^{k,+}_j} (\bar W) - \delta \bG_{\ell^{k,+}_j} (\bar W)\big) + \sum_j \big(\delta \bG_{v^{k,-}_j} (\bar W) - \delta \bG_{\ell^{k,-}_j} (\bar W)\big)\, ,
\]
and reduce the proof of \eqref{e:variational-averages2} to 
\begin{equation}
\lim_{k\to \infty} \bE_k^{-\sfrac{1}{2}}
\beta_{\pi_0} (\bS_k)^{-1} \big(\delta \bG_{v^{k,\pm}_j} (\bar W) - \delta \bG_{\ell^{k,\pm}_j} (\bar W)\big)
= - \int_{\bB_{1/8}\cap \pi_0^\pm\setminus B_r (V)} \mu^\pm_j \nabla \bar{w}^{\pm}_j \cdot \nabla \frac{\partial w}{\partial v}\, 
\end{equation}
(note that summing over $\pm$ we then use the fact that the function is even to achieve \eqref{e:variational-averages2}).
The proof is the same for all $2Q$ functions: we will therefore restrict to the case $(+,1)$ and, in order to simplify our notation, we will drop the indices $+$ and $1$, so that our functions become
\begin{align}
v_k (z) &= (w_k (z) + \tilde l_k (z), \psi_k (z))\, ,\\
\ell_k (z) &= (\tilde l_k (z), 0)\, .
\end{align}
It is important to recall that 
\begin{align}
\|\tilde l_k - \beta_{\pi_0} (\bS_k) \mu^+_1 \bar{l}^+_1\|_{C^1} &= o (\beta_{\pi_0} (\bS_k)) \,, \label{e:stima_pallosa_1}\\
\|w_k - \bbE_k^{\sfrac{1}{2}} \bar w^+_1\|_{C^1} &= o (\bbE_k^{\sfrac{1}{2}})\,,\label{e:stima_pallosa_2}\\
\|\psi_k\|_{C^2} &= O (\bA_k) = o (\beta_{\pi_0}(\bS_k)^{\sfrac12})\, , \label{e:stima_pallosa_3}\\
\|v_k\|_{C^1} + \|\tilde l_k\|_{C^1} &= O (\beta_{\pi_0} (\bS_k))\, \label{e:stima_pallosa_4}\, ,\\
\|w_k\|_{C^1} & = O (\bbE_k^{\sfrac{1}{2}}) \label{e:stima_pallosa_5}.
\end{align}
We denote by $\zeta$ the function $\frac{\partial w}{\partial v}$ and consider, for small $\varepsilon$, the diffeomorphism $\Phi_\varepsilon (p) = p+ \varepsilon \bar W (p)$ of $\mathbb R^{m+n}$ onto itself and the diffemorpshim $\Psi_\varepsilon (z) = z - \varepsilon \zeta(z) e$. If a current is the graph $\bG_v$ of a $C^1$ function $v$ over some domain $\Omega$ of $\pi_0$, then $(\Phi_\varepsilon)_\sharp \bG_v$ is the graph of $v_\varepsilon := v \circ \Psi_\varepsilon$. The variation $\delta \bG_v (\bar W)$ can then be computed as 
\begin{equation}\label{e:internal-variation}
\delta \bG_v (\bar W) = \left.\frac{d}{d\varepsilon}\right|_{\varepsilon =0} \int_{\Psi_\epsilon^{-1}(\Omega)} \mathcal{A} (Dv_\varepsilon)  = - \int_\Omega \left[\frac{\partial \mathcal{A}}{\partial A} (Dv) : (Dv\cdot e\otimes \nabla \zeta) - \mathcal{A} (Dv) \frac{\partial \zeta}{\partial e}\right] \, , 
\end{equation}
where $\mathcal{A} (A)$ is the area integrand. The latter can be written explicitely as 
\[
\mathcal{A} (A) := \sqrt{1+|A|^2 + \sum_{M \in \mathcal{M}_i (A), i\geq 2}(\det M)^2}\, ,
\]
where $\mathcal{M}_i (A)$ denotes the set of $i\times i$ minors of $A$.

Observe first that 
\begin{equation}\label{e:stima_pallosa_11}
\mathcal{A} (D \ell_k) \frac{\partial \zeta}{\partial e}
= \sqrt{1+|\nabla \tilde l_k|^2} \frac{\partial \zeta}{\partial e}
\end{equation}
and 
\begin{equation}\label{e:stima_pallosa_12}
\frac{\partial \mathcal{A}}{\partial A} (D\ell_k) : (D\ell_k \cdot e \otimes \nabla \zeta) =
\frac{(\nabla \tilde l_k \cdot e) (\nabla \tilde l_k \cdot \nabla \zeta)}{\sqrt{1+|\nabla \tilde l_k|^2}}\, . 
\end{equation}
Note next that, for any $M\in \mathcal{M}_i (Dv_k)$ with $i\geq 2$, $\det M$ is the product of $2i$ entries of $Dv$, of which at least two are partial derivatives of $\psi_k$. Taking then into consideration \eqref{e:stima_pallosa_1}-\eqref{e:stima_pallosa_5} we get
\begin{align}
\mathcal{A} (Dv_k) \frac{\partial \zeta}{\partial e} &= \sqrt{1+|\nabla \tilde l_k|^2} \frac{\partial \zeta}{\partial e} +  \frac{(\nabla \tilde{l}_k \cdot \nabla w_k) \frac{\partial \zeta}{\partial e}}{\sqrt{1+ |\nabla \tilde l_k|^2}} + o (\beta_{\pi_0} (\bS_k) \bbE_k^{\sfrac{1}{2}})\label{e:stima-pallosa-13}\\
\frac{\partial \mathcal{A}}{\partial A} (Dv_k)
&= \frac{Dv_k}{\sqrt{1+|\nabla \tilde l_k|^2}}
+ O (\bA_k + \bbE_k + \beta_{\pi_0} (\bS_k)^2 \bbE_k^{\sfrac{1}{2}})\, . \label{e:stima-pallosa-14}
\end{align}
Using then \eqref{e:stima_pallosa_4} and \eqref{e:stima_pallosa_5} we gain furthermore the expansion
\begin{align}
 & \frac{\partial \mathcal{A}}{\partial A} (Dv_k) : (Dv_k \cdot e \otimes \nabla \zeta)\nonumber\\
= &\frac{(\nabla \tilde{l}_k \cdot e) (\nabla \tilde l_k \cdot \nabla \zeta) + (\nabla \tilde l_k \cdot e)(\nabla w_k \cdot \nabla \zeta)+(\nabla w_k \cdot e) (\nabla \tilde{l}_k\cdot \nabla\zeta)}{\sqrt{1+ |\nabla \tilde{l}_k|^2}}
+ o(\beta_{\pi_0} (\bS_k) \bbE_k^{\sfrac{1}{2}})\, .\label{e:stima-pallosa-15}
\end{align}
Inserting \eqref{e:stima_pallosa_11}, \eqref{e:stima_pallosa_12}, \eqref{e:stima-pallosa-13}, and \eqref{e:stima-pallosa-15} in \eqref{e:internal-variation} we then get 
\begin{align}
 &(\delta \bG_{v_k} (\bar W) - \delta \bG_{\ell_k} (\bar W)\nonumber\\
= &- \int_{\pi_0^+ \cap \bB_{1/8}\setminus B_r (V)} 
\frac{(\nabla \tilde l_k \cdot e)(\nabla w_k \cdot \nabla \zeta)+(\nabla w_k \cdot e) (\nabla \tilde{l}_k\cdot \nabla\zeta) - (\nabla w_k \cdot \nabla \tilde{l}_k)\frac{\partial \zeta}{\partial e}}{\sqrt{1+ |\nabla \tilde{l}_k|^2}}\nonumber\\
& + o(\beta_{\pi_0} (\bS_k) \bbE_k^{\sfrac{1}{2}})\, .
\end{align}
However, since $\nabla \tilde{l}_k = \tilde\mu^{k,+}_j e$ for some real numbers $\tilde\mu^{k,+}_j$, we easily see that in fact 
\[
(\nabla w_k \cdot e) (\nabla \tilde{l}_k\cdot \nabla\zeta) =
\tilde\mu^{k,+}_j \frac{\partial w_k}{\partial e} \frac{\partial \zeta}{\partial e} = (\nabla w_k \cdot \nabla \tilde{l}_k)\frac{\partial \zeta}{\partial e}
\]
In particular
\begin{align}
 (\delta \bG_{v_k} (\bar W) - \delta \bG_{\ell_k} (\bar W))
= &- \int_{\pi_0^+ \cap \bB_{1/8}\setminus B_r (V)} 
\frac{(\nabla \tilde l_k \cdot e)(\nabla w_k \cdot \nabla \zeta)}{\sqrt{1+ |\nabla \tilde{l}_k|^2}}
+ o(\beta_{\pi_0} (\bS_k) \bbE_k^{\sfrac{1}{2}})\nonumber\\
= & - \int_{\pi_0^+ \cap \bB_{1/8}\setminus B_r (V)} 
(\nabla \tilde l_k \cdot e)(\nabla w_k \cdot \nabla \zeta)+ o(\beta_{\pi_0} (\bS_k) \bbE_k^{\sfrac{1}{2}})\, .
\end{align}
We now use \eqref{e:stima_pallosa_1} and \eqref{e:stima_pallosa_2} to conclude that 
\[
(\beta_{\pi_0} (\bS_k)^{-1} \bbE_k^{-\sfrac{1}{2}})
(\nabla \tilde l_k \cdot e)(\nabla w_k \cdot \nabla \zeta)
\to \mu_1^+ \nabla \bar w^+_1 \cdot \nabla \zeta
\]
uniformly on $\bB_{1/8}\cap \pi_0^+ \setminus B_r (V)$. In particular we finally get 
\[
\lim_{k\to \infty} (\beta_{\pi_0} (\bS_k)^{-1} \bbE_k^{-\sfrac{1}{2}})  (\delta \bG_{v_k} (\bar W) - \delta \bG_{\ell_k} (\bar W))
= - \mu_1^+ \int_{\bB_{1/8}\cap \pi_0^+ \setminus B_r (V)} \nabla \bar w^+_1 \cdot \nabla \zeta\, ,
\]
which completes the proof.
\end{proof}

\section{Proof of \texorpdfstring{Proposition \ref{p:decay-2}}{decay-2}: Decay for the linearization}

The aim of this section is to prove the fundamental integral decay property of the blow-up maps $\bar w^\pm_j$ which will allow us to conclude the proof of Proposition \ref{p:decay-2}. 

\begin{proposition}\label{p:linear-decay}
There exists a constant $C\geq 0$ depending only upon $m$ and $Q$, with the following properties. Let $\bar w^\pm_j$ be the maps in Corollary \ref{c:scoppia}. Then there are:
\begin{itemize}
\item[(i)] $2Q$ linear maps $a^\pm_j : \pi_0 \to \pi_0^{\perp_0}$ which vanish on $V$,
\item[(ii)] a linear map $b_v: V\to \pi_0^{\perp_0}$,
\item[(iii)] and a linear map $b_o: V \to \pi_0 \cap V^\perp$ 
\end{itemize}
such that 
\begin{equation}\label{e:bound-on-new-linear-maps}
\|b_o\|_{C^1}+\|b_v\|_{C^1} + \|a^\pm_j\|_{C^1}\leq C
\end{equation}
and 
\begin{equation}\label{e:decad-lineare}
\int_{\pi_0^\pm\cap \bB_\rho} \big|\bar{w}^\pm_j (t,y) - a^\pm_j (t) - (b_v (y) - \bar l_j^{\pm} (b_o (y)))\big|^2\, dy\, dt \leq C \rho^{m+4} \qquad \forall \rho<\frac{1}{32}\, .
\end{equation}
\end{proposition}

\subsection{Smoothness and properties of Simon's and Wickramasekera's averages} In this subsection we use the variational identities \eqref{e:variational-averages1} and \eqref{e:variational-averages2} to conclude the following 

\begin{lemma}\label{l:media}
Let $\bar w$ be as in Corollary \ref{c:scoppia} and define $\omega$ and $\varpi$ as in \eqref{e:average} and \eqref{e:weight-average}. Then:
\begin{itemize}
    \item[(i)] $\omega$ and $\varpi$ are harmonic and can be extended to harmonic functions (still denoted $\omega$ and $\varpi$) on $\bB_{1/8}\cap \pi_0$ with the property that $\frac{\partial^2 \omega}{\partial t\partial v} = \frac{\partial^2 \varpi}{\partial t\partial v}=0$ on $V\cap \bB_{1/8}$ for every $v\in V$;
    \item[(ii)] $\omega (0) = \varpi (0) = 0$.
\end{itemize}
\end{lemma}

The proof is verbatim the same as the ones for the analogous claims in \cite[Lemma 11.2]{DHMSS} and is left to the reader.
 We  just remark that corollary \ref{c:scoppia} (iii) together with \eqref{bu:nonconcentration w12} implies that 
\[\int_{B_{\frac{1}{8}} \cap \pi_0^+} \frac{|\nabla \varpi|^2}{t^\frac12} \,dtdy \le C\,.\]
Hence $\partial_t\varpi$ has a well-defined trace on $V$.

\subsection{Proof of Proposition \texorpdfstring{\ref{p:linear-decay}}{decay-linear}} We start by claiming the existence of a $t_k \in [2^{-k-1}, 2^{-k}]$ such that the following estimate holds for every $t\in [2^{-k+1}, 2^{-4}]$,
\begin{equation}\label{e:first-claim}
\int_t^{2t}\int_{\bB_{1/16}\cap V} |(\bar \xi_v (\tau, y) - \bar{l}_j^\pm (\bar\xi_o (\tau, y))) - (\bar\xi_v (t_k, y) - \bar{l}_j^\pm (\bar\xi_o (t_k, y)))|^2\, dy\, d\tau
\leq C |t|^{\sfrac{5}{2}} \, .
\end{equation}
%{\color{red} Camillo: Guys, here's the proof of \eqref{e:first-claim} I had in mind. Am I missing something?} 
Indeed denote by $f$ the function $\bar w^\pm_j - \bar l^\pm_j$ and by $g$ the function $\bar\xi_v  - \bar l_j^\pm\circ \bar\xi_o$ and first of all use Fubini and \eqref{bu:nonconcentration w12} to choose a $t_k \in [2^{-k}, 2^{-k+1}]$ such that 
\begin{equation}\label{e:Fubini}
\int_{\bB_{1/16}\cap V} |f (t_k, y) - g (t_k, y)|^2
dy\, dt\, \leq C 2^{-3 k/2}\, .
\end{equation}
Hence integrate in $t$ and use the second part of \eqref{bu:nonconcentration w12} to prove 
\[
\int_{2^{-k}}^{2^{-k+1}} \int_{\bB_{1/16}\cap V} |f (t, y)-f(t_k,y)|^2 dy\, dt \leq C 2^{-5 k/2}
\]
Considering that, again by \eqref{bu:nonconcentration w12}
\[
\int_{2^{-k}}^{2^{-k+1}} \int_{\bB_{1/16}\cap V}  |f(t,y) - g (t,y)|^2 dy\, dt \leq C 2^{-5 k/2}\, ,
\]
we can estimate
\begin{align}
&\int_{2^{-k}}^{2^{-k+1}} \int_{\bB_{1/16}\cap V} |g (t, y) - g (t_k, y)|^2 dy\, dt\nonumber\\
\leq & 2 \int_{2^{-k}}^{2^{-k+1}} \int_{\bB_{1/16}\cap V} (|g (t, y) - f (t, y)|^2 + |f (t, y) - f (t_k, y)|^2 + |f (t_k, y) - g(t_k, y)|^2\, dy\, dt\nonumber\\
\leq & C 2^{-5k/2}\, .\label{e:dyadic-1}
\end{align}
Observe also that we can use the second part of  \eqref{bu:nonconcentration w12} again to prove 
\[
\int_{V \cap \bB_{1/16}} |f (t_j, y)- f (t_k, y)|^2\, dy
\leq 2^{-3j/2} \qquad \forall j \leq k\, .
\]
Combined with \eqref{e:Fubini} we then gain
\begin{equation}\label{e:dyadic-2}
\int_{V \cap \bB_{1/16}} |g (t_j, y)- g (t_k, y)|^2\, dy
\leq 2^{-3j/2} \qquad \forall j \leq k\, .
\end{equation}
We can now combine \eqref{e:dyadic-1} and \eqref{e:dyadic-2} to reach
\begin{align*}
&\int_{2^{-k}}^{2^{-j}} |g (t,y) - g (t_k, y)|^2\, dt\, dy\\
\leq & \sum_{j+1 \leq i \leq k} 
\int_{2^{-i}}^{2^{-i+1}} 2 (|g (t,y)-g (t_i, y)|^2 + |g (t_i, y)- g (t_k, y)|^2)\, dy\, dt \leq C \sum_{j+1 \leq i \leq k} 2^{-5i/2}\, .
\end{align*}
%{\color{red} Camillo: Guys, here is how I would then derive the next ones}.
Recall next the definition of the coefficients $\mu_j^\pm$, so that $\bar{l}^\pm_j (\bar\xi_o (t, y) ) = \mu^\pm_j (\bar\xi_o (t, y))$ upon identifying $\pi_0 \cap V^\perp$ with $\R$. Use then \eqref{e:opened-and-bounded} in Lemma \ref{l:opened-and-bounded} to conclude the existence of two indices in the collection $\{(\pm, j)\}$ whose absolute value of the difference is larger than an absolute positive constant. Let $\bar\mu$ and $\hat\mu$ be the corresponding coefficients and observe that the inverse of the matrix 
\[
M := \left(
\begin{array}{ll}
1 & - \bar \mu \\
1 & - \hat \mu
\end{array}
\right)
\]
is bounded by a universal constant. In particular we can write 
$\bar\xi_v$ and $\bar\xi_o$ as a linear combination of $\bar\xi_v - \bar \mu \bar\xi_o$ and $\bar\xi_v - \hat \mu \bar\xi_o$ to pass from \eqref{e:first-claim} to 
\begin{align}
\int_t^{2t} \int_{\bB_{1/16}\cap V} |\bar\xi_v (\tau, y) - \bar\xi_v (t_k, y)|^2\, dy\, d\tau
& \leq C |t|^{\sfrac{5}{2}}\\
\int_t^{2t}\int_{\bB_{1/16}\cap V} |\bar\xi_o (\tau, y) -\bar\xi_o (t_k, y)|^2\, dy\, d\tau
& \leq C |t|^{\sfrac{5}{2}}\, .
\end{align}
Note moreover that from the above estimates it follows that the sequences $\bar\xi_v (t_k, \cdot)$ and $\bar\xi_o (t_k, \cdot)$ are Cauchy in $L^2 (V\cap \bB_{1/16})$ and their limits are bounded functions
\begin{align*}
& \bar\xi_v (0, \cdot): V\cap \bB_{1/16} \to \pi_0^{\perp_0}\\
& \bar\xi_o (0, \cdot): V\cap \bB_{1/16} \to \pi_0 \cap V^\perp
\end{align*}
with the property that 
\begin{align}
\int_{\bB_{1/16}\cap V} |\bar\xi_v (t_k, y) - \bar\xi_v (0, y)|^2\, dy
& \leq C 2^{-3k/2}\\
\int_{\bB_{1/16}\cap V} |\bar\xi_o (t_k, y) - \bar\xi_o (0, y)|^2\, dy
& \leq C 2^{-3k/2}
\end{align}
In particular we can combine this information again with \eqref{bu:nonconcentration w12} to estimate
\[
\int_{2^{-k}}^{2^{-k+1}} |\bar w^\pm_j (t, y)- (\bar\xi_v (0,y) - \mu^\pm_j \bar\xi_o (0,y))|^2\, dt\, dy \leq C 2^{-5k/2}
\]
Summing over all the dyadic scales we then conclude
\begin{equation}\label{e:weighted-trace}
\int_{\bB_{1/16}\cap \pi_0^\pm} \frac{|\bar w^\pm_j (t, y)- (\bar\xi_v (0,y) - \mu^\pm_j \bar\xi_o (0,y))|^2}{|t|^{\frac{9}{4}}} dt\, dy \leq C \sum_{k=4}^\infty 2^{-k/4}\leq C\, .
\end{equation}
We next introduce the coefficients 
\begin{align*}
\alpha &:= \sum_j (\mu^+_j + \mu^-_j)\\
\beta &:= \sum_j ((\mu^+_j)^2 + (\mu^-_j)^2)\, ,
\end{align*}
and use \eqref{e:average} and \eqref{e:weight-average} to show that 
\begin{align}
\int_{\bB_{1/16}\cap \pi_0^+} \frac{|\omega (t, y) - (2Q \bar\xi_v (0, y) - \alpha \bar\xi_o (0,y))|^2}{|t|^{\frac{9}{4}}} dy\, dt & \leq C\\
\int_{\bB_{1/16}\cap \pi_0^+} \frac{|\varpi (t, y) - (\alpha \bar\xi_v (0, y) - \beta \bar\xi_o (0,y))|^2}{|t|^{\frac{9}{4}}} dy\, dt & \leq C
\end{align}
Consider moreover the $2\times 2$ matrix 
\[
M := \left( 
\begin{array}{ll}
2Q & - \alpha\\
\alpha & - \beta
\end{array}
\right)
\]
and observe that, by Cauchy-Schwartz and \eqref{e:opened-and-bounded}, 
\[
C^{-1} \leq -\det M \leq |M|^2 \leq C\, .
\]
In particular the inverse 
\[
M^{-1} = \left(
\begin{array}{ll}
\alpha' & \beta'\\
\gamma' & \delta'
\end{array}
\right)
\]
satsifies $|M^{-1}|\leq C$. Now we gather therefore 
\begin{align}
\int_{\bB_{1/16}\cap \pi_0^+} \frac{|\bar\xi_v (0, y)
- (\alpha' \omega (t,y) + \beta' \varpi (t,y))|^2}{|t|^{\frac{9}{4}}} dy\, dt & \leq C\\
\int_{\bB_{1/16}\cap \pi_0^+} \frac{|\bar\xi_o (0, v) - (\gamma' \omega (t, y) + \delta \varpi (t, y))|^2}{|t|^{\frac{9}{4}}} dy\, dt & \leq C\, .
\end{align}
Therefore $\bar\xi_v (0, \cdot)$ is the trace of the harmonic function 
\[
h_v := \alpha' \omega + \beta' \varpi
\]
while $\bar\xi_o (0, \cdot)$ is the trace of the harmonic function 
\[
h_o := \gamma' \omega + \delta' \varpi\, .
\]
Since by Lemma \ref{l:media}(i) both functions can be extended as harmonic functions on $\bB_{1/8}\cap \pi_0$ and their $L^2$ norms are bounded by a universal constant, we conclude that 
\begin{align}
|\nabla h_o (0)| + |\nabla h_v (0)|
&\leq C\\
\|D^2 h_o\|_{C^0 (\bB_{1/16}\cap \pi_0)} + \|D^2 h_v \|_{C^0 (\bB_{1/16}\cap \pi_0)} &\leq C\, .
\end{align}
Next, consider the harmonic functions
\[
\hat{w}^\pm_j := \bar w^\pm_j - (h_v - \mu^\pm_j h_o)\, .
\]
Observe that the trace of these harmonic function on $V\cap \bB_{1/16}$ is identically $0$. So, by Schwartz reflection they can be extended to an odd harmonic function on $\pi_0 \cap \bB_{1/16}$. Consider thus that 
\begin{align}
|\nabla \hat{w}^\pm_j (0)| &\leq C\\
\|D^2 \hat{w}^\pm_j\|_{C^2 (\bB_{1/32} \cap \pi_0)} & \leq C\, .
\end{align}
Moreover $h_o (0) = h_v (0) =0$ by Lemma \ref{l:media} and by \eqref{e:weighted-trace}, while $\hat{w}^\pm_j (0) =0$ because the trace of $\hat{w}^\pm_j$ on $V$ vanishes identically, we conclude that 
\begin{equation}\label{e:e:decay}
\int_{\bB_\rho \cap \pi_0^\pm} 
|\bar w^\pm_j (z) - \nabla \hat{w}^\pm_j (0)\cdot z - (\nabla h_v (0) \cdot z- \mu^\pm_j \nabla h_o (0)\cdot z)|^2\, dz 
\leq C \rho^{m+4}\, .
\end{equation}
Next observe that $\nabla \hat{w}^\pm_j (0)$ must be directed along the unit vector $e\in \pi_0^+\cap V^\perp$, given that $\hat{w}^\pm_j$ vanishes identically on $V$. Thus, if we introduce the orthonormal coordinates $y_1, \ldots, y_{m-1}$ on $V$, \eqref{e:decad-lineare} holds for the linear functions
\begin{align}
a^\pm_j (t) &= \left(\frac{\partial \hat{w}^\pm_j}{\partial e} (0)+ \frac{\partial h_v}{\partial e} (0) - \mu_j^\pm \frac{\partial h_o}{\partial e} (0)\right) t\\ \label{vertical part of rotation}
b_v (y) &= \sum_i \frac{\partial h_v}{\partial y_i} (0) y_i\\ \label{horizontal part of rotation}
b_o (y) &= \sum_i \frac{\partial h_o}{\partial y_i} (0) y_i\, .
\end{align}

\section{Proof of Proposition \texorpdfstring{\ref{p:decay-2}}{decay-2}: final step} \label{s:final-step}

In this section we finally complete the proof of Proposition \ref{p:decay-2}.

We let $r_2>0$ be a fixed small radius, whose choice will be specified later, and argue by contradiction. Assuming that the proposition is false, we find a blow-up sequence $(T_k, \Sigma_k, \pi_k, \bS_k)$, together with linear subspaces $\tau_0 = T_0 \Sigma_k\supset \pi_0\supset V (\bS_k) = V$ as in Definition \ref{d:scoppia}, with the additional property that 
\begin{equation}\label{e:key-contradiction-1}
\bbE (T_k, 0, r_2)\geq \frac{1}{2} \bbE_k
\end{equation}
We can therefore, upon extraction of a suitable subsequence, assume that Corollary \ref{c:scoppia} and Proposition \ref{p:strong-convergence} apply, and we let $\bar w^{\pm}_j: \bB_{1/8} \to \pi_0^{\perp_0}$ be the corresponding functions. We then consider the linear maps $b_v$, $b_o$, and $a^\pm_j$ produced by Proposition \ref{p:linear-decay}. 

The linear maps $a^\pm_j$ are used to ``adjust'' the pages of ${\tilde{\bS}}_k$ in the following way. For each $j$ we consider the half-spaces $\bar \bH^{\pm,k}_j$ given by the graphs over $\pi_0^\pm$ of the linear functions
\[
\tilde{l}^\pm_j + \bbE_k^{\sfrac{1}{2}} a^\pm_j\, .
\]
Hence we let $\bar \bS_k$ be the open book with pages $\bar \bH^{\pm,k}_j$. Note that the open book $\bar \bS_k$ has the same spine $V$ as $\tilde{\bS}_k$.

The linear maps $b_o$ and $b_v$ will instead be used to rotate suitably $\bar \bS_k$. More precisely, with a slight abuse of notation we let $b_o$ and $b_v$ denote the vectors such that $b_o(y) = b_o \cdot y$ and $b_v(y) = b_v \cdot y$ for every $y \in V$, see \eqref{vertical part of rotation}-\eqref{horizontal part of rotation}, and we then let $V \to \tau_0=V\oplus \R e \oplus \R e_{m+1}$ be the linear map
\[
V \ni v\mapsto (e\otimes \beta_{\pi_0}(\tilde\bS_k)^{-1}\, b_o) (v) + (e_{m+1}\otimes b_v) (v) \;\in \pi_0^{\perp_0}\times (V^\perp \cap \pi_0) = V^\perp\cap \tau_0 \subset \tau_0\, .
\]
Observe that there is a unique skew-symmetric linear map $b_k: \tau_0 \to \tau_0$ which extends it, i.e. $b_k=(e\otimes \beta_{\pi_0}(\tilde\bS_k)^{-1}\, b_o + e_{m+1} \otimes b_v ) - (\beta_{\pi_0}(\tilde\bS_k)^{-1}\,b_o\otimes e  + b_v \otimes e_{m+1}) $. For every real number $s$ we consider the exponential map $\exp (s b_k) : \tau_0 \to \tau_0$ and observe that it is an element of ${\rm SO} (\tau_0)$. We then set $\hat \bS_k := \exp (\bbE_k^{\sfrac{1}{2}} b_k) (\bar \bS_k)$. Observe that, by construction, $\hat \bS_k$ is an open book in $\mathscr{B} (0, \Sigma_k)$. The proof of Proposition \ref{p:decay-2} will then be completed by the following\\

\begin{lemma}\label{l:contradiction}
Let $(T_k, \Sigma_k, \pi_k, \bS_k)$ be the blow-up sequence fixed above and consider the open book $\hat\bS_k$ just defined. Then there is a constant $C$, independent of $k$ and $\rho$, such that 
\begin{equation}\label{e:key-contradiction-2}
\limsup_{k\to \infty} \bbE_k^{-1} (\bbE (T_k, \hat \bS_k, 0, \rho)) \leq C \rho^2
\end{equation}
for every fixed $\rho < \frac{1}{32}$. 
\end{lemma}

Indeed, given that $\hat\bS_k \in \mathscr{B} (0, \Sigma_k)$, from \eqref{e:key-contradiction-2} we conclude
\begin{equation}\label{e:key-contradiction-3}
\limsup_{k\to \infty} \bbE_k^{-1} (\bbE (T_k, 0, r_2)) \leq C r_2^{2}\, .
\end{equation}
Since $C$ is independent of $r_2$, by choosing $r_2$ sufficiently small we get that \eqref{e:key-contradiction-1} and \eqref{e:key-contradiction-3} are in contradiction, thus proving Proposition \ref{p:decay-2}.

\begin{proof}[Proof of Lemma \ref{l:contradiction}] Let $\Phi_s(x,z)= (x,z) + s(X_o(x,z), X_v(x,z) + O(s^2)$ the flow of the vector field $X(x,z)=(X_o(x,z),X_v(x,z)) \in \R^m\times \R^n$ and $f\colon \Omega\subset \R^m \to \R^n$ a $C^1$ regular map. Then there exists $f_s\colon \Omega_s \to \R^n$ such that 
\[ \bG_{f_s}(\Omega_s) = {\Phi_s}_\sharp \bG_f(\Omega)\]
where we have the expansion
\begin{equation}\label{eq.transformedgraph}f_s(x) = f(x-s X_o(x,f(x))) + s X_v(x,f(x)) + O(s^2)\,.\end{equation}
Indeed, note that $\phi_s(x) =\mathbf{p}_{\pi_0}\Phi_s(x,f(x))= x + s X_o(x,f(x)) + O(s^2)$ is a $C^1$-diffeomorphism from $\Omega$ to $\Omega_s$ with inverse 
\[\phi^{-1}_s(x) = x- sX_o(x,f(x)) + O(s^2)\,, \]
and observe that
 \[f_s(x) = \mathbf{p}_{\pi_0^\perp} \Phi_s(\cdot, f(\cdot))\circ \phi_s^{-1}(x)\,,\]
 has precisely the claimed properties.
 
 \medskip
 
Next let us fix a $k$ in the contradiction sequence, which we will not write in the following, and  apply \eqref{eq.transformedgraph} to one of the linear functions $f=(\tilde l_j + \bbE^\frac12 a_j)\,e_{m+1}$ over the domain $\Omega=\pi_0^+$ and with $s=\bbE^{\frac12}$ and $\Phi_s$ given by the rotation $\exp(\bbE^\frac12 b)$. Notice that $\pi_0 \cap \{ t>r\} \subset \Omega_{\bbE^\frac12}$ for sufficient small $\bbE$, and so
\begin{equation}\label{eq.trasnformedlinearfunction}
    \hat{l}_j(t,y):=(f)_{\bbE^\frac12} (t,y) = \left(\tilde l_j(t) + \bbE^{\frac12} a_j t  - \bbE^{\frac12} \bar l_j (b_o(y)) \right) e_{m+1} + \bbE^{\frac12} b_v(y) + O (\bbE)
\end{equation}
where we have used that $X_o(t,y,z)= e \, \beta_{\pi_0}(\tilde\bS_k)^{-1} \, (b_o\cdot y)- \beta_{\pi_0}(\tilde\bS_k)^{-1}\,b_o \,t - b_v (e_{m+1}\cdot z) $ and $X_v(t,y,z)= e_{m+1} (b_v\cdot y)$, with $(t,y,z) \in V^\perp\cap \pi_0\times V\times \pi_0^\perp $, the definition $\bar l_j = \beta_{\pi_0}(\tilde\bS_k)^{-1} \, \tilde l_j$, and moreover that $\tilde l_j(b_o) = \tilde l_j (b_v) = \tilde l_j (0)=0$ given that the vectors $b_0$ and $b_v$ are directed along $V$.

We are now ready to estimate the excess along the blow-up sequence. We observe that $\dist(q,\hat{\bS}) \le \dist(q, \tilde{\bS}) + \bbE^\frac12 |q|$ hence we deduce from \eqref{e:est spine} for any $\rho_\infty < r$ 
\[ \int_{\bB_{\frac18} \cap \bB_r(V)} \dist(q, \hat{\bS})^2 \, d\norm{T}(q) \le C r^\frac12 \bbE\,\,\]
so that
\begin{equation}\label{eq.key-contradiction-leftover1}
    \limsup_{k \to \infty} \bbE^{-1}_k \int_{\bB_{\rho} \cap \bB_r(V)} \dist(q, \hat{\bS})^2 \, d\norm{T_k}(q) \le C r^\frac12\,.
\end{equation}
For $\bbE$ sufficiently small $T$ agrees with the graph of the multi-function $u^\pm_j, j=1, \dotsc , m$ over $\pi_0$. Furthermore the $\bbE^{-\frac12} (u^\pm_j- \tilde l_j^\pm)$ converge to the harmonic functions $\bar w_j^\pm$. 
Hence we conclude, using \eqref{eq.trasnformedlinearfunction} in the second inequality below, that  
\begin{align*}
    &\int_{\bB_{\rho}\cap \bB_r(V)^c} \dist(q, \hat{\bS})^2 d\norm{T} \le \sum_{\pm, j} \int_{|x|>\frac{r}{2}} |u^\pm_j - \hat{l}^\pm_j|^2\\ &\le C \sum_{\pm, j} \int_{|x|>\frac{r}{2}} |(u^\pm_j-\tilde l^\pm_j) -\bbE^\frac12 \left(a_j^\pm t+ b_v(y) - \bar l_j^\pm (b_o (y))\right)|^2 + O(\bbE^2)\,.
\end{align*}
Thus by Proposition \ref{p:linear-decay} 
\begin{align}\label{eq.key-contradiction-leftover2}
    \limsup_{k \to \infty} &\;\bbE^{-1}_k \int_{\bB_{\rho} \cap \bB_r(V)^c} \dist(q, \hat{\bS})^2 \, d\norm{T_k}(q)\notag\\
    &\le C \int_{\pi_0^\pm\cap \bB_{\rho}} \big|\bar{w}^\pm_j (t,y) - a^\pm_j (t) - (b_v (y) - \bar l_j^{\pm} (b_o (y)))\big|^2\, dy\, dt \leq C \rho^{m+4}\,.
\end{align}
For the double sided excess we need to bound the distance from $\hat{\bS}$ to $T$ outside of $\bB_{\frac{\rho}{8}}(V)$. Since $T$ and $\hat{\bS}$ are graph over $\pi_0$ in this region, we can estimate it as above by the distance between the graphs. Hence combining \eqref{eq.key-contradiction-leftover1} with \eqref{eq.key-contradiction-leftover2} we conclude
\[\limsup_{k \to \infty}\; \bbE^{-1}_k(\bbE(T_k,\hat\bS_k,v 0,\rho)) \le C\frac{r^\frac12}{\rho^{m+2}} + C \rho^2 \,.\]
Since $r>0$ is arbitrary, \eqref{e:key-contradiction-2} follows.
\end{proof}

\appendix

\section{Proof of Lemma \texorpdfstring{\ref{l:monot-with-A^2}}{l:monot}}

The proof follows closely the one given in \cite{DHMSS}
for \cite[Lemma 8.2]{DHMSS}

For any $0<r<R$, we consider the vector field
\[
W_{a,r} (q):= \left(\frac{1}{\max(r,\abs{q})^{m+a}}-\frac{1}{R^{m+a}}\right)^+ q\,.
\]
We then insert $g^2 W_{a,r}$ in the first variation formula, cf. \cite[Lemma 5.1]{DHMSS} to derive
\begin{align*}
- &\int_{\bB_R} g^2 W_{a,r} \cdot \vec{H}_T\, d\norm{T} =\frac{m}{r^{m+a}} \int_{\bB_r} g^2 \, d\norm{T} - \frac{m}{R^{m+a}} \int_{\bB_R} g^2 \, d\norm{T}\\&-a \int_{\bB_R\setminus \bB_r} \frac{g^2 (q)}{\abs{q}^{m+a}}\, d\norm{T} (q) + (m+a) \int_{\bB_R\setminus \bB_r} g^2 (q)\frac{\abs{q^\perp}^2}{\abs{q}^{m+a+2}}\, d\norm{T} (q)\\
&+ \int_{\bB_R} W_{a,r}^T \cdot \nabla g^2 \, \norm{T} \, ,
\end{align*}
where $W_{a,r}^T (q)$ denotes the projection on the tangent plane to $T$ at $q$ of the vector $W_{a,r} (q)$.
Here, the generalized mean curvature $\vec{H}_T (q)$ is given by 
\[
\sum_{i=1}^m A_\Sigma (e_i, e_i)\, ,
\]
where $e_1, \ldots, e_{n-1}$ is an orthonormal base of the approximate tangent space to $T$ at $q$.

Observe that $W_{a,r}^T (q)$ is in fact parallel to $q^T$. Now we can use the homogeneity of $g$ and the identity $q= q^T + q^\perp$ to deduce that 
\[
\nabla g^2 (q) \cdot q^T= 2k g^2 (q) - 2g (q) \nabla g (q) \cdot q^\perp \ge \left(2k-\frac\epsilon2\right) g^2 (q) - \frac{2}{\epsilon}\abs{\nabla g (q)}^2 \abs{q^\perp}^2\,.
\]
In particular we may choose $a=2k-\alpha$, $\epsilon = \alpha$ to estimate 
\begin{align*}
- &\int_{\bB_R} g^2 W_{a,r} \cdot \vec{H}_T\, d\norm{T} \ge \frac{m+2k-\sfrac\alpha2}{r^{m+2k-\alpha}} \int_{\bB_r} g^2 \, d\norm{T} - \frac{m+2k-\sfrac\alpha2}{R^{m+2k-\alpha}} \int_{\bB_R} g^2 \, d\norm{T}\\&\frac{\alpha}{2} \int_{\bB_R\setminus \bB_r} \frac{g^2 (q)}{\abs{q}^{m+2k-\alpha}}\, d\norm{T} (q) + (m+2k-\alpha) \int_{\bB_R\setminus \bB_r} g^2\frac{\abs{q^\perp}^2}{\abs{q}^{m+2k+2-\alpha}}\, d\norm{T} (q)\\&- \frac{2}{\alpha}\int_{\bB_R}\left(\frac{1}{\max(r,\abs{q})^{m+2k-\alpha}}-\frac{1}{R^{m+2k-\alpha}}\right)^+ \abs{\nabla g (q)}^2 \abs{q^\perp}^2\,d \norm{T} (q) \,\,.
\end{align*}
We now claim that 
\begin{equation}\label{e:A^2-comes}
|g^2 W_{a,r} \cdot \vec{H}_T| (q)\leq C \|\hat{g}\|_\infty^2\bA^2 R^\alpha |q|^{1-m}\, .
\end{equation}
First of all observe $|\vec{H}_T|\leq m \bA$. Then note that $\vec{H}_T$ is orthogonal to $T_q \Sigma$. We can thus write, for $|q|\leq R$,
\[
|g^2 W_{a,r} \cdot \vec{H}_T| (q)\leq C \|\hat{g}\|_\infty^2\bA R^\alpha |q|^{-m} |\mathbf{p}_{T_q \Sigma^\perp} (q)|\, .
\]
However, given that both $0$ and $q$ belong to $\Sigma$, we see that 
\[
|\mathbf{p}_{T_q \Sigma^\perp} (q)|\leq C\bA |q|
\]
Having proven \eqref{e:A^2-comes}, we exploit the monotonicity formula to estimate
\[
\int_{\bB_R} |q|^{1-m} d\|T\| (q) \leq C \frac{\|T\| (\bB_R)}{R^m}\, .
\]
We thus conclude
\begin{align*}
\frac{\alpha}{2} \int_{\bB_R\setminus \bB_r} \frac{g^2 (q)}{\abs{q}^{m+2k-\alpha}}\, d\norm{T} (q)\leq
&\frac{m+2k}{R^{m+2k-\alpha}} \int_{\bB_{R}} g^2 \,d\norm{T} + C \bA^2 \|\hat g\|_\infty^2 \frac{\norm{T}(B_{R})}{R^{m-\alpha}}\\
&+ \frac{2}{\alpha}\int_{\bB_R}\frac{\abs{\nabla g (q)}^2 \abs{q^\perp}^2}{\max(r,\abs{q})^{m+2k-\alpha}} \,d \norm{T} (q)\, .
\end{align*}
Letting $r\downarrow 0$ we then conclude \eqref{e.h_k monotonicity}.

\bibliographystyle{plain}
\bibliography{references}

\begin{thebibliography}{10}

\bibitem{Allard72}
William~K. Allard.
\newblock {On the first variation of a varifold}.
\newblock {\em Ann. of Math. (2)}, 95:417--491, 1972.

\bibitem{CoEdSp}
Maria Colombo, Nick Edelen, and Luca Spolaor.
\newblock The singular set of minimal surfaces near polyhedral cones.
\newblock {\em J. Differential Geom.}, 120(3):411--503, 2022.

\bibitem{D-Allard}
Camillo De~Lellis.
\newblock Allard's interior regularity theorem: an invitation to stationary
  varifolds.
\newblock In {\em Nonlinear analysis in geometry and applied mathematics.
  {P}art 2}, volume~2 of {\em Harv. Univ. Cent. Math. Sci. Appl. Ser. Math.},
  pages 23--49. Int. Press, Somerville, MA, 2018.

\bibitem{DHMSS}
Camillo {De Lellis}, Jonas Hirsch, Andrea Marchese, Luca Spolaor, and Salvatore
  Stuvard.
\newblock Area minimizing hypersurfaces modulo $p$: a geometric free-boundary
  problem.
\newblock {\em Preprint arXiv:2105.08135}, 2021.

\bibitem{DHMSS_final}
Camillo {De Lellis}, Jonas Hirsch, Andrea Marchese, Luca Spolaor, and Salvatore
  Stuvard.
\newblock Fine structure of the singular set of area minimizing hypersurfaces
  modulo p.
\newblock {\em Preprint arXiv:2201.10204}, 2022.

\bibitem{DLHMS}
Camillo De~Lellis, Jonas Hirsch, Andrea Marchese, and Salvatore Stuvard.
\newblock Regularity of area minimizing currents {${\rm mod}\,p$}.
\newblock {\em Geom. Funct. Anal.}, 30(5):1224--1336, 2020.

\bibitem{DLHMS_linear}
Camillo De~Lellis, Jonas Hirsch, Andrea Marchese, and Salvatore Stuvard.
\newblock Area-minimizing currents {${\rm mod}\,2Q$}: linear regularity theory.
\newblock {\em Comm. Pure Appl. Math.}, 75(1):83--127, 2022.

\bibitem{DLMiSk}
Camillo De~Lellis, Paul Minter, and Anna Skorobogatova.
\newblock The fine structure of the singular set of area-minimizing integral
  currents {III}, flat singular points and $\mathcal{H}^{m-2}$-a.e. uniqueness
  of tangent cones.
\newblock {\em Preprint arXiv:2304.11552}, 2023.

\bibitem{DLS_Lp}
Camillo {De Lellis} and Emanuele Spadaro.
\newblock {Regularity of area minimizing currents {I}: gradient {$L^p$}
  estimates}.
\newblock {\em Geom. Funct. Anal.}, 24(6):1831--1884, 2014.

\bibitem{DLS_Center}
Camillo {De Lellis} and Emanuele Spadaro.
\newblock {Regularity of area minimizing currents {II}: center manifold}.
\newblock {\em Ann. of Math. (2)}, 183(2):499--575, 2016.

\bibitem{MW}
Paul Minter and Neshan Wickramasekera.
\newblock A structure theory for stable codimension 1 integral varifolds with
  applications to area minimising hypersurfaces mod p.
\newblock {\em Preprint arXiv:2111.11202}, 2021.

\bibitem{Simon}
Leon Simon.
\newblock Cylindrical tangent cones and the singular set of minimal
  submanifolds.
\newblock {\em J. Differential Geom.}, 38(3):585--652, 1993.

\bibitem{Spolaor}
Luca Spolaor.
\newblock Almgren's type regularity for semicalibrated currents.
\newblock {\em Adv. Math.}, 350:747--815, 2019.

\bibitem{Taylor}
Jean~E. Taylor.
\newblock {Regularity of the singular sets of two-dimensional area-minimizing
  flat chains modulo {$3$} in {$R^{3}$}}.
\newblock {\em Invent. Math.}, 22:119--159, 1973.

\bibitem{White79}
Brian White.
\newblock {The structure of minimizing hypersurfaces mod {$4$}}.
\newblock {\em Invent. Math.}, 53(1):45--58, 1979.

\bibitem{White86}
Brian White.
\newblock {A regularity theorem for minimizing hypersurfaces modulo {$p$}}.
\newblock In {\em {Geometric measure theory and the calculus of variations
  ({A}rcata, {C}alif., 1984)}}, volume~44 of {\em {Proc. Sympos. Pure Math.}},
  pages 413--427. Amer. Math. Soc., Providence, RI, 1986.

\bibitem{Wic}
Neshan Wickramasekera.
\newblock A general regularity theory for stable codimension 1 integral
  varifolds.
\newblock {\em Ann. of Math. (2)}, 179(3):843--1007, 2014.

\end{thebibliography}

\end{document}